\documentclass{amsart}

\newtheorem{theorem}{Theorem}[section]
\newtheorem{lemma}[theorem]{Lemma}
\newtheorem{proposition}[theorem]{Proposition}
\newtheorem{corollary}[theorem]{Corollary}
\theoremstyle{definition}
\newtheorem{definition}[theorem]{Definition}
\newtheorem{example}[theorem]{Example}

\theoremstyle{remark}
\newtheorem{remark}[theorem]{Remark}

\numberwithin{equation}{section}

\usepackage{graphicx}
\usepackage[all]{xy}
\usepackage{tikz}
\usepackage{tikz-cd}
\usetikzlibrary{arrows,automata}
\usetikzlibrary{shapes,snakes,backgrounds}
\usepackage[dvips]{epsfig}
\usepackage{amsfonts}
\usepackage{amssymb}
\usepackage{amscd}
\usepackage{amstext}
\usepackage{amsmath}
\usepackage{pstricks}
\usepackage{enumerate}
\usepackage{hyperref}

\usepackage{ulem}

\gdef\minishuffle{{\scriptstyle \shuffle}}

\def\abs#1{|#1|}
\def\absv#1{\Vert#1\Vert}

\def\shuffle{\mathop{_{^{\sqcup\!\sqcup}}}} 

\def\mod{{\rm\ mod\ }}

\def\adots{\mathinner{\mkern2mu\raise1pt\hbox{.}
\mkern3mu\raise4pt\hbox{.}\mkern1mu\raise7pt\hbox{.}}}

\def\pointir{\unskip . -- \ignorespaces}

\def\up#1{\raise 1ex\hbox{\footnotesize#1}}

\def\H{\mathcal{H}}

\def\span{\mathop\mathrm{span}\nolimits}

\def\path{\rightsquigarrow}

\def\Lyn{{\mathcal Lyn}}

\def\N{{\mathbb N}}
\def\C{{\mathbb C}}
\def\R{{\mathbb R}}
\def\Z{{\mathbb Z}}
\def\Q{{\mathbb Q}}

\def\Q{{\mathbb Q}}

\def\L{\mathrm{L}}
\def\H{\mathrm{H}}

\renewcommand{\Z}{{\mathbb Z}}
\renewcommand{\Q}{{\mathbb Q}}
\renewcommand{\R}{{\mathbb R}}
\renewcommand{\C}{{\mathbb C}}

\newcommand{\calD}{{\mathcal D}}
\newcommand{\calH}{{\mathcal H}}
\newcommand{\calZ}{{\mathcal Z}}
\newcommand{\calL}{{\mathcal L}}
\newcommand{\calM}{{\mathcal M}}
\newcommand{\calR}{{\mathcal R}}
\newcommand{\calX}{{\mathcal X}}

\newcommand{\calS}{{\mathcal S}}
\newcommand{\Li}{\operatorname{Li}}
\newcommand{\ad}{\operatorname{ad}}

\newcommand{\Frac}[2]{\displaystyle \frac{#1}{#2}}

\newcommand{\serie}[2]{#1 \langle \! \langle #2 \rangle \! \rangle}

\newcommand{\poly}[2]{#1 \langle #2 \rangle}

\def\QX{\mathbb{Q}\langle X \rangle}

\def\CX{\poly{\C}{X}}

\def\CY{\poly{\C}{Y}}

\def\pol#1{\langle #1 \rangle}
\def\Lyn{\mathcal Lyn}

\def\calC{\mathcal{C}}
\def\calZ{\mathcal{Z}}
\def\calG{\mathcal{G}}

\def\CX{\C \langle X \rangle}
\def\CY{\C \langle Y \rangle}

\def\L{\mathrm{L}}
\def\H{\mathrm{H}}
\def\P{\mathrm{P}}
\def\abs#1{|#1|}
\def\scal#1#2{\langle #1\mid#2 \rangle}
\def\ncp#1#2{#1\langle #2\rangle}
\def\ncs#1#2{#1\langle \!\langle #2\rangle \!\rangle}
\def\pol#1{\langle #1 \rangle}

\def\AX{A \langle X \rangle}

\def\AY{A \langle Y \rangle}

\def\CXX{\serie{\C}{X}}
\def\Lie{{\mathcal L}ie}

\def\LXX{\Lie_{\C} \langle\!\langle X \rangle \!\rangle}

\def\shuffle{\mathop{_{^{\sqcup\!\sqcup}}}} 
\def\conc{\mathop{\tt conc}} 

\def\calC{{\mathcal C}}

\def\calG{{\mathcal G}}
\def\calH{{\mathcal H}}
\def\calL{{\mathcal L}}

\def\calT{{\mathcal T}}
\def\calV{{\mathcal V}}
\def\calX{{\mathcal X}}
\def\calZ{{\mathcal Z}}
\def\scal#1#2{\langle #1 | #2 \rangle}

\def\ncs#1#2{#1\langle\langle #2\rangle\rangle}

\def\ncp#1#2{#1\langle #2\rangle}

\def\LXX{\Lie_{\C} \langle\!\langle X \rangle \!\rangle}
\def\Li{\mathrm{Li}}

\def\calA{\mathcal{A}}

\gdef\stuffle{\;%
  \setlength{\unitlength}{0.0125cm}%
  \begin{picture}(20,10)(220,580)
  \thinlines
  \put(220,592){\line( 0,-1){ 10}}
  \put(220,582){\line( 1, 0){ 20}}
  \put(240,582){\line( 0, 1){ 10}}
  \put(230,592){\line( 0,-1){ 10}}
  \put(225,587){\line( 1, 0){ 10}}
  \end{picture}\;
}
\def\pol#1{\langle #1 \rangle}
\def\rd{\triangleright}
\def\rg{\triangleleft}
\def\trr{\triangleright}
\def\trl{\triangleleft}
\def\path{\rightsquigarrow}

\newcommand\rsmraise[1]{%
  \ifx#1\displaystyle .8\else
    \ifx#1\textstyle .8\else
      \ifx#1\scriptstyle .6\else
        .45%
      \fi
    \fi
  \fi}
\begingroup
\def\2#1{\ifnum#1<10 0\fi\the#1}
\xdef\isodayandtime{
{\2\day-\2\month-\the\year\space\2{\count0}:%
\2{\count2}}}
\endgroup

\def\poly#1#2{#1\langle #2 \rangle}

\def\CX{\poly{\C}{X}}

\def\CY{\poly{\C}{Y}}

\def\AX{A \langle X \rangle}

\def\AY{A \langle Y \rangle}

\newcommand*{\longtwoheadrightarrow}{\ensuremath{\joinrel\relbar\joinrel\twoheadrightarrow}}

\begingroup
\def\2#1{\ifnum#1<10 0\fi\the#1}
\xdef\isodayandtime{
{\2\day-\2\month-\the\year\space\2{\count0}:%
\2{\count2}}}
\endgroup

\newcounter{per1}
\definecolor{MyDarkBlue}{rgb}{0,0.08,0.4}

\begin{document}

\title[Renormalization and Regularization of Zeta Functions]{On The Global Renormalization and Regularization of Several Complex Variable Zeta Functions by Computer}

\author{V.C. Bui}
\address{University of Sciences, Hue University, 77, Nguyen Hue, Hue, Viet Nam,}
\email{bvchien.vn@gmail.com}
\author{V. Hoang Ngoc Minh}
\address{University of Lille, 1 Place D\'eliot, 59024 Lille, France,\\
LIPN - UMR 7030, CNRS, 93430 Villetaneuse, France,}
\email{vincel.hoang-ngoc-minh@univ-lille.fr,minh@lipn.univ-paris13.fr}
\author{Q. H. Ngo}
\address{Hanoi University of Science and Technology, 1 Dai Co Viet, Hai Ba Trung, Ha Noi, Viet Nam,}
\email{hoan.ngoquoc@hust.edu.vn}
\author{V. Nguyen Dinh}
\address{LIPN-UMR 7030, 99 avenue Jean-Baptiste Cl\'ement, 93430 Villetaneuse, France,}
\email{nguyendinh@lipn.univ-paris13.fr}

\subjclass[2020]{05E16, 11M32, 16T05, 20F10, 33F10, 44A20}

\date{\today}

\keywords{Eulerian functions \and zeta function \and Gamma function.}

\begin{abstract}
This review concerns the resolution of a special case of Knizhnik-Zamolodchikov equations ($KZ_3$)
using our recent results on combinatorial aspects of zeta functions on several variables and
software on noncommutative symbolic computations.

In particular, we describe the actual solution of $(KZ_3)$ leading to the unique noncommutative series,
$\Phi_{KZ}$, so-called Drinfel’d associator (or Drinfel’d series). Non-trivial expressions for series
with rational coefficients, satisfying the same properties with $\Phi_{KZ}$, are also explicitly provided
due to the algebraic structure and the singularity analysis of the polylogarithms and harmonic sums.
\end{abstract}

\maketitle

\tableofcontents


\section{Introduction}\label{introduction}
In order to renormalize (Theorems \ref{renormalization1} and \ref{renormalization2} bellow)
and to regularize (Corollary \ref{pont}, Theorem \ref{fin} bellow), over\footnote{For any
$z\in\C,\mathbf{Re}(z)$ denotes the real part of $z$.}
\begin{eqnarray}
&\calH_r=\{(s_1,\ldots,s_r)\in\C^r\vert\forall m=1,..,r,\mathbf{Re}(s_1)+\ldots+\mathbf{Re}(s_m)>m\},&r\ge1,
\end{eqnarray}
the following several complex variables functions
\begin{eqnarray}
	\zeta_r:\calH_r\longrightarrow{\C},&(s_1,\ldots,s_r)\longmapsto\displaystyle\sum_{n_1>\ldots>n_r>0}{n_1^{-s_1}\ldots n_r^{-s_r}},
\end{eqnarray}
we will base mainly on the combinatorics of the commutative and noncommutative generating series
of polylogarithms $\{\Li_{s_1,\ldots,s_r}\}^{r\ge1}_{(s_1,\ldots,s_r)\in\C^r}$,
of harmonic sums $\{\H_{s_1,\ldots,s_r}\}^{r\ge1}_{(s_1,\ldots,s_r)\in\C^r}$
defined, for $z\in\C,\abs{z}<1$ and $n\in\N_{+}$, as follows
\begin{eqnarray}
\Li_{s_1,\ldots,s_r}(z)&:=&\sum_{n_1>\ldots>n_r>0}\frac{z^{n_1}}{n_1^{s_1}\ldots n_r^{s_r}},\label{def1}\\
\H_{s_1,\ldots,s_r}(n)&:=&\sum_{n_1>\ldots>n_r>0}^n\frac1{n_1^{s_1}\ldots n_r^{s_r}},\label{def2}
\end{eqnarray}
For $(s_1,\ldots,s_r)\in\calH_r$, after a theorem by Abel, one has
\begin{eqnarray}\label{zetavalues}
	\lim_{{z\rightarrow 1^-}}\Li_{s_1,\ldots,s_r}(z)
	=\lim_{{n\rightarrow +\infty}}\H_{s_1,\ldots,s_r}(n)
	=\zeta_r(s_1,\ldots,s_r)
\end{eqnarray}
and, for simplification, we adopt the notation $\zeta$ for $\zeta_r,r\in\N$. In this case, $\zeta$
might be also viewed as function from the monoid generated by $\C$ to $\R$, for more correct.

These functions, in \eqref{def1}--\eqref{def2} satisfy
\begin{eqnarray}\label{gs}
\frac{\Li_{s_1,\ldots,s_r}(z)}{1-z}=\sum_{n\ge1}\H_{s_1,\ldots,s_r}(n)z^n
\end{eqnarray}
and involve in the {\it normalization} of solutions of the following first order differential equation,
with singularities in $\{0,1,+\infty\}$ and noncommutative indeterminates in $X=\{x_0,x_1\}$:
\begin{equation}\label{DE}
\frac{dG(z)}{dz}=\Big(x_0\frac{dz}z+x_1\frac{dz}{1-z}\Big)G(z).
\end{equation}

Let $\calH(\Omega)$ denotes the ring of holomorphic functions on the simply connected
domain $\Omega:=\widetilde{\C-\{0,1\}}$, with $1_{\Omega}:\Omega\rightarrow\calH(\Omega)$
(mapping $z$ to $1$) as the neutral element and let us introduce the following differential forms
\begin{eqnarray}
\omega_0(z):=\frac{dz}z&\mbox{and}&\omega_1(z):=\frac{dz}{1-z}.
\end{eqnarray}
The resolution of \eqref{DE}, on the completion of ${\calH(\Omega)}\pol{X}$,
uses the so-called {\it Chen series}, of $\omega_0,\omega_1$ and along a path
$z_0\path z$ on $\Omega$, defined by \cite{cartier1,fliess1}:
\begin{eqnarray}\label{chen}
C_{z_0\path z}:=\sum_{w\in X^*}\alpha_{z_0}^z(w)w&\in&\widehat{{\calH(\Omega)}\pol{X}}=\serie{\calH(\Omega)}{X},
\end{eqnarray}
where $X^*$ is the free monoid generated by $X$ (the neutral element is the empty word denoted by $1_{X^*}$)
and, for a subdivision $(z_0,z_1\ldots,z_k,z)$ of $z_0\path z$ and the coefficient
$\alpha_{z_0}^z(w)\in\calH(\Omega)$ is defined by $\alpha_{z_0}^z(1_{X^*})=1_{\Omega}$
and, for $w=x_{i_1}\cdots x_{i_k}\in X^*X$, 
\begin{eqnarray}\label{iteratedintegral}
	\alpha_{z_0}^z(w)=\int_{z_0}^z\omega_{i_1}(z_1)
	\ldots\int_{z_0}^{z_{k-1}}\omega_{i_k}(z_k).
\end{eqnarray}
The series $C_{z_0\path z}$ is group-like \cite{ree}, meaning that there exists
a primitive series $L_{z_0\path z}$ such that $e^{L_{z_0\path z}}=C_{z_0\path z}$.
The challenge is then to determine $L_{z_0\path z}$, via the Magnus' Lie-integral-functional expansion \cite{magnus2}
and to regularize $C_{0\path1}$ and $L_{0\path1}$ (although a lot of iterated integrals be divergent).
In \cite{drinfeld2}, essentially interested in solutions of \eqref{DE} over the interval $]0,1[$
and, using the involution $z\mapsto1-z$, Drinfel'd stated that \eqref{DE} admits a unique solution
$G_0$ (resp. $G_1$) satisfying\footnote{\textit{i.e.} $\lim_{z\to0}G_0(z)e^{x_0\log(z)}=1$
(resp. $\lim_{z\to1}e^{x_1\log(1-z)}G_1=1$).}
\begin{eqnarray}\label{asymcond}
&G_0(z)\sim_0z^{x_0}=e^{x_0\log(z)}&(\mbox{resp. }G_1(z)\sim_1(1-z)^{-x_1}=e^{-x_1\log(1-z)}).
\end{eqnarray}
Since $G_0,G_1$ are group-like then there is a unique group-like series $\Phi_{KZ}\in\serie{\R}{X}$,
so-called Drinfel’d associator \cite{racinet} (or Drinfel’d series \cite{gonzalez}), such that\footnote{The
series $\Phi_{KZ}$ satisfies a lot of relations: duality, pentagonal, hexagonal \cite{drinfeld2}. Although
these relations do not play an explicit role below, it can be kept in mind with a view towards applications.}
$G_0=G_1\Phi_{KZ}$.
He also proved that there exists an expression, with rational coefficients,
for Drinfel'd series but neither did construct such expression nor did explicit $G_0,G_1$.

In \cite{drinfeld2}, substituting $x_0\leftarrow A/2{\rm i}\pi$ and $x_1\leftarrow-B/2{\rm i}\pi$,
the coefficients $\{c_{k,l}\}_{k,l\ge0}$ of $\log\Phi_{KZ}$ are identified as follows
\begin{itemize}
\item Setting $\bar A=A/2{\rm i}\pi$ and $\bar B=B/2{\rm i}\pi$, over $]0,1[$,
the standard solutions $G_0=z^{\bar A}(1-z)^{\bar B}V_0(z)$ and $G_1=z^{\bar A}(1-z)^{\bar B}V_1(z)$,
where $V_0,V_1$ have continuous extensions to $]0,1[$ and is group-like solution of the following
noncommutative differential equation, in the topological free Lie algebra,
$\mathfrak p\mathrm{span}\{\ad_A^k\ad_B^l[A,B]\}_{k,l\ge0}$, with $V_0(0)=1,V_1(1)=1$.
\begin{eqnarray}\label{approx_Drinfeld}
{\bf d}S(z)=Q(z)S(z),&Q(z):=e^{\ad_{-\log(1-z)\bar B}}e^{\ad_{-\log(z)\bar A}}\Frac{\bar B}{z-1}\in\mathfrak p,
\end{eqnarray}

\item Since $G_0=G_1\Phi_{KZ}$ then $\Phi_{KZ}=V(0)V(1)^{-1}$, where $V$ is a solution of \eqref{approx_Drinfeld}
and then, by identification, the coefficients $\{c_{k,l}\}_{k,l\ge0}$ of $\log\Phi_{KZ}$ are obtained, in the
abelianization $\mathfrak p/[\mathfrak p,\mathfrak p]$ and by serial expansions, as follows
\begin{eqnarray}\label{logPhiKZ}
\log\Phi_{KZ}
&=&\sum_{k,l\ge0}c_{k,l}B^{k+1}A^{l+1}\cr
&=&\int_0^1Q(z)dz\mod[\mathfrak p,\mathfrak p]\cr
&=&\int_0^1e^{\ad_{-\log(1-z)\bar B}}e^{\ad_{-\log(z)\bar A}}\frac{\bar Bdz}{z-1}\mod[\mathfrak p,\mathfrak p]\cr
&=&\sum_{k,l\ge0}\frac1{l!k!}\int_0^1\log^l\frac1{1-z}\log^k\Big(\frac1z\Big)
\ad_{\bar B^k\bar A^l}\bar B\frac{dz}{z-1}\mod[\mathfrak p,\mathfrak p].
\end{eqnarray}

\item Since $\bar B^k\bar A^l\bar B=B^kA^lB/(2{\rm i}\pi)^{k+l+1}$ then the following divergent
(iterated) integral is regularized\footnote{The readers are invited to consult \cite{CM} for a
comparison of these regularized values yielding expressions of $\Phi_{KZ}$ and $\log\Phi_{KZ}$,
in which involve polyzetas.} by
\begin{eqnarray}
c_{k,l}=\frac{1}{(2{\rm i}\pi)^{k+l+2}(k+1)!l!}\int_0^1\log^l\Big(\frac1{1-z}\Big)\frac{dz}{z-1}
\end{eqnarray}
and, by a Legendre's formula \cite{legendre} (see also \eqref{Weierstrass} bellow), Drinfel'd stated that previous process is
equivalent to the following identification\footnote{Note that the summation on right side starts with $n=2$ and then $\gamma$
could not be appeared in the regularization proposed in \cite{drinfeld2}.} \cite{drinfeld2}:
\begin{eqnarray}
&\displaystyle 1+\sum_{k,l\ge0}c_{k,l}B^{k+1}A^{l+1}=\exp\Big(\sum_{n\ge2}\frac{\zeta(2n)}{(2{\rm i}\pi)^nn}(B^n+A^n-(B+A)^n)\Big).
\end{eqnarray}
\end{itemize}
After that, L\^e and Murakami expressed, in particular, the divergent coefficients of $\Phi_{KZ}$ as
linear combinations of $\{\zeta_r(s_1,\ldots,s_r)\}^{r\ge1}_{(s_1,\ldots,s_r)\in\N^r_{\ge1},s_1\ge2}$,
via a regularization based on representation of the chord diagram algebras \cite{lemurakami}.

In this work, providing $G_0$ and a series with rational coefficients, satisfying same properties with $\Phi_{KZ}$ (see \eqref{asymcond}),
we use implementations in \cite{Bui,JSC,Ngo,FPSAC96,FPSAC97,FPSAC98} and the isomorphy the Cauchy and Hadamard algebras of polylogarithmic
functions (as defined in \eqref{def1} and \eqref{gs}) to, respectively, the shuffle and quasi-shuffle algebras.
We will use also the isomorphy of the shuffle and quasi-shuffle bialgebras, to the respective enveloping algebras
of their primitive elements, leading to the constructions of dual bases to factorize characters thanks to the
Cartier-Quillen-Milnor-Moore and Poincar\'e-Birkhoff-Witt theorems (CQMM and PBW for short).

In Appendix, we will explain how (\ref{DE}) arises in Knizhnik--Zamolodchikov differential equations $(KZ_n)$,
as proposed for $n=3$ in \cite{cartier1,drinfeld2}. Since the work presented in this text is quite extensive
then it will be continued in \cite{QTS} by treating the case of $n\ge4$, as application of a Picard-Vessiot
theory of noncommutative differential equations (see \cite{PVNC}).

\vfill

The organization of this paper is as follows:
\begin{itemize}
\item In Section \ref{Combinatorial}, algebraic combinatorial frameworks will be introduced.
In particular, we will give an explicit isomorphism, $\varphi_{\pi_1}$, from shuffle bialgebra to quasi-shuffle bialgebra
(Theorem \ref{isomorphy}). Doing with $\varphi_{\pi_1}$, the construction by M\'elan\c{c}on-Reutenauer-Sch\"uzenberger
(MRS, for short), initially elaborated in shuffle bialgebra and useful to factorize group-like series and then
noncommutative rational series (Theorem \ref{residual} and \ref{exchangeable}), will be extended in quasi-shuffle bialgebra
for similar factorizations, via constructions of pairs of bases in duality.

\item In Section \ref{polylogarithms}, to study the structures via noncommutative generating series the polylogarithms and harmonic sums, at
integral multi-indices, will be encoded by words over various alphabets (Theorems \ref{structure1}, \ref{structure2} and \ref{Indexation2}).

In particular, their noncommutative generating series will be put in the MRS form (their logarithms will also be provided in
Proposition \ref{LH}) and the bi-integro differential algebra of polylogarithms will be examined (Proposition \ref{integro}).

Concerning the polylogarithms at positive indexes, we will insist on the fact that their noncommutative generating series
is the actual solution of \eqref{DE} and the noncommutative generating series of the finite parts of their singular
expansions corresponds to $\Phi_{KZ}$ which will be also put in MRS form without divergent zeta values as local coordinates.

\item In Section \ref{Global}, with noncommutative generating series, the global renormalization and regularization of
polylogarithms and of harmonic sums will be achieved, providing Drindfel'd series (Theorems \ref{renormalization1},
\ref{renormalization2} and \ref{fin}).

With commutative ones, many functions (algebraic functions with singularities in $\{0,1,+\infty\}$)
forgotten in the straight algebra of polylogarithms, thanks to Theorem \ref{exchangeable}, will be recovered.

On the other hand, since polylogarithms at negative indices are polynomial, of coefficients in $\Z$ on $(1-z)^{-1}$
(Propositions \ref{polynomes} and \ref{explicit}) then the generating series of the finite parts of their singular expansions,
will precise the {\it regularization characters} (Corollary \ref{pont} and Proposition \ref{reg_alg}) and will give examples
of Drindfel'd series with rational coefficients (Theorem \ref{fin}).
\end{itemize}

\section{Combinatorial framework}\label{Combinatorial}
In this section, coefficients belong to a commutative ring\footnote{although some of
the properties already hold for a general commutative semiring \cite{berstel}.} $A$ and,
unless explicitly stated, all tensor products will be considered over the ambient ring (or field). 

\subsection{Factorization in bialgebras}
In section \ref{introduction}, the encoding alphabet $X$ was already introduced.
Let us note that there are one-to-one correspondences
\begin{eqnarray}
(s_1,\ldots,s_r)\in\N_+^r\leftrightarrow x_0^{{s_1}-1}x_1\ldots x_0^{{s_r}-1}x_1\in X^*x_1
\displaystyle\mathop{\rightleftharpoons}_{\pi_X}^{\pi_Y}y_{s_1}\ldots y_{s_r}\in Y^*,
\end{eqnarray}
where $Y:=\{y_k\}_{k\ge1}$ and $\pi_X$ is the $\conc$ morphism, from $\ncp{A}{Y}$ to $\ncp{A}{X}$, mapping
$y_k$ to $x_0^{k-1}x_1$. This morphism $\pi_X$ admits an adjoint $\pi_Y$ for the two standard scalar products\footnote{
That is to say $(\forall p\in\ncp{A}{X})\;(\forall q\in\ncp{A}{Y})\;(\scal{\pi_Yp}{q}_Y=\scal{p}{\pi_Xq}_X)$.}
which has a simple combinatorial description: the restriction of $\pi_Y$ to the subalgebra $(A1_{X^*}\oplus\ncp{A}{Y}x_1,\conc)$,
is an isomorphism given by $\pi_Y(x_0^{k-1}x_1)=y_k$ (and the kernel of the non-restricted $\pi_Y$ is $\ncp{A}{X}x_0$).
For all matters concerning finite ($X$ and similar) or infinite ($Y$ and similar) alphabets,
we will use a generic model noted $\calX$ in order to state their common combinatorial features.
Let us recall also that the coproduct $\Delta_{\conc}$ is defined, for any $w\in\calX^*$, as follows
\begin{eqnarray}\label{Dconc}
\Delta_{\conc}w=\sum_{u,v\in\calX^*,uv=w}u\otimes v.
\end{eqnarray}

As an algebra the $A$-module $\ncp{A}{\calX}$ is equipped with the associative unital
concatenation and the associative commutative and unital shuffle product. The latter being
defined, for any $x,y\in\calX$ and $u,v,w\in\calX^*$, by the following recursion
\begin{eqnarray}\label{shuffle}
w\shuffle 1_{\calX^*}=1_{\calX^*}\shuffle w=w&\mbox{and}&
xu\shuffle yv=x(u\shuffle yv)+y(xu\shuffle v)
\end{eqnarray}
or, equivalently, by its dual comultiplication (which is a morphism for concatenations, defined, for each letter $x\in\calX$, by
\begin{eqnarray}\label{Dshuffle}
\Delta_{\shuffle}x=1_{\calX^*}\otimes x+x\otimes1_{\calX^*}.
\end{eqnarray}

Once $\calX$ has been totally ordered\footnote{For technical reasons, the orders $x_0<x_1$
and $y_1>\ldots y_n>y_{n+1}>\ldots$ are usual.}, the set of Lyndon words over $\calX$ will
be denoted by $\Lyn\calX$. A pair of Lyndon words $(l_1,l_2)$ is called the standard factorization
of a Lyndon $l$ (and will be noted $(l_1,l_2)=st(l)$) if $l=l_1l_2$ and $l_2$ is the longest nontrivial
proper right factor of $l$ or, equivalently, its smallest such (for the lexicographic ordering, see
\cite{lothaire} for proofs). According to a Radford's theorem, the set of Lyndon words
form a pure transcendence basis of the $A$-shuffle algebras $(\ncp{A}{\calX},\shuffle,1_{\calX^*})$.
It is well known that the enveloping algebra $\mathcal{U}(\ncp{\calL ie_{A}}{\calX})$
is isomorphic to the (connected, graded and co-commutative) bialgebra\footnote{In case $A$
is a $\Q$-algebra, the isomorphism $\mathcal{U}(\ncp{\calL ie_{A}}{\calX})\simeq\calH_{\shuffle}(\calX)$
can also be seen as an easy application of the CQMM theorem.} $\calH_{\shuffle}(\calX)=(\ncp{A}{\calX},
\allowbreak\conc,1_{\calX^*},\Delta_{\shuffle},{\tt e})$ (the counit being here
${\tt e}(P)=\scal{P}{1_{\calX^*}}$) and, via the pairing
\begin{eqnarray}
\ncs{A}{\calX}\otimes_{A}\ncp{A}{\calX}\longrightarrow A,&
T\otimes P\longrightarrow\scal{T}{P}:=\displaystyle\sum_{w\in\calX^*}\scal{T}{w}\scal{P}{w},\label{pairing}
\end{eqnarray}
we can, classically, endow $\ncp{A}{\calX}$ with the graded\footnote{For ${\calX}=X$ or $=Y$ the 
corresponding monoids are equipped with length functions, for $X$ we consider the length of words
and for $Y$ the length is given by the weight $\ell(y_{i_1}\ldots y_{i_n})=i_1+\ldots+i_n$.
This naturally induces a grading of $\ncp{A}{{\calX}}$ and $\ncp{\calL ie_{A}}{{\calX}}$ in free
modules of finite dimensions. For general  ${\calX}$, we consider the fine grading \cite{reutenauer}
\textit{i.e.} the grading by all partial degrees which, as well, induces a grading of $\ncp{A}{{\calX}}$
and $\ncp{\calL ie_{A}}{{\calX}}$ in free modules of finite dimensions.} linear basis
$\{P_w\}_{w\in \calX^*}$ (expanded after any homogeneous basis $\{P_l\}_{l\in \Lyn\calX}$
of $\ncp{\calL ie_{A}}{\calX}$) and its graded dual basis $\{S_w\}_{w\in\calX^*}$ 
(containing the pure transcendence basis $\{S_l\}_{l\in\Lyn\calX}$ of the $A$-shuffle algebra).
In the case when $A$ is a $\Q$-algebra, we also have the following factorization\footnote{
Also called MSR factorization after the names of M\'elan\c con,
Sch\"utzenberger and Reutenauer.} of the diagonal series,
\textit{i.e.} \cite{reutenauer} (here all tensor products are over $A$)
\begin{eqnarray}\label{diagonalX}
{\calD}_{\calX}:=\sum_{w\in\calX^*}w\otimes w=\sum_{w\in\calX^*}S_w\otimes P_w
=\prod_{l\in\Lyn\calX}^{\searrow}e^{S_l\otimes P_l}
\end{eqnarray}
and (still in case $A$ is a $\Q$-algebra) dual bases of homogenous 
polynomials $\{P_w\}_{w\in\calX^*}$
and $\{S_w\}_{w\in\calX^*}$ can be constructed recursively as follows
\begin{eqnarray*}
\left\{\begin{array}{llll}
P_x=x,
&S_x=x
&\mbox{for }x\in\calX,\\
P_l=[P_{l_1},P_{l_2}],
&S_l=yS_{l'},
&{\displaystyle\mbox{for }l=yl'\in\Lyn\calX - \calX\atop\displaystyle st(l)=(l_1,l_2),}\\
P_w=P_{l_1}^{i_1}\ldots P_{l_k}^{i_k},
&S_w=\displaystyle\frac{S_{l_1}^{\shuffle i_1}\shuffle\ldots\shuffle S_{l_k}^{\shuffle i_k}}{i_1!\ldots i_k!},
&{\displaystyle\mbox{for }w=l_1^{i_1}\ldots l_k^{i_k},\mbox{ with }l_1,\ldots,\atop\displaystyle l_k\in\Lyn\calX,l_1>\ldots>l_k.}
\end{array}\right.
\end{eqnarray*}
The graded dual of $\calH_{\shuffle}(\calX)$ is $\calH_{\shuffle}^{\vee}(\calX)
=(\ncp{A}{\calX},{\shuffle},1_{\calX^*},\Delta_{\conc},\epsilon)$.

As an algebra, the module $\ncp{A}{Y}$ is also equipped with 
the associative commutative and quasi-shuffle product defined, for $u,v,w\in Y^*$ and $y_i,y_j\in Y$, by
\begin{eqnarray}
&w\stuffle 1_{Y^*}=1_{Y^*}\stuffle w=w,\\
&y_iu\stuffle y_jv=y_i(u\stuffle y_jv)+y_j(y_iu\stuffle v)+y_{i+j}(u\stuffle v).
\end{eqnarray}
This product also can be dualized according to ($y_k\in Y$)
\begin{eqnarray}\label{Dstuffle}
\Delta_{\stuffle}y_k:=y_k\otimes 1_{Y^*}+1_{Y^*}\otimes y_k +\sum_{i+j=k}y_i\otimes y_j
\end{eqnarray}
which is also a $\conc$-morphism (see \cite{siblings}). We then get another
(connected, graded and co-commutative) bialgebra which, in case $A$ is a $\Q$-algebra, is
isomorphic to the enveloping algebra of the Lie algebra of its primitive elements,
\begin{eqnarray}
\calH_{\stuffle}(Y)=(\ncp{A}{Y},\conc,1_{Y^*},\Delta_{\stuffle},{\tt e})\cong\mathcal{U}(\mathrm{Prim}(\calH_{\stuffle}(Y))),
\end{eqnarray}
where $\mathrm{Prim}(\calH_{\stuffle}(Y))=\mathrm{Im}(\pi_1)=\span_{A}\{\pi_1(w)\vert{w\in Y^*}\}$
and $\pi_1$ is the eulerian projector defined, for any $w\in Y^*$, by \cite{acta,VJM}
\begin{eqnarray}\label{piii_1}
\pi_1(w)=w+\sum_{k=2}^{(w)}\frac{(-1)^{k-1}}k\sum_{u_1,\ldots,u_k\in Y^+}\scal{w}{u_1\stuffle\ldots\stuffle u_k}u_1\ldots u_k,
\end{eqnarray}
and, for any $w=y_{i_i}\ldots y_{i_k}\in Y^*$, $(w)$ denotes the number $i_i+\ldots+i_k$.

\begin{remark}\label{prim}
By \eqref{Dconc} and \eqref{Dshuffle}, any letter $x\in\calX$ is primitive, for $\Delta_{\conc}$ and $\Delta_{\shuffle}$. 
By \eqref{Dstuffle}, the polynomials $\{\pi_1(y_k)\}_{k\ge2}$ and only the letter $y_1$ are primitive, for $\Delta_{\stuffle}$.
\end{remark}

Now, let $\{\Pi_w\}_{w\in Y^*}$ be the linear basis, expanded by decreasing Poincar\'e-Birkhoff-Witt
(PBW for short) after any basis $\{\Pi_l\}_{l\in \Lyn Y}$ of $\mathrm{Prim}(\calH_{\stuffle}(Y))$
homogeneous in weigh, and let $\{\Sigma_w\}_{w\in Y^*}$ be its dual basis which contains the pure
transcendence basis $\{\Sigma_l\}_{l\in\Lyn Y}$ of the $A$-quasi-shuffle algebra. One also has the
factorization of the diagonal series ${\calD}_Y$, on $\calH_{\stuffle}(Y)$, which reads\footnote{
Again all tensor products will be taken over $A$. Note that this factorization holds for any
enveloping algebra as announced in \cite{reutenauer}. Of course, the diagonal series no 
longer exists and must be replaced by the identity $Id_{\mathcal{U}}$.} \cite{acta,VJM,CM} 
\begin{eqnarray}\label{diagonalY}
{\calD}_Y:=\sum_{w\in Y^*}w\otimes w=
\sum_{w\in Y^*}\Sigma_w\otimes\Pi_w
=\prod_{l\in\Lyn Y}^{\searrow}e^{\Sigma_l\otimes\Pi_l}.
\end{eqnarray}

We are now in the position to state the following 
\begin{theorem}[\cite{VJM,CM}]\label{isomorphy}
Let $A$ be a $\Q$-algebra, then the endomorphism of algebras
$\varphi_{\pi_1}:(\ncp{A}{Y},\conc,1_{Y^*})\longrightarrow(\ncp{A}{Y},\conc,1_{Y^*})$
mapping $y_k$ to $\pi_1(y_k)$, is an automorphism of $\ncp{A}{Y}$ realizing an isomorphism
of bialgebras between $\calH_{\shuffle}(Y)$ and
$\calH_{\stuffle}(Y)\cong\mathcal{U}(\mathrm{Prim}(\calH_{\stuffle}(Y)))$.
In particular, the following diagram commutes
\begin{center}
\begin{tikzcd}[column sep=3em]
\ncp{A}{Y}\ar[hook]{r}{\Delta_{\shuffle}}\ar[swap]{d}{\varphi_{\pi_1}}& 
\ncp{A}{Y}\otimes\ncp{A}{Y}\ar[]{d}{\varphi_{\pi_1}\otimes\varphi_{\pi_1}}\\
\ncp{A}{Y}
\ar[hook]{r}{\Delta_{\stuffle}}& 
\ncp{A}{Y}\otimes\ncp{A}{Y} 
\end{tikzcd}
\end{center}
\end{theorem}

Hence, the bases $\{\Pi_w\}_{w\in Y^*}$ and $\{\Sigma_w\}_{w\in Y^*}$
of $\mathcal{U}(\mathrm{Prim}(\calH_{\stuffle}(Y)))$ are images by
$\varphi_{\pi_1}$ and by the adjoint mapping of its inverse, $\check\varphi_{\pi_1}^{-1}$
of $\{P_w\}_{w\in Y^*}$ and $\{S_w\}_{w\in Y^*}$, respectively.
Algorithmically, by Remark \ref{prim}, the dual bases of homogenous polynomials $\{\Pi_w\}_{w\in Y^*}$
and $\{\Sigma_w\}_{w\in Y^*}$ can be constructed directly and recursively by
\begin{eqnarray*}
\left\{\begin{array}{llll}
\Pi_{y_s}=\pi_1(y_s),
&\Sigma_{y_s}=y_s
&\mbox{for }y_s\in Y,\\
\Pi_{l}=[\Pi_{l_1},\Pi_{l_2}],
&\Sigma_l=\displaystyle\sum{(*)}\frac{y_{s_{k_1}+\ldots+s_{k_i}}}{i!}\Sigma_{l_1\ldots l_n},
&{\displaystyle\mbox{for }l\in\Lyn Y-Y\atop\displaystyle st(l)=(l_1,l_2),}\\
\Pi_{w}=\Pi_{l_1}^{i_1}\ldots\Pi_{l_k}^{i_k},
&\Sigma_w=\displaystyle\frac{\Sigma_{l_1}^{\stuffle i_1}\stuffle\ldots\stuffle\Sigma_{l_k}^{\stuffle i_k}}{i_1!\ldots i_k!},
&{\displaystyle\mbox{for }w=l_1^{i_1}\ldots l_k^{i_k},\mbox{ with }l_1,\ldots,\atop\displaystyle l_k\in\Lyn Y,l_1>\ldots>l_k.}
\end{array}\right.
\end{eqnarray*}
In $(*)$, the sum is taken over all $\{k_1,\ldots,k_i\}\subset\{1,\ldots,k\}$
and $l_1\ge\ldots\ge l_n$ such that $(y_{s_1},\ldots,y_{s_k})\stackrel{*}
{\Leftarrow}(y_{s_{k_1}},\ldots,y_{s_{k_i}},l_1,\ldots,l_n)$, where
$\stackrel{*}{\Leftarrow}$ denotes the transitive closure of the relation
on standard sequences, denoted by $\Leftarrow$ \cite{SLC74,reutenauer}.

To end this section, let us extend $\conc$ and $\shuffle$, for any series $S,R\in\ncs{A}{\calX}$, by
\begin{eqnarray}\label{extenoverseries}
SR&=&\sum_{w\in\calX^*}\biggl(\sum_{u,v\in\calX^*\atop uv=w}\scal{S}{u}\scal{R}{v}\biggr)w,\\
S\shuffle R&=&\sum_{u,v\in\calX^*}\scal{S}{u}\scal{R}{v}u\shuffle v,\\
S\stuffle R&=&\sum_{u,v\in Y^*}\scal{S}{u}\scal{R}{v}u\stuffle v.
\end{eqnarray}
Let us also extend the coproduct $\Delta_{\stuffle}$ (resp. $\Delta_{\conc}$ and $\Delta_{\shuffle}$)
given in \eqref{Dstuffle} (resp. \eqref{Dconc} and \eqref{Dshuffle})
over $\ncs{A}{Y}$ (resp. $\ncs{A}{\calX}$) by linearity as follows
\begin{eqnarray}
\forall S\in\ncs{A}{Y},\quad
\Delta_{\stuffle}S&=&\sum_{w\in Y^*}\scal{S}{w}\Delta_{\stuffle}w\in{\ncs{A}{Y^*\otimes Y^*}},\label{D1}\\
\forall S\in\ncs{A}{\calX},\quad
\Delta_{\shuffle}S&=&\sum_{w\in\calX^*}\scal{S}{w}\Delta_{\shuffle}w\in{\ncs{A}{\calX^*\otimes\calX^*}},\label{D2}\\
\forall S\in\ncs{A}{\calX},\quad
\Delta_{\conc}S&=&\sum_{w\in\calX^*}\scal{S}{w}\Delta_{\conc}w\in{\ncs{A}{\calX^*\otimes\calX^*}}.\label{D3}
\end{eqnarray}

\begin{remark}
In \eqref{D1}--\eqref{D3}, if $A=K$, being a field, then $\ncs{K}{\calX}\otimes\ncs{K}{\calX}$
embeds (injectively) in $\ncs{K}{\calX^*\otimes\calX^*}\cong\ncs{[\ncs{K}{\calX}]}{\calX}$.
Indeed, $\ncs{K}{\calX}\otimes\ncs{K}{\calX}$ contains the elements of the form
$\sum_{i\in I}G_i\otimes D_i$, for some $I$ finite and $(G_i,D_i)\in\ncs{K}{\calX}\times\ncs{K}{\calX}$.
But $\sum_{i\ge0}u^i\otimes v^i$ belongs to $\ncs{K}{\calX^*\otimes\calX^*}$
and does not belong to $\ncs{K}{\calX}\otimes\ncs{K}{\calX}$, for $u,v\in\calX^{\ge1}$.
\end{remark}

\begin{definition}\label{dec0}
Any series $S\in\ncs{A}{Y}$ (resp. $\ncs{A}{\calX}$) is said to be
\begin{enumerate}
\item a $\stuffle$ (resp. $\shuffle,\conc$)-character of $(\ncp{A}{X},\conc,1_{X^*})$ iff,
for any $u,v\in Y^*$ (resp. $\calX^*$), one has $\scal{S}{1_{Y^*}}=1_A$
(resp. $\scal{S}{1_{\calX^*}}=1_A$) and $\scal{S}{u\stuffle v}=\scal{S}{u}\scal{S}{v}$
(resp. $\scal{S}{u\shuffle v}=\scal{S}{u}\scal{S}{v},\scal{S}{uv}=\scal{S}{u}\scal{S}{v}$).
\item an infinitesimal $\stuffle$ (resp. $\shuffle,\conc$)-character of $(\ncp{A}{X},\conc,1_{X^*})$
iff, for any $u,v\in Y^*$ (resp. $\calX^*$), one has
$\scal{S}{u\stuffle v}=\scal{S}{u}\scal{v}{1_{Y^*}}+\scal{u}{1_{Y^*}}\scal{S}{v}$
(resp. $\scal{S}{u\shuffle v}=\scal{S}{u}\scal{v}{1_{Y^*}}+\scal{u}{1_{Y^*}}\scal{S}{v},
\scal{S}{uv}=\scal{S}{u}\scal{v}{1_{Y^*}}\allowbreak+\scal{u}{1_{Y^*}}\scal{S}{v}$).
Moreover, if $\scal{S}{1_{Y^*}}=1_A$ (resp. $\scal{S}{1_{\calX^*}}=1_A$)
then $\scal{S}{u\stuffle v}=0$ (resp. $\scal{S}{u\shuffle v}=0,\scal{S}{uv}=0$).
\end{enumerate}
\end{definition}

\subsection{Representative series}
Representative (or rational) series are the representative functions on the free monoid\footnote{
These functions were considered on groups in \cite{Cartier,ChariPressley}.} \cite{SDSC}.

\begin{definition}\label{dec0}
For $\stuffle$ (resp. $\shuffle$ and $\conc$), a series $S\in\ncs{A}{Y}$
(resp. $\ncs{A}{\calX}$) is said to be
\begin{enumerate}
\item a character of $\ncp{A}{Y}$ (resp. $\ncp{A}{\calX}$) if and only if
$\scal{S}{1_{Y^*}}=1_A$ (resp. $\scal{S}{1_{\calX^*}}\allowbreak=1_A$)
and $\scal{S}{u\stuffle v}=\scal{S}{u}\scal{S}{v}$
(resp. $\scal{S}{u\shuffle v}=\scal{S}{u}\scal{S}{v}$ and $\scal{S}{uv}=\scal{S}{u}\scal{S}{v}$).

\item an infinitesimal character of $\ncp{A}{Y}$ (resp. $\ncp{A}{\calX}$) if and only if
$\scal{S}{u\stuffle v}=\scal{S}{u}\scal{v}{1_{Y^*}}+\scal{u}{1_{Y^*}}\scal{S}{v}$
(resp. $\scal{S}{u\shuffle v}=\scal{S}{u}\scal{v}{1_{Y^*}}+\scal{u}{1_{Y^*}}\scal{S}{v}$
and $\scal{S}{uv}=\scal{S}{u}\scal{v}{1_{Y^*}}+\scal{u}{1_{Y^*}}\scal{S}{v}$),
for $u$ and $v\in Y^*$ (resp. $\calX^*$).
Moreover, if $\scal{S}{1_{Y^*}}=1_A$ (resp. $\scal{S}{1_{\calX^*}}=1_A$) then
$\scal{S}{u\stuffle v}=0$ (resp. $\scal{S}{u\shuffle v}=0$ and $\scal{S}{uv}=0$).
\end{enumerate}
\end{definition}

\begin{definition}\label{dec1}
Let $S\in\ncs{A}{\calX}$ (resp. $\ncp{A}{\calX}$) and $P\in\ncp{A}{\calX}$ (resp. $\ncs{A}{\calX}$).
\begin{enumerate}
\item The {\it left} (resp. {\it right}) \textit{shift}\footnote{These are called {\it residuals}
and extend shifts of functions in harmonic analysis \cite{jacob}.}
of $S$ by $P$, is $P\rd S$ (resp. {$S\rg P$}) defined, for any $w\in\calX^*$, by
$\scal{{P\rd S}}{w}=\scal{S}{wP}$ (resp. $\scal{{S\rg P}}{w}=\scal{S}{Pw}$).

\item Let $S\in\ncs{A}\calX $ such that $\scal{S}{1_{\calX^*}}=0$.
The Kleene star of $S$, denoted by $S^*$, is the infinite sum $1+S+S^2+\cdots$.

\item\label{iii} Let  $A=K$  be a field. One defines also the Sweedler's dual
$\calH_{\shuffle}^{\circ}(\calX)$ (resp. $\calH_{\stuffle}^{\circ}(Y)$) of
$\calH_{\shuffle}(\calX)$ (resp. $\calH_{\stuffle}(Y)$), for a finite set $I$ by
\begin{eqnarray*}
S\in\calH_{\shuffle}^{\circ}(\calX)\mbox{ (resp. $\calH_{\stuffle}^{\circ}(Y)$)}
&\iff&\Delta_{\conc}(S)=\sum_{i\in I}G_i\otimes D_i,
\end{eqnarray*}
where $\{G_i,D_i\}_{i\in I}$ are series which can be chosen in $\calH_{\shuffle}^{\circ}(\calX)$ (see \cite{CM}).
\end{enumerate}
\end{definition}

\begin{theorem}[\cite{DR,SDSC,orlando,reutenauer}]\label{residual}
Let $S\in\ncs{A}{\calX}$. Then the following assertions are equivalent
\begin{enumerate}
\item\label{i} The shifts $\{S\trl w\}_{w\in\calX^*}$ (resp. $\{w\rd S\}_{w\in \calX^*}$)
lie in a finitely generated shift-invariant $A$-module \cite{jacob2}.

\item The series $S$ belongs to the (algebraic) closure of $\widehat{A.\calX}$ by the rational operations\footnote{
In here, $\widehat{A.\calX}$ is understood as the set of all homogenuous series of degree $1$, of the form
$\sum\limits_{x \in \calX} a_xx$.} $\{\conc,+,*\}$ (within $\ncs{A}\calX$). 

\item There is a linear representation $(\nu,\mu,\eta)$, of rank $n$, for $S$ with
$\nu\in M_{1,n}(A),\allowbreak\eta\in M_{n,1}(A)$ and $\mu:\calX^*\rightarrow M_{n,n}(A)$
such that $\scal{S}{w}=\nu\mu(w)\eta$, for $w\in\calX^*$.
\end{enumerate}
\end{theorem}

A series satisfying one of the conditions of Theorem \ref{residual} is called \textit{rational}.
The set of these series, a $A$-module\footnote{In fact (we will see it) a unital $A$-algebra for
$\conc$ and $\shuffle.$} denoted by ${\ncs{A^{\mathrm{rat}}}\calX}$, isclosed by $\{\conc,+,*\}$.
The set of rational series,
a $A$-module, is denoted by ${\ncs{A^{\mathrm{rat}}}\calX}$ (which is closed by $\{\conc,+,*\}$).

\begin{remark}
With the notations in Definition \ref{dec1}, these shifts are associative and mutually commute, \textit{i.e.}
\begin{eqnarray*}
S\rg(PR)=(S\rg P)\rg R,&P\rd(R\rd S)=(PR)\rd S,&(P\rg S)\rd R=P\rg(S\rd R)
\end{eqnarray*}
and then
\begin{eqnarray*}
\forall x,y\in\calX,&\forall w\in\calX^*,&x\rd(wy)=(yw)\rg x=\delta_x^yw.
\end{eqnarray*}
\end{remark}

\begin{corollary}[\cite{CM}]\label{Sweedler}
With notations in Definition \ref{dec0}, if $A$ is a field $K$ then there exists a finite double family of series
$(G_i,D_i)_{i\in I}$ which satisfies three assertions of Theorem \ref{dec1} or, equivalently, one of  the following assertions holds.
\begin{enumerate}
\item[4.] $\scal{S}{PQ}=\sum\limits_{i\in I}\scal{G_i}{P}\scal{D_i}{Q}$, for $P$ and $Q\in\ncp{K}{\calX}$ (resp. $\in\ncp{K}{Y}$),
\item[5.] $\Delta_{\conc}(S)=\sum\limits_{i\in I}G_i\otimes D_i$.
\end{enumerate}
Hence, the Sweedler's dual of the bialgebra $\calH_{\shuffle}(\calX)$ (resp. $\calH_{\stuffle}(Y)$) is isomorphic to
$(\ncs{K^{\mathrm{rat}}}{\calX},\shuffle,1_{\calX^*},\Delta_{\conc})$ (resp. $(\ncs{K^{\mathrm{rat}}}{Y},\stuffle,1_{Y^*},\Delta_{\conc}$).
\end{corollary}

\begin{proof}
Let $(\beta,\mu,\eta)$ of dimension $n$ be a linear representation of $S\in\ncs{K^{\mathrm{rat}}}{\calX}$.
It can be also associated to the following linear representations of dimension $n$
\begin{eqnarray*}\label{LR0}
\forall 1\le i\le n,\quad
(\beta,\mu,e_i)\mbox{ of }L_i&\mbox{and}&({}^te_i,\mu,\eta)\mbox{ of }R_i,\cr
\mbox{where }e_i\in\calM_{1,n}(K)&\mbox{and}&{}^te_i=\begin{matrix}(0&\ldots&0&1\!\!_{_{_{\displaystyle i}}}&0&\ldots&0)\end{matrix}.
\end{eqnarray*}
Hence, using the morphism of monoids $\mu$, one has, for any $u$ and $v\in\calX^*$,
\begin{eqnarray*}
\scal{S}{uv}=\beta\mu(u)\mu(v)\eta=\sum_{i=1}^n(\beta\mu(u)e_i)({}^te_i\mu(v)\eta)=\sum_{i=1}^n\scal{L_i}{u}\scal{R_i}{v},\label{LR1}\\
\scal{\Delta_{\tt conc}(S)}{u\otimes v}=\scal{S}{uv}=\sum_{i=1}^n\scal{L_i}{u}\scal{R_i}{v}=\sum_{i=1}^n\scal{L_i\otimes R_i}{u\otimes v}.\label{LR2}
\end{eqnarray*}
One deduces then the following criterion yielding the expected results
\begin{eqnarray*}
S\in\ncs{K^{\mathrm{rat}}}{\calX}\iff\Delta_{\tt conc}(S)=\sum_{i=1}^nL_i\otimes R_i.
\end{eqnarray*}
\end{proof}

One also has the following constructions of linear representations (only the last one is new and
the first ones are already treated in \cite{jacob}, see also \cite{DFLL}). The constructions of
$R_1\shuffle R_2$ and $R_1\stuffle R_2$ base on coproducts and tensor products of linear representations.

\begin{proposition}\label{linearrepresentation}
The module ${\ncs{A^{\mathrm{rat}}}\calX}$ (resp. ${\ncs{A^{\mathrm{rat}}}{Y}}$)
is closed by $\shuffle$ (resp. $\stuffle$). Moreover, for any $i=1,2$, let
$R_i\in{\ncs{A^{\mathrm{rat}}}\calX}$ and $(\nu_i,\mu_i,\eta_i)$ be its
representation of dimension $n_i$. Then the linear representation of
\begin{eqnarray*}
R_i^*&\mbox{is}&\Big(\begin{pmatrix}0&1\end{pmatrix},
\left\{\begin{pmatrix}\mu_i(x)+\eta_i\nu_i\mu_i(x)&0\cr\nu_i\eta_i&0\end{pmatrix}\right\}_{x\in\calX},
\begin{pmatrix}\eta_i\cr1\end{pmatrix}\Big),\cr
\mbox{that of }R_1+R_2&\mbox{is}&\Big(\begin{pmatrix}\nu_1&\nu_2\end{pmatrix},
\left\{\begin{pmatrix}\mu_1(x)&{\bf 0}\cr{\bf 0}&\mu_2(x)\end{pmatrix}\right\}_{x\in\calX},
\begin{pmatrix}\eta_1\cr\eta_2\end{pmatrix}\Big),\cr
\mbox{that of }R_1R_2&\mbox{is}&\Big(\begin{pmatrix}\nu_1&0\end{pmatrix},
\left\{\begin{pmatrix}\mu_1(x)&\eta_1\nu_2\mu_2(x)\cr0&\mu_2(x)\end{pmatrix}\right\}_{x\in\calX},
\begin{pmatrix}\eta_1\mu_2\eta_2\cr\eta_2\end{pmatrix}\Big),\cr
\mbox{that of }R_1\shuffle R_2&\mbox{is}&(\nu_1\otimes\nu_2,\{\mu_1(x)\otimes\mathrm{I}_{n_2}
+\mathrm{I}_{n_1}\otimes\mu_2(x)\}_{x\in\calX},\eta_1\otimes\eta_2),\\
\mbox{that of }R_1\stuffle R_2&\mbox{is}&(\nu_1\otimes\nu_2,
\{\mu_1(y_k)\otimes\mathrm{I}_{n_2}+\mathrm{I}_{n_1}\otimes\mu_2(y_k)\cr
&&+\sum_{i+j=k}\mu_1(y_i)\otimes\mu_2(y_j)\}_{k\ge1},\eta_1\otimes\eta_2).
\end{eqnarray*}
\end{proposition}

\begin{example}[Identity $(-t^2x_0x_1)^*\shuffle(t^2x_0x_1)^*=(-4t^4x_0^2x_1^2)^*$, \cite{words03,orlando}]\label{quiplait}
\centerline{$\begin{array}{ccc}
\begin{tikzpicture}[->,>=stealth',shorten >=1pt,auto,node distance=2cm,
                    semithick,every node/.style={fill=white}]				
  \node[initial,state,accepting] (A)                    {$1$};  
	\node[state]                   (B) [right of=A] {$2$};
  \path (A) edge [bend left] node {$x_0,\mathrm{i}t$} (B)
        (B) edge [bend left] node {$x_1,\mathrm{i}t$}  (A);
\end{tikzpicture}
&&
\begin{tikzpicture}[->,>=stealth',shorten >=1pt,auto,node distance=2cm,
                    semithick,every node/.style={fill=white}]				
  \node[initial,state,accepting] (1)                    {$\mathrm{I}$};  
	\node[state]                   (2) [right of=A] {$\mathrm{II}$};
  \path (A) edge [bend left] node {$x_0,t$} (B)
        (B) edge [bend left] node {$x_1,t$}  (A);
\end{tikzpicture}\cr
(-t^2x_0x_1)^*\leftrightarrow(\nu_2,\{\mu_2(x_0),\mu_2(x_1)\},\eta_2),
&&
(t^2x_0x_1)^*\leftrightarrow(\nu_1,\{\mu_1(x_0),\mu_1(x_1)\},\eta_1).
\end{array}$}

\centerline{$\begin{array}{rlrl}
\nu_1=\begin{pmatrix}1&0\end{pmatrix},
&\mu_1(x_0)=\begin{pmatrix}0&t\cr0&0\end{pmatrix},
&\mu_1(x_1)=\begin{pmatrix}0&0\cr t&0\end{pmatrix},
&\eta_1=\begin{pmatrix}1\cr0\end{pmatrix},\cr
\nu_2=\begin{pmatrix}1&0\end{pmatrix},
&\mu_2(x_0)=\begin{pmatrix}0&\mathrm{i}t\cr0&0\end{pmatrix},
&\mu_2(x_1)=\begin{pmatrix}0&0\cr\mathrm{i}t&0\end{pmatrix},
&\eta_2=\begin{pmatrix}1\cr0\end{pmatrix}.
\end{array}$}
\centerline{$\begin{array}{rcl}
(-t^2x_0x_1)^*\shuffle(t^2x_0x_1)^*&\leftrightarrow&(\nu,\{\mu(x_0),\mu(x_1)\},\eta)\\
&=&(\nu_1\otimes\nu_2,\{\mu_1(x_0)\otimes\mathrm{I}_{n_2}+\mathrm{I}_{n_1}\otimes\mu_2(x_0),\\
&&\mu_1(x_1)\otimes\mathrm{I}_{n_2}+\mathrm{I}_{n_1}\otimes\mu_2(x_1),\eta_1\otimes\eta_2).
\end{array}$}
\centerline{$\begin{array}{rlrl}
\begin{tikzpicture}[->,>=stealth',shorten >=1pt,auto,node distance=33mm,
                    semithick,every node/.style={fill=white}]				
	\node[initial,state,accepting] (A)                    {$(1,\mathrm{I})$};
  \node[state]                   (B) [above right of=A] {$(2,\mathrm{I})$};
  \node[state]                   (D) [below right of=A] {$(2,\mathrm{II})$};
  \node[state]                   (C) [below right of=B] {$(1,\mathrm{II})$};
  \path (A) edge [bend left] node {$x_0,\mathrm{i}t$} (B)
            edge [bend left] node {$x_0,t$}           (D)
        (B) edge [bend left] node {$x_0,t$}           (C)
				    edge [bend left] node {$x_1,\mathrm{i}t$} (A)
        (C) edge [bend left] node {$x_1,t$}           (B)
				    edge [bend left] node {$x_1,\mathrm{i}t$} (D)
        (D) edge [bend left] node {$x_1,t$}           (A)
            edge [bend left] node {$x_0,\mathrm{i}t$} (C);
\end{tikzpicture}
\end{array}$}
\centerline{$\begin{array}{c}
\nu=\begin{pmatrix}1&0&0&0\end{pmatrix},\hfill\eta=\begin{pmatrix}1\cr0\cr0\cr0\end{pmatrix},\cr
\begin{array}{c}
\mu(x_0)=\begin{pmatrix}
0&0&t&0\cr
0&0&0&t\cr
0&0&0&0\cr
0&0&0&0
\end{pmatrix}
+
\begin{pmatrix}
0&\mathrm{i}t&0&0\cr
0&0&0&0\cr
0&0&0&\mathrm{i}t\cr
0&0&0&0
\end{pmatrix}
=\begin{pmatrix}
0&\mathrm{i}t&t&0\cr
0&0&0&t\cr
0&0&0&\mathrm{i}t\cr
0&0&0&0
\end{pmatrix},\\
\mu(x_1)=\begin{pmatrix}
0&0&0&0\cr
0&0&0&0\cr
t&0&0&0\cr
0&t&0&0
\end{pmatrix}
+
\begin{pmatrix}
0&0&0&0\cr
\mathrm{i}t&0&0&0\cr
0&0&0&0\cr
0&0&\mathrm{i}t&0
\end{pmatrix}
=\begin{pmatrix}
0&0&0&0\cr
\mathrm{i}t&0&0&0\cr
t&0&0&0\cr
0&t&\mathrm{i}t&0
\end{pmatrix},
\end{array}
\end{array}$}
\end{example}

\begin{example}[Identity $(-t^2y_2)^*\stuffle(t^2y_2)^*=(-4t^4y_4)^*$, \cite{words03,orlando}]\label{quiplaitbis}

\centerline{$\begin{array}{c}
\begin{tikzpicture}[->,>=stealth',shorten >=1pt,auto,node distance=4cm,
                    semithick,every node/.style={fill=white}]				
  \node[initial,state,accepting] (A)              {$1$};  
	\node[initial,state,accepting] (B) [right of=A] {$2$};
	\node[initial,state,accepting] (C) [right of=B] {$3$};
  \path (A) edge [loop,above] node {$y_2,-t^2$} (A)
        (B) edge [loop,above] node {$y_2, t^2$} (B)
        (C) edge [loop,above] node {$y_4,-t^4$} (C);
\end{tikzpicture}\cr
\begin{array}{r}(-t^2y_2)^*\leftrightarrow(\nu_2,\mu_2(y_2),\eta_2)\phantom{,}\\=(1,-t^2,1),\end{array}
\begin{array}{r}( t^2y_2)^*\leftrightarrow(\nu_1,\mu_1(y_2),\eta_1)\phantom{,}\\=(1,t^2,1),\end{array}
\begin{array}{r}(-t^4y_4)^*\leftrightarrow(\nu,\mu(y_4),\eta)\phantom{.}\\=(1,-t^4,1).\end{array}
\end{array}$}
\end{example}

\begin{definition}\label{cylindric}
Any series $S\in\ncs{A}{{\calX}}$ is called
\begin{enumerate}
\item syntactically exchangeable if and only if it is constant on multi-homogeneous classes, \textit{i.e.}
$\begin{array}{lcr}
(\forall u,v\in\calX^*)([(\forall x\in\calX)(|u|_x=|v|_x)]&\Rightarrow&\scal{S}{u}=\scal{S}{v}).
\end{array}$
The set of syntactically exchangeable series is denoted by $\ncs{A_{\mathrm{exc}}^{\mathrm{synt}}}{\calX}$.

\item rationally exchangeable if and only if it admits a representation $(\nu,\mu,\eta)$
such that the matrices $\{\mu(x)\}_{x\in \calX}$ commute and the set of these series, a shuffle
subalgebra of $\ncs{A}{X}$, is denoted by $\ncs{A_{\mathrm{exc}}^{\mathrm{rat}}}{\calX}$.
\end{enumerate}
\end{definition}   

\begin{remark}
$S$ is syntactically exchangeable if and only if it is of the form
\begin{eqnarray*}
S=\sum_{\alpha\in\N^{(\calX)},\mathrm{supp}(\alpha)=\{x_1,\ldots,x_k\}}
s_{\alpha},x_1^{\alpha(x_1)}\shuffle\ldots\shuffle x_k^{\alpha(x_k)}.
\end{eqnarray*}
When $A$ is a field, the rational exchangeable series are exactly those who admit a representation with commuting matrices
(at least the minimal one is such) and it is taken as definition as, even for rings, implying syntactic exchangeability. 
\end{remark}

\begin{theorem}[See \cite{KSO,CM}]\label{exchangeable}
\begin{enumerate}
\item\label{ratex1} In all cases, one has
\begin{eqnarray*}
\ncs{A_{\mathrm{exc}}^{\mathrm{rat}}}{\calX}
\subset\ncs{A^{\mathrm{rat}}}{\calX}\cap\ncs{A_{\mathrm{exc}}^{\mathrm{synt}}}{\calX}.
\end{eqnarray*}
The equality holds when $A$ is a field and
\begin{eqnarray*}
\ncs{A_{\mathrm{exc}}^{\mathrm{rat}}}{X}
=\ncs{A^{\mathrm{rat}}}{x_0}\shuffle\ncs{A^{\mathrm{rat}}}{x_1}=\shuffle\limits_{x\in X}\ncs{A^{\mathrm{rat}}}{x},\cr
\ncs{A_{\mathrm{exc}}^{\mathrm{rat}}}{Y}\cap\ncs{A_{\mathrm{fin}}^{\mathrm{rat}}}{Y}
=\bigcup\limits_{k\ge0}\shuffle\limits_{j=1}^k\ncs{A^{\mathrm{rat}}}{y_j}
\subsetneq\ncs{A_{\mathrm{exc}}^{\mathrm{rat}}}{Y},
\end{eqnarray*}
where
$\ncs{A_{\mathrm{fin}}^{\mathrm{rat}}}{Y}=\bigcup\limits_{F\subset_{finite}Y}\ncs{A^{\mathrm{rat}}}{F}$,
the algebra of series over finite subalphabets.
\item\label{Kronecker} (Kronecker's theorem \cite{berstel}) One has
$\ncs{A^{\mathrm{rat}}}{x}=\{P(1-xQ)^{-1}\}_{P,Q\in A[x]}$ (for $x\in\calX$) and if $A=K$
then it is an algebraically closed field of characteristic zero one also has
$\ncs{K^{\mathrm{rat}}}{x}=\span_K\{(ax)^*\shuffle\ncp{K}{x}\vert a\in K\}$.
\item\label{conccharacter} Series $\Big(\sum\limits_{x\in \calX}\alpha_xx\Big)^*$ are $\conc$-characters. Any
$\conc$-character is of this form.
\item\label{indepchargen}
$A$ is supposed without zero divisors. If the family $(\varphi_i)_{i\in I}$is $\Z$-linearly independent within $\widehat{A\calX}$ then the
family $\Lyn(\calX)\uplus\{\varphi_i^*\}_{i\in I}$ is $A$-algebraically free within $(\ncs{A^{\mathrm{rat}}}{\calX},\shuffle,1_{\calX^*})$. 
\item\label{indepchar} In particular, if $A$ is a ring without zero divisors
$\{x^*\}_{x\in\calX}$ (resp. $\{y^*\}_{y\in Y}$) are algebraically independent over
$(\ncp{A}{\calX},\shuffle,1_{\calX^*})$ (resp. $(\ncp{A}{Y},\stuffle,1_{Y^*})$)
within $(\ncs{A^{\mathrm{rat}}}{\calX},\shuffle,1_{\calX^*})$
(resp. $(\ncs{A^{\mathrm{rat}}}{Y},\stuffle,1_{Y^*})$).
\end{enumerate}
\end{theorem}

\begin{proof}
\begin{enumerate}
\item The inclusion is obvious in view of Definition \ref{cylindric}.
For the equality, it suffices to prove that, when $A$ is a field,  every rational and
exchangeable series admits a representation with commuting matrices. This is true of
any minimal representation as shows the computation of shifts (see \cite{DR,KSO,CM}).

Now, if $\calX$ is finite, then (all matrices commute)
\begin{eqnarray*}
\sum_{w\in\calX^*}\mu(w)w=\Big(\sum_{x\in\calX}\mu(x)x\Big)^*=\shuffle_{x\in\calX}(\mu(x)x)^*
\end{eqnarray*}
and the result comes from the fact that $R$ is a linear combination of matrix elements.
As regards the second equality, inclusion $\supset$ is straightforward. We remark that
$\bigcup\limits_{k\ge1}\ncs{A^{\mathrm{rat}}}{y_1}\shuffle\ldots\shuffle\ncs{A^{\mathrm{rat}}}{y_k}$
is directed as these algebras are nested in one another. With this in view, the reverse inclusion comes from the fact that every
$S\in\ncs{A_{\mathrm{fin}}^{\mathrm{rat}}}{Y}$ is a series over a finite alphabet and the result follows from the first equality.
\item Let $\calA=\{P(1-xQ)^{-1}\}_{P,Q\in A[x]}$. Since $P(1-xQ)^{-1}=P(xQ)^{*}$ then it is obvious that
$\calA\subset\ncs{A^{\mathrm{rat}}}{x}$. Next, it is easy to check that $\calA$ contains $\ncp{A}{x}(=A[x])$
and is closed by $+$ and $\conc$ as, for instance,
\begin{eqnarray*}
(1-xQ_1)(1-xQ_2)=(1-x(Q_1+Q_2-xQ_1Q_2)).
\end{eqnarray*}
We need to prove that $\calA$ is closed for $*$. For this to be applied to $P(1-xQ)^{-1}$,
we must suppose that $P(0)=0$ (as, indeed, $\scal{P(1-xQ)^{-1}}{1_{x^*}}=P(0)$) and,
in this case, $P=xP_1$. Now
\begin{eqnarray*}
\Big(\frac{P}{1-xQ}\Big)^*=\Big(1-\frac{P}{1-xQ}\Big)^{-1}=\frac{1-xQ}{1-x(Q+P_1)}\in\calA.
\end{eqnarray*}

\item Let $S=\Big(\sum\limits_{x\in\calX}\alpha_x\,x\Big)^*$. Then $\scal{S}{1_{\calX^*}}=1_A$. Furthermore, if $w=xu$ then $\scal{S}{xu}=\alpha_x\scal{S}{u}$.
Thus, by recurrence on the length, $\scal{S}{x_1\ldots x_k}=\prod\limits_{i=1}^k\alpha_{x_i}$
showing that $S$ is a $\conc$-character. Conversely, by Sch\"utzenberger's reconstruction
lemma, \textit{i.e.} for any series $S$ one has $$S=\scal{S}{1_{\calX^*}}.1_A+\sum\limits_{x\in\calX}x.x^{-1}S.$$
But, if $S$ is a $\conc$-character (\textit{i.e.} $\scal{S}{1_{\calX^*}}=1$ and $x^{-1}S=\scal{S}{x}S$)
then the previous expression reads $S=1_A+\Big(\sum\limits_{x\in\calX}\scal{S}{x}x\Big)S$. The last 
equality is equivalent to $S=\Big(\sum\limits_{x\in\calX}\scal{S}{x}x\Big)^*$ proving the claim. 

\item As $(\ncp{A}{\calX},\shuffle)$ and $(\ncp{A}{Y},\stuffle)$ are enveloping algebras,
this property is an application of the fact that, on an enveloping  $\mathcal{U}$, the
characters are linearly independent with respect to the convolution algebra $\mathcal{U}^*_{\infty}$
(see the general construction and proof in \cite{DGM}. Here, this  convolution algebra
($\mathcal{U}^*_{\infty}$) contains the polynomials (is equal in case of finite $\calX$).
Now, consider a monomial 
$(\varphi_{i_1}^*)^{\shuffle\alpha_1}\ldots(\varphi_{i_n}^*)^{\shuffle\alpha_n}
=\Big(\sum\limits_{k=1}^n\alpha_{i_k}\varphi_{i_k}\Big)^*$.
The $\Z$-linear independence of the monomials in $(\varphi_i)_{i\in I}$
implies that all these monomials are linearly independent over $\ncp{A}{\calX}$
which proves  algebraic independence of the family $(\varphi_i)_{i\in I}$.

To end with, the fact that $\Lyn(\calX)\uplus\{\varphi_i^*\}_{i\in I}$ is algebraically free comes 
from Radford theorem $(\ncp{A}{\calX},\shuffle,1_{\calX^*})\simeq A[\Lyn(\calX)]$ and the transitivity 
of polynomial algebras (see \cite{B_Alg1}).

\item Comes directly as an application of the preceding point.
\end{enumerate}
\end{proof}

\begin{remark}\label{starofplane}
\begin{enumerate}
\item The last inclusion of item \ref{ratex1} is strict as shows the example
of the following identity, living in $\ncs{A_{\mathrm{exc}}^{\mathrm{rat}}}{Y}$ but not in 
$\ncs{A_{\mathrm{exc}}^{\mathrm{rat}}}{Y}\cap \ncs{A_{\mathrm{fin}}^{\mathrm{rat}}}{Y}$
\begin{eqnarray*}
(ty_1+t^2y_2+\ldots)^*&=&\lim\limits_{k\rightarrow+\infty}(ty_1+\ldots+t^ky_k)^*\cr
&=&\lim\limits_{k\rightarrow+\infty}(ty_1)^*\shuffle\ldots\shuffle(t^ky_k)^*=\shuffle_{k\ge1}(t^ky_k)^*.
\end{eqnarray*}

\item Item \ref{Kronecker} can be rephrased in terms of stars
as $\ncs{A^{\mathrm{rat}}}{x}=\{P(xQ)^*\}_{P,Q\in A[x]}$ holds for every ring and is therefore
characteristic free, unlike the shuffle version requiring algebraic closure and denominators.
\end{enumerate}
\end{remark}

\begin{corollary}[Kleene stars of the plane]\label{Kleene}
Let $R$ and $L\in\ncs{A^{\mathrm{rat}}}{\calX}$, such that $\scal{R}{1_{\calX^*}}=1_A$, $\scal{L}{1_{\calX^*}}=0$ and $L^*=R$. Then the following assertions are equivalent.
\begin{enumerate}
\item\label{caracter} $R$ is a $\conc$-character of $(\ncp{A}{\calX},\conc,1_{\calX^*})$.
\item\label{exp_rat} There is a family of coefficients $(c_x)_{x\in\calX}$ such that
\begin{eqnarray*}
R=\Big(\sum_{x\in\calX}c_xx\Big)^*.
\end{eqnarray*}
\item\label{lin_rep} The series $R$ admits a linear representation of dimension one\footnote{
The dimension is here (as in \cite{berstel}) the size of the matrices.}.
\item\label{plane} $L$ belongs to the plane ${A.\calX}$.
\item\label{log} $L$ is an infinitesimal $\conc$-character of $(\ncp{A}{\calX},\conc,1_{X^*})$.
\end{enumerate}
\end{corollary}

\begin{proof}
\underline{\ref{caracter} $\Leftrightarrow$ \ref{exp_rat}:}
This corresponds to the point \ref{conccharacter} of Theorem \ref{exchangeable} above.

\underline{\ref{exp_rat} $\Leftrightarrow$ \ref{lin_rep}}
This is a direct consequence of Theorem \ref{residual}.

\underline{\ref{exp_rat} $\Leftrightarrow$ \ref{plane}:}
This is obvious, by construction (in which $L$ is viewed as the
$\shuffle$-logarithm of $R$). Indeed, since $(c_x x)^n=\dfrac{(c_xx)^{\shuffle n}}{n!}$
for any $n\in\mathbb{N}$,  doing as in Remark \ref{starofplane}, one has
\begin{eqnarray*}
R=\Big(\sum_{x\in\calX}c_xx\Big)^*=\shuffle_{x\in\calX}(c_xx)^*=
\shuffle_{x\in\calX}\exp_{\shuffle}(c_xx)=\exp_{\shuffle}\Big(\sum_{x\in\calX}c_xx\Big).
\end{eqnarray*}

\underline{\ref{plane} $\Leftrightarrow$ \ref{log}:}
If $L$ is an infinitesimal character then, by Definition \ref{dec0}, $$\scal{L}{uv}=
\scal{L}{u}\scal{v}{1_{\calX^*}}+\scal{u}{1_{\calX^*}}\scal{L}{v},$$ for $u,v\in\calX^*$. 
Hence, for $w=uv\in\calX^{\ge2}$ with $u,v\in\calX^+$, one gets $\scal{L}{w}=\scal{L}{uv}=0$.
In addition, for $u=v=1_{\calX^*}$, one also gets $\scal{L}{1_{\calX^*}}=0$ and it follows that
$L=\sum\limits_{x\in\calX}\scal{L}{x}x$. Conversely, since for any $u,v\in\calX^+$ and $x\in\calX$,
one has $$\scal{uv}{x}=\scal{u}{x}\scal{v}{1_{\calX^*}}+\scal{u}{1_{\calX^*}}\scal{v}{x}=0$$
then, by the pairing in \eqref{pairing}, one deduces that
$$\scal{L}{uv}=\sum\limits_{x\in\calX}\scal{L}{x}\scal{uv}{x}=0$$
meaning that $L$ is an infinitesimal $\conc$-character.
\end{proof}

\begin{remark}\label{reetheorem}
In Corollary \ref{Kleene}, if $A=K$ being a field, point \ref{caracter} (resp. \ref{log})
can be rephrased as ``$R$ is a group like element" (resp. ``$L$ is a primitive element")
of ${\ncs{K^{\mathrm{rat}}}\calX}$, for $\Delta_{\conc}$.
Indeed, in \eqref{D1}--\eqref{D3}, if $S\in\ncs{K}{Y}$ (resp. $\ncs{K}{\calX}$ is a
$\stuffle$ (resp. $\shuffle,\conc$)-character of  $(\ncp{K}{Y},\conc,1_{Y^*})$
(resp. $(\ncp{K}{\calX},\conc,1_{\calX^*})$ then
\begin{enumerate}
\item Using the fact that $S\otimes S=\sum\limits_{u,v\in\calX^*}\scal{S}{u}\scal{S}{v}u\otimes v$, one has
$\Delta_{\stuffle}(S)=S\otimes S$ (resp. $\Delta_{\shuffle}(S)=S\otimes S$ and $\Delta_{\conc}(S)=S\otimes S$)
meaning that $S$ is group like, for $\Delta_{\stuffle}$ (resp. $\Delta_{\shuffle}$ and $\Delta_{\conc}$).
\item Using the fact that $\Delta_{\stuffle}$ (resp. $\Delta_{\shuffle}$ and $\Delta_{\conc}$)
and the maps $T\longmapsto T\otimes1_{Y^*}$ and $T\longmapsto1_{Y^*}\otimes T$
(resp. $T\longmapsto T\otimes1_{\calX^*}$ and $T\longmapsto1_{\calX^*}\otimes T$)
are continuous homomorphisms, one has\footnote{Here, $\log S\otimes1_{Y^*}$ and $1_{Y^*}\otimes\log S$
(resp. $\log S\otimes1_{\calX^*}$ and $1_{\calX^*}\otimes\log S$) commute.}
$\Delta_{\stuffle}(\log S)=\log S\otimes1_{\calX^*}+1_{\calX^*}\otimes\log S$
(resp. $\Delta_{\shuffle}(\log S)=\log S\otimes1_{\calX^*}+1_{\calX^*}\otimes\log S$
and $\Delta_{\conc}(\log S)=\log S\otimes1_{\calX^*}+1_{\calX^*}\otimes\log S$)
meaning that $\log S$ is primitive, for $\Delta_{\stuffle}$ 
(resp. $\Delta_{\shuffle}$ and $\Delta_{\conc}$).
\end{enumerate}
Then $S$ is group like, for $\{\Delta_{\stuffle},\Delta_{\shuffle},\Delta_{\conc}\}$,
 if and only if $\log S$ is primitive meaning that the equivalence, between
\ref{caracter} and \ref{log}, is an extension of the Ree's theorem\footnote{
This theorem was first established for $\shuffle$ \cite{reutenauer} and
then was adapted for $\stuffle$ \cite{CM}.} \cite{reutenauer}.
\end{remark}

\begin{proposition}\label{un}
Let $\{\alpha_x\}_{x\in\calX},\{\beta_x\}_{x\in\calX}$ and $\{a_s\}_{s\ge1},\{b_s\}_{s\ge1}$ be complex numbers. Then
\begin{eqnarray*}
\Big(\sum\limits_{x\in\calX}\alpha_xx\Big)^*\shuffle\Big(\sum\limits_{x\in\calX}\beta_xx\Big)^*
&=&\Big(\sum\limits_{x\in\calX}(\alpha_x+\beta_x)x\Big)^*,\cr
\Big(\sum\limits_{s\ge1}a_sy_s\Big)^*\stuffle\Big(\sum\limits_{s\ge1}b_sy_s\Big)^*
&=&\Big(\sum\limits_{s\ge1}(a_s+b_s)y_s+\sum\limits_{r,s\ge1}a_sb_ry_{s+r}\Big)^*.
\end{eqnarray*}
\end{proposition}

\begin{proof}
Using (\ref{Dshuffle}), for any $x_i\in\calX$, one has
\begin{eqnarray*}
&&\scal{\Big(\sum\limits_{x\in\calX}\alpha_xx\Big)^*\shuffle\Big(\sum\limits_{x\in\calX}\beta_xx\Big)^*}{x_i}\\
&&=\scal{\Big(\sum\limits_{x\in\calX}\alpha_xx\Big)^*\otimes\Big(\sum\limits_{x\in\calX}\beta_xx\Big)^*}{\Delta_{\shuffle}(x_i)}\\
&&=\scal{\Big(\sum\limits_{x\in\calX}\alpha_xx\Big)^*\otimes\Big(\sum\limits_{x\in\calX}\beta_xx\Big)^*}{x_i\otimes1_{X^*}+1_{X^*}\otimes x_i}\\
&&=\alpha_i+\beta_i\\
&&=\scal{\Big(\sum\limits_{x\in\calX}(\alpha_x+\beta_x)x\Big)^*}{x_i}.
\end{eqnarray*}
Using (\ref{Dstuffle}), for any $y_t\in Y$, one also has
\begin{eqnarray*}
&&\scal{\Bigl(\sum_{s\ge1}a_sy_s\Bigr)^*\stuffle\Bigl(\sum_{s\ge1}b_sy_s\Bigr)^*}{y_t}\\
&&=\scal{\Bigl(\sum_{s\ge1}a_sy_s\Bigr)^*\otimes\Bigl(\sum_{s\ge1}b_sy_s\Bigr)^*}{\Delta_{\stuffle}(y_t)}\\
&&=\scal{\Bigl(\sum_{s\ge1}a_sy_s\Bigr)^*\otimes\Bigl(\sum_{s\ge1}b_sy_s\Bigr)^*}
{y_t\otimes1_{Y^*}+1_{Y^*}\otimes y_t+\sum_{r,s\ge1,r+s=t}y_s\otimes y_r}\\
&&=a_t+b_t+\sum_{r,s\ge1,r+s=t}a_sb_r\\
&&=\scal{\Bigl(\sum_{s\ge1}(a_s+b_s)y_s+\sum_{r,s\ge1}a_sb_ry_{s+r}\Bigr)^*}{y_t}.
\end{eqnarray*}
\end{proof}

\begin{corollary}
Let $x\in\calX,y_k\in Y$ and $n\in\N_{\ge1}$. Then
\begin{eqnarray*}
(x^*)^{\shuffle n}=(nx)^*&\mbox{and}&(y_k^*)^{\stuffle n}=\Big(\sum\limits_{i=1}^n(n-i+1)y_{ik}\Big)^*.
\end{eqnarray*}
\end{corollary}

\begin{proof}
It is immediate for $n=1$ and supposed true up to rank $n\ge1$.
Next, since $(x^*)^{\shuffle n+1}=x^*\shuffle(x^*)^{\shuffle n}$ and
$(y_k^*)^{\stuffle n+1}=y_k^*\stuffle(y_k^*)^{\stuffle n}$ then by hypothesis and then,
by Proposition \ref{un}, one deduces that $(x^*)^{\shuffle n+1}=((n+1)x)^*$ and also
\begin{eqnarray*}
(y_k^*)^{\stuffle n+1}
&=&y_k^*\stuffle\Big(\sum\limits_{i=1}^n(n-i+1)y_{ik}\Big)^*\cr
&=&\Big((n+1)y_k^*+\sum\limits_{i=1}^n(n-i+1))y_{(i+1)k}\Big)^*\cr
&=&\Big(\sum_{i=0}^n(n+1-i)y_{(i+1)k}\Big)^*\cr
&=&\Big(\sum_{i=1}^{n+1}(n-i)y_{ik}\Big)^*.
\end{eqnarray*}
\end{proof}

\begin{example}
\begin{enumerate}
\item For any $y_k\in Y$, one has
\begin{eqnarray*}
(y_k^*)^{\stuffle 2}=&y_k^*\stuffle y_k^*&=(2y_k+y_{2k})^*,\cr
(y_k^*)^{\stuffle 3}=&(2y_k+y_{2k})^*\stuffle y_k^*&=(3y_k+2y_{2k}+y_{3k})^*,\cr
(y_k^*)^{\stuffle 4}=&(3y_k+2y_{2k}+y_{3k})^*\stuffle y_k^*&=(4y_k+3y_{2k}+2y_{3k}+y_{4k})^*.
\end{eqnarray*}

\item Let $y_s,y_r\in Y$ and $a_s,a_r\in\C$ (see also Example \ref{quiplaitbis}). Then
\begin{eqnarray*}
(a_sy_s)^*\stuffle(a_ry_r)^*&=&(a_sy_s+a_ry_r+a_sa_ry_{s+r})^*,\\
(-a_sy_s)^*\stuffle(a_sy_s)^*&=&(-a_s^2y_{2s})^*.
\end{eqnarray*}
\end{enumerate}
\end{example}

\subsection{Linear representations of exchangeable rational series over $\serie{A_{\mathrm{exc}}^{\mathrm{rat}}}{X}$}\label{TST2}
As examples, one can consider and examine the linear representations of rational series of the following forms \cite{CM}
\begin{eqnarray}
&E_1x_{i_1}\ldots E_jx_{i_j}E_{j+1},\mbox{ where }x_{i_1},\ldots,x_{i_j}\in X,E_1,\ldots,E_j\in\serie{A^{\mathrm{rat}}}{x_0},\label{F_0}\\
&E_1x_{i_1}\ldots E_jx_{i_j}E_{j+1},\mbox{ where }x_{i_1},\ldots,x_{i_j}\in X,E_1,\ldots,E_j\in\serie{A^{\mathrm{rat}}}{x_1},\label{F_1}\\
&E_1x_{i_1}\ldots E_jx_{i_j}E_{j+1},\mbox{ where }x_{i_1},\ldots,x_{i_j}\in X,E_1,\ldots,E_j\in\serie{A_{\mathrm{exc}}^{\mathrm{rat}}}{X}.\label{F_2}
\end{eqnarray}

\begin{theorem}[Triangular sub bialgebras of $(\ncs{A^{\mathrm{rat}}}{\calX},\shuffle,1_{X^*},\Delta_{\conc})$, \cite{CM}]\label{Subalgebras}
Let $\rho=(\nu,\mu,\eta)$ a representation of $R\in\ncs{A^{\mathrm{rat}}}{\calX}$. Then
\begin{enumerate}
\item\label{comm1} If $\{\mu(x)\}_{x\in \calX}$ mutually commute
and if the alphabet is finite, any rational exchangeable series decomposes as 
$R=\sum\limits_{i=1}^n\shuffle\limits_{x\in \calX}R_x^{(i)}$, with $R_x^{(i)}\in\ncs{A^{\mathrm{rat}}}{x}$.

\item If $\calL$ consists of upper-triangular matrices then $R\in{\ncs{A_{\mathrm{exc}}^{\mathrm{rat}}}{\calX}}\shuffle\ncp{A}{\calX}$.

\item Let $M(x):=\mu(x)x$, for $x\in\calX$. Then $M(R)=\sum\limits_{w\in\calX^*}\scal{R}{w}\mu(w)w$ and then $R=\nu{M(\calX^*)}\eta$.
Moreover,
\begin{enumerate}
\item Since $A$ contains $\Q$ then, by \eqref{diagonalX}, one has $M(\calX^*)=\prod\limits_{l\in\Lyn\calX}^{\searrow}e^{S_l\mu(P_ l)}$.

By \eqref{diagonalY}, one has in addition, $M(Y^*)=\prod\limits_{l\in\Lyn Y}^{\searrow}e^{\Sigma_l\mu(\Pi_ l)}$.

\item If $\{\mu(x)\}_{x\in\calX}$ are upper-triangular then there exists a diagonal (resp. strictly upper-triangular) letter matrix
$D(\calX)$ (resp. $N(\calX)$) such that $M(\calX)=D(\calX)+N(\calX)$ and then, by Lazard factorization \cite{lothaire}, one has
$M(\calX^*)=((D(\calX^*)N(\calX))^*D(\calX^*))$.

\item For $X=\{x_0,x_1\}$, similarly (by Lazard factorization again),
\begin{eqnarray*}
&&M((x_0+x_1)^*)=(M(x_1^*)M(x_0))^*M(x_1^*)=(M(x_0^*)M(x_1))^*M(x_0^*)
\end{eqnarray*}
and the modules generated by the families series in \eqref{F_0}--\eqref{F_2} are closed by $\conc$ and $\shuffle$.

Furthermore, it follows that $R$ is a linear combination of series in \eqref{F_0} (resp. \eqref{F_1}) if $M(x_1^*)M(x_0)$
(resp. $M(x_0^*)M(x_1)$) is strictly upper-triangular.
\end{enumerate}
\end{enumerate}
\end{theorem}

In order to establish Theorem \ref{Subbialgebras} below, we will use the following
\begin{lemma}\label{lemma}
Let $(\nu,\tau,\eta)$ a representation of $S$ of dimension $r$ such that, for all 
$x\in \calX$, ($\tau(x)-c(x)I_r$) is strictly upper triangular, then 
$S\in\ncs{K_{\mathrm{exc}}^{\mathrm{rat}}}{\calX}\shuffle\ncp{K}{\calX}$.
\end{lemma}

\begin{proof}
Let $(e_i)_{1\le i\le r}$ be the canonical basis of $M_{1,r}(K)$.
We construct the representations $\rho_1=(\nu,(x\longmapsto\tau(x)-c(x)I_r),\eta)$, 
$\rho_2=(e_1,(x\longmapsto c(x)I_r),{}^te_1)$ of $S_1$ and $S_2$ and remark that
$S_1\shuffle S_2$ admits the representation 
$\rho_3=(\nu\otimes e_1,((\tau(x)-c(x)I_r)\otimes I_r+I_r\otimes c(x)I_r)_{x\in\calX},\eta\otimes{}^te_1)$
as $I_r\otimes c(x)I_r=c(x)I_r\otimes I_r$, $\rho_3$ is, in fact, 
$(\nu\otimes e_1,(\tau(x)\otimes I_r)_{x\in \calX},\eta\otimes{}^te_1)$
which represents $S$, the result now comes from the fact that $S_1\in\ncp{K}{\calX}$
and $S_2=\Big(\sum\limits_{x\in \calX}c(x)x\Big)^*\in\ncs{K_{\mathrm{exc}}^{\mathrm{rat}}}{\calX}$.
\end{proof}   

We first begin by properties essentially true over algebraically closed fields. 

\begin{theorem}[Triangular sub bialgebras of $(\ncs{K^{\mathrm{rat}}}{\calX},\shuffle,1_{X^*},\Delta_{\conc})$, \cite{CM}]\label{Subbialgebras}
We suppose that $K$ is an algebraically closed field and that $\rho=(\nu,\mu,\eta)$
is a linear representation of $R\in\ncs{K^{\mathrm{rat}}}{\calX}$ of minimal dimension $n$,
we note $\calL=\calL(\mu)\subset K^{n\times n}$ the Lie algebra generated by the matrices 
$(\mu(x))_{x\in\calX}$. Then 
\begin{enumerate}
\item\label{CommRep} $\calL$ is commutative if and only if $R\in \ncs{K_{\mathrm{exc}}^{\mathrm{rat}}}{\calX}$,
\item $\calL$ is nilpotent if and only if 
$R\in{\ncs{K_{\mathrm{exc}}^{\mathrm{rat}}}{\calX}}\shuffle\ncp{K}{\calX}$,
\item $\calL$ is solvable if and only if $R$ is a linear combination of expressions in the form $(F_2)$.
\end{enumerate}

Moreover, denoting $\ncs{K_{\mathrm{nil}}^{\mathrm{rat}}}{\calX}$ (resp. $\ncs{K_{\mathrm{sol}}^{\mathrm{rat}}}{\calX}$),
the set of rational series such that  $\calL(\mu)$ is nilpotent (resp. solvable), we get a tower of sub Hopf algebras of
the Sweedler's dual, $\ncs{K_{\mathrm{nil}}^{\mathrm{rat}}}{\calX}\subset\ncs{K_{\mathrm{sol}}^{\mathrm{rat}}}{\calX}
\subset\calH_{\shuffle}^{\circ}(\calX)$.
\end{theorem}

\begin{proof}
\begin{enumerate}
\item Let us remark that, for $x,y\in\calX,p,s\in \calX^*$, we have 
$\scal{R}{pxys}=\scal{R}{pyxs}$ which is due to the commutation of matrices.
Conversely, since $\rho$ is minimal then there is $P_i,Q_i\in\ncp{K}{\calX},i=1...n$
such that $$\mu(u)=(\scal{P_i\trr R\trl Q_i}{u})_{1\leq i,j\leq n}=(\scal{R}{Q_iuP_i})_{1\leq i,j\leq n},$$
for $u\in\calX^*$ (see \cite{berstel,DR}). 
Now, $$\mu(xy)=(\scal{R}{Q_ixyP_i})_{1\leq i,j\leq n}\stackrel{*}{=}(\scal{R}{Q_iyxP_i})_{1\leq i,j\leq n}=\mu(yx),$$
for $x,y\in \calX$ (equality $\stackrel{*}{=}$ being due to exchangeability).

\item Let $K^n$ be the space of the representation of $\calL$ given by $\mu$ and $\bigoplus\limits_{j=1}^m V_j$
be a decomposition of $K^n$ into indecomposable $\calL$-modules (see \cite{dixmier} for characteristic $0$, or
\cite{Lie7} for arbitrary characteristic), we know that $V_j$ is a $\calL$-module and the action of $\calL$ is
triangularizable with constant diagonals inside each sector $V_j$. Thus, it is an invertible matrix  $P$ in $\mathrm{GL}(n,K)$ such that
\begin{eqnarray*}
\forall x\in \calX,&P\mu(x)P^{-1}=\mathrm{blockdiag}(T_1,T_2\ldots,T_k)=
\begin{pmatrix}
T_1&0&0&\ldots&0\cr
0&T_2&0&\ldots&0\cr
\vdots&\ddots&\ddots&\ddots&\vdots\cr
0&0&\ldots&0&T_k
\end{pmatrix}
\end{eqnarray*}  
where the $T_j$ are upper triangular matrices with scalar diagonal \textit{i.e.} is of the form
$T_j(x)=\lambda(x)I+N(x)$ where $N(x)$ is strictly upper-triangular\footnote{Even, as $K$ is infinite,
there is a global linear form on $\calL$, $\lambda_{lin}$ such that, for all $g\in\calL$,
$PgP^{-1}-\lambda_{lin}(g)I$ is strictly upper-triangular.}. Set $d_j$ to be the dimension of $T_j$
(so that $n=\sum\limits_{j=1}^m\,d_j$), partitioning  $\nu P^{-1}=\nu'$ (resp. $P\eta=\eta'$)
with these dimensions we get blocks so that each $(\nu'_j,T_j,\eta'_j)$ is the representation
of a series $R_j$ and $R=\sum\limits_{j=1}^m\, R_j$. It suffices then to prove that any 
$R_j\in \ncs{K_{\mathrm{exc}}^{\mathrm{rat}}}{\calX}\shuffle\ncp{K}{\calX}$.
This is a consequence of Lemma \ref{lemma}.

Conversely, if $\rho_i=(\nu_i,\tau_i,\eta_i),i=1,2$, are two representations then
$$[\tau_1(x)\otimes I_r+I_r\otimes\tau_2(x),\tau_1(y)\otimes I_r+I_r\otimes\tau_2(y)]
=[\tau_1(x),\tau_1(y)]\otimes I_r+I_r\otimes [\tau_2(x),\tau_2(y)]$$
and a similar formula holds for $m$-fold brackets (Dynkin combs), so that if $\calL(\tau_i)$'s
are nilpotent, the Lie algebra $\calL(\tau_1\otimes I_r+I_r\otimes \tau_2)$ is also nilpotent.
The point here comes from the fact that series in $\ncs{K_{\mathrm{exc}}^{\mathrm{rat}}}{\calX}$
as well as in $\ncp{K}{\calX}$ admit nilpotent representations, so, let $(\alpha,\tau,\beta)$
such a representation and $(\alpha',\tau',\beta')$ its minimal quotient (obtained by minimization,
see \cite{berstel}), then $\calL(\tau')$ is nilpotent as a quotient of $\calL(\tau)$.
Now two minimal representations being isomorphic, $\calL(\mu)$ is isomorphic to $\calL(\tau)$
and then it is nilpotent.

\item As $\calL$ is solvable and $K$ is algebraically closed, by Lie's theorem,
we can find a conjugate form of $\rho=(\nu,\mu,\eta)$ such that $\mu(x)$s are upper-triangular.
Since this form also represents $R$, letting $D(\calX)$ (resp. $N(\calX)$) be the diagonal
(resp. strictly upper-triangular) letter matrice such that $M(\calX)=D(\calX)+N(\calX)$ then 
$R=\nu M(\calX^*)\eta=\nu(D(\calX^*)N(\calX))^*D(\calX^*)\eta .$
Since $D(\calX^*)N(\calX)$ is nilpotent of order $n$ then
$(D(\calX^*)N(\calX))^*=\sum\limits_{j=0}^n(D(\calX^*)N(\calX))^j .$
Hence, letting ${\calS}$ be the vector space generated by forms of type $(F_2)$ which is
closed by concatenation, we have $D(\calX^*)N(\calX)\in{\calS}^{n\times n}$ and then
$(D(\calX^*)N(\calX))^*\in{\calS}^{n\times n}$.
Finally, $R=\nu M(\calX^*)\eta\in{\calS}$ which is the claim.

Conversely, as sums and quotients of solvable representations 
are solvable is suffices to show that a single form of type $F_2$ admits a solvable representation
and end by quotient and isomorphism as in (ii). From Proposition \eqref{linearrepresentation},
we get the fact that, if $R_i$ admit solvable representations so does $R_1R_2$,
then the claim follows from the fact that, firstly, single letters admit solvable (even nilpotent)
representations and secondly series of $\shuffle\limits_{x\in\calX}\{\ncs{K^{\mathrm{rat}}}{x}\}$
admit solvable representations. Finally, we choose (or construct) a solvable representation of $R$,
call it $(\alpha,\tau,\beta)$ and  $(\alpha',\tau',\beta')$ its minimal quotient,
then $\calL(\tau')$ is solvable as a quotient of $\calL(\tau)$. Now two minimal representations
being isomorphic, $\calL(\mu)$ is isomorphic to $\calL(\tau)$, hence solvable.
%
\end{enumerate}
\end{proof}

\begin{remark}
For an example of series $S$ with solvable representation but such that
$S\notin\ncs{K_{\mathrm{exc}}^{\mathrm{rat}}}{\calX}\shuffle\ncp{K}{\calX}$, 
one can take $\calX=\{a,b\}$ and $S=a^*b(-a)^*$.
\end{remark}

\section{Indexation of polylogarithms and harmonic sums}\label{polylogarithms}

\subsection{Indexation by noncommutative generating series}
Any integral multi-indices $(s_1,\ldots,s_r)\in\N_{\ge1}^r,r\in\N_{\ge1},$ can be associated
with $x_0^{s_1-1}x_1\ldots x_0^{s_r-1}x_1\in X^*x_1$. Similarly, any $(s_1,\ldots,s_r)\in\N^r$
can be associated with the word $y_{s_1}\ldots y_{s_r}\in Y_0^*$.
Let then $\Li_{x_0^r}(z)={(\log(z))^r}/{r!}$ and let $\Li_{s_1,\ldots,s_k}$ (resp. $\Li_{-s_1,\ldots,-s_k}$)
and $\H_{s_1,\ldots,s_k}$ (resp. $\H_{-s_1,\ldots,-s_k}$) be indexed by words, \textit{i.e.}
\begin{eqnarray}
	\Li_{x_0^{s_1-1}x_1\ldots x_0^{s_r-1}x_1}=\Li_{s_1,\ldots,s_r}&(\mbox{resp.}&
	\Li^-_{y_{s_1}\ldots y_{s_r}}=\Li_{-s_1,\ldots,-s_r}),\\
	\H_{y_{s_1}\ldots y_{s_r}}=\H_{s_1,\ldots,s_r}&(\mbox{resp.}&
	\H^-_{y_{s_1}\ldots y_{s_r}}=\H_{-s_1,\ldots,-s_r}).
\end{eqnarray}
In particular, $\Li^-_{y_0^r}(z)=({z}/(1-z))^r$
and $\H^-_{y_0^r}(n)=\binom{n}{r}=(n)_r/r!$.

\begin{theorem}[\cite{FPSAC98,words03}]\label{structure1}
	The following morphisms of algebras are {\it injective}
	\begin{eqnarray*}
		\H_{\bullet}:(\Q\pol{Y},\stuffle,1_{Y^*})\longrightarrow(\Q\{\H_w\}_{w\in Y^*},\times,1),&&w\longmapsto\H_w,\label{H}\cr
		\Li_{\bullet}:(\QX,\shuffle,1_{X^*})\longrightarrow(\Q\{\Li_w\}_{w\in X^*},\times,1_{\tilde{B}}),&&w\longmapsto\Li_w\label{Li}
	\end{eqnarray*}
	
	Hence, $\{\H_w\}_{w\in Y^*}$ (resp. $\{\H_l\}_{l\in\Lyn Y}$) and $\{\Li_w\}_{w\in X^*}$
	(resp. $\{\Li_l\}_{l\in\Lyn X}$) are $\Q$-linearly (resp.algebraically ) independent.
\end{theorem}

\begin{theorem}[\cite{Ngo}]\label{structure2}
	There exists also a law of algebra, denoted by $\top$, in $\serie{\Q}{Y_0}$ such that
	the following morphisms of algebras are surjective
	\begin{eqnarray*}
		\H^-_{\bullet}:(\Q\pol{Y_0}, \stuffle,1_{Y_0^*})\longrightarrow(\Q\{\H^-_w\}_{w\in Y_0^*},\times,1),&&w\longmapsto\H^-_w,\cr
		\Li^-_{\bullet}:(\Q\pol{Y_0},\top,1_{Y_0^*})\longrightarrow(\Q\{\Li^-_w\}_{w\in Y_0^*},\times,1_{\Omega}),&&w\longmapsto\Li^-_w.
	\end{eqnarray*}
	Then $\ker\H^-_{\bullet}=\ker\Li^-_{\bullet}=\Q\langle\{w-w\top 1_{Y_0^*}\mid w\in Y_0^*\}\rangle$
	and the families $\{\H^-_{y_k}\}_{k\ge0}$ and $\{\Li^-_{y_k}\}_{k\ge0}$ are $\Q$-linearly independent.
	Moreover, let $\top':{\Q}\pol{Y_0}\times{\Q}\pol{Y_0}\rightarrow{\Q}\pol{Y_0}$
	be a law such that $\Li^-_{\bullet}$ is a morphism for $\top'$ and
	$(1_{Y_0^*}\top'{\Q}\pol{Y_0})\cap\ker(\Li^-_{\bullet})=\{0\}$.
	Then $\top'=g\circ\top$, where $g\in GL({\Q}\pol{Y_0})$
	is such that $\Li^-_{\bullet}\circ g=\Li^-_{\bullet}$.
\end{theorem}

Let $\calG$ denotes the group of transformations, generated by $\{z\mapsto1-z,z\mapsto1/z\}$,
permuting the singularities in $\{0,1,+\infty\}$.
Let us also consider the differential rings (considered as a subring of $\calH(\Omega)$)
\begin{eqnarray}
\calC_0:=\C[z,z^{-1}],&\calC_1:=\C[z,(1-z)^{-1}],&\calC:=\C[z,z^{-1},(1-z)^{-1}],\\
\calC_0':=\C[z^{-1}],&\calC_1':=\C[(1-z)^{-1}],&\calC':=\C[z^{-1},(1-z)^{-1}].
\end{eqnarray}
Then we have

\begin{lemma}[\cite{CM}]\label{stability}
	\begin{enumerate}
		\item For any $G\in\calC$ and $g\in\mathcal{G}, G(g(z))\in\calC$.			
		\item For any $G(z)=p_1(z)+p_2(z^{-1})+p_3((1-z)^{-1})\in\calC$,
		with $p_1,p_2,p_3\in\C[z]$, $p_2(0)=p_3(0)=0$ and $p_2,p_3\neq0$.
		Letting $G_0(z):=P_2(z^{-1})\in\calC'_0$ and
		$G_1(z):=P_3((1-z)^{-1})\in\calC'_1$, one has
		$G(z)\sim_0G_0(z)$ and $G(z)\sim_1G_1(z)$.
		
		\item The morphism $\lambda:(\C[x_0^*,(-x_0)^*,x_1^*],\shuffle,1_{X^*})\rightarrow(\calC,\times,1_{\Omega})$,
		mapping $R$ to $\Li_R$, is {\it surjective} and
		$\ker\lambda$ is the shuffle-ideal generated by $x_0^*\shuffle x_1^*-x_1^*+1$.
		
		\item The following morphisms of algebras are {\it bijective} 
		$$\begin{array}{@{}lllll@{}}
			&\lambda':(\C[x_0^*,x_1^*],\shuffle,1_{X^*})&\longrightarrow&(\calC',\times,1_{\Omega}),&R\longmapsto\Li_R,\\
			\forall i=0,1,&\lambda'_i:(\C[x_i^*],\shuffle,1_{X^*})&\longrightarrow&(\calC'_i,\times,1_{\Omega}),&R\longmapsto\Li_R.
		\end{array}$$
	\end{enumerate}
\end{lemma}

\begin{lemma}[\cite{Ngo}]\label{rat_ext}
	\begin{enumerate}
		\item $\forall x_i\in X,\serie{\C^{\mathrm{rat}}}{x_i}=
		\mathrm{span}_{\C}\{(tx_i)^*\shuffle{\C}\pol{x_i}\vert{t\in\C}\}$.
		
		\item The series $x_0^*,x_1^*$ are transcendent over $\CX$
		and the family $\{x_0^*,x_1^*\}$ is algebraically independent
		over $(\CX,\shuffle)$ within $(\CXX,\shuffle)$.
		
		\item The module $(\CX,\shuffle,1_{X^*})[x_0^*,x_1^*,(-x_0)^*]$ is free over $\CX$ and the family
		$\{(x_0^*)^{\shuffle k}\shuffle (x_1^*)^{\shuffle l}\}^{(k,l)\in\Z\times\N}$
		forms a $\CX$-basis of it. Hence, the family
		$\{w\shuffle(x_0^*)^{\shuffle k}\shuffle(x_1^*)^{\shuffle l}\}^{(k,l)\in\Z\times\N}_{w\in X^*}$
		is a $\C$-basis of it.
	\end{enumerate}
\end{lemma}

\begin{proposition}[\cite{Ngo2,FPSAC98}]\label{integro}
Let $\theta_0:=z\partial_z$ and $\theta_1:=(1-z)\partial_z$ and their sections,
$\iota_0$ and $\iota_1$, such that $\theta_0\iota_0=\theta_1\iota_1=\mathrm{Id}$.

	Then the algebra $\calC\{\Li_w\}_{w\in X^*}$, $\cong\calC\otimes\C\{\Li_w\}_{w\in X^*}$,
	is closed by $\{\theta_0,\theta_1,\iota_0,\iota_1\}$.
	
	Moreover, for any $\ell\in\calC\{\Li_w\}_{w\in X^*}$ and $g\in{\mathcal G},\ell(g(z))\in\calC\{\Li_w\}_{w\in X^*}$.
\end{proposition}

Now, let us consider generating series of polylogarithms and of harmonic sums.

\begin{theorem}[\cite{FPSAC98,VJM}]\label{LH}
	The graphs of $\Li_{\bullet}$ and $\H_{\bullet}$, defined in Theorem \ref{structure1}, read
	\begin{eqnarray*}
		\L=\sum_{w\in X^*}\Li_ww=(\H_{\bullet}\otimes\mathrm{Id}_Y){\mathcal D}_{\stuffle}
		&\mbox{and}&
		\H=\sum_{w\in Y^*}\H_ww=(\Li_{\bullet}\otimes\mathrm{Id}_X){\mathcal D}_X.
	\end{eqnarray*}
	Then $\Delta_{\stuffle}(\H)=\H\otimes\H,\Delta_{\shuffle}(\L)=\L\otimes\L$,
	$\scal{\H}{1_{Y^*}}=\scal{\L}{1_{X^*}}=1$ and, for corresponding co-products,
	their logarithms are primitive. Moreover,
	\begin{eqnarray*}
		\H=\prod_{l\in\Lyn Y}^{\searrow}e^{\H_{\Sigma_l}\Pi_l}
		&\mbox{and}&
		\log(\H)
		=\sum_{k\ge1}\frac{(-1)^{k-1}}k\sum_{u_1,\ldots,u_k\in Y^+}
		\H_{u_1\stuffle\ldots\stuffle u_k}\;u_1\ldots u_k,\\
		\L=\prod_{l\in\Lyn X}^{\searrow}e^{\Li_{S_l}P_l}
		&\mbox{and}&
		\log(\L)
		=\sum_{k\ge1}\frac{(-1)^{k-1}}k\sum_{u_1,\ldots,u_k\in X^+}
		\Li_{u_1\shuffle\ldots\shuffle u_k}\;u_1\ldots u_k.
	\end{eqnarray*}
\end{theorem}
One can then set the following series 
\begin{eqnarray}
	Z_{\stuffle}:=\prod_{l\in\Lyn Y\setminus\{y_1\}}^{\searrow}e^{\H_{\Sigma_l}(+\infty)\Pi_l}
	&\mbox{and}&Z_{\shuffle}:=\prod_{l\in\Lyn X\setminus X}^{\searrow}e^{\Li_{S_l}(1)P_l}.
\end{eqnarray}

By term wise differentiation, the series $\L$ in Theorem \ref{LH} satisfies \eqref{DE} \cite{FPSAC98}.
Definition \eqref{zetavalues} and Theorem \ref{structure1}
lead then to the following {\it surjective} polymorphism
\begin{eqnarray}\label{zeta}
	\zeta:{\displaystyle(\Q1_{X^*}\oplus x_0\QX x_1,\shuffle,1_{X^*})\atop
		\displaystyle
		(\Q1_{Y^*}\oplus(Y\setminus\{y_1\})\Q\pol{Y},\stuffle,1_{Y^*})}&\longtwoheadrightarrow&(\calZ,\times,1),\\
	{\displaystyle x_0x_1^{s_1-1}\ldots x_0x_1^{s_k-1}\atop
		\displaystyle y_{s_1}\ldots y_{s_k}}&\longmapsto&\sum_{n_1>\ldots>n_k>0}{n_1^{-s_1}\ldots n_k^{-s_k}},\notag
\end{eqnarray}
where $\calZ$ is the $\Q$-algebra generated by
$\{\zeta(l)\}_{l\in\Lyn X\setminus X}$ (resp. $\{\zeta(S_l)\}_{l\in\Lyn X\setminus X}$),
or equivalently, generated by $\{\zeta(l)\}_{l\in\Lyn Y\setminus\{y_1\}}$
(resp. $\{\zeta(\Sigma_l)\}_{l\in\Lyn Y\setminus\{y_1\}}$).

\subsection{Global asymptotic behaviors by noncommutative generating series}\label{Global}
Singularities analysis on the coefficients of the noncommutative
generating series of $\{\Li_w\}_{w\in X^*}$, putted in the factorized
form (see Theorem \ref{LH}) leads to the following asymptotic behavior \cite{FPSAC98}
\begin{eqnarray}\label{globalasymptotic}
	\L(z)\sim_0\exp(x_0\log z)&\mbox{and}&
	\L(z)\sim_1\exp(-x_1\log(1-z))Z_{\shuffle}.
\end{eqnarray}
In \cite{CASC2018,CM}, the profs of the uniqueness of the series $\L$
and $Z_{\shuffle}$ (\textit{i.e.} $\Phi_{KZ}$) are also given.
Via an identity of the type Newton-Girard, we obtain \cite{Daresbury,JSC}
\begin{eqnarray}\label{asymptotic}
	\H(n)\sim_{+\infty}\sum_{k\ge0}\H_{y_1^k}y_1^k\pi_Y(Z_{\shuffle})
	&\mbox{and}&
	\sum_{k\ge0}\H_{y_1^k}y_1^k=e^{\sum_{k\ge1}\H_{y_k}(n){(-y_1)^k}/k}.
\end{eqnarray}
In other terms, we have the following global renormalization

\begin{theorem}[first Abel like theorem, \cite{Daresbury,JSC}]\label{renormalization1}
	Let
	\begin{eqnarray*}
		\pi_Y:(A1_{X^*}\oplus\AX x_1,{\tt conc},1_{Y^*})\rightarrow(\AY,{\tt conc},1_{Y^*})
	\end{eqnarray*}
	be the morphism, mapping $x_0^{s_1-1}x_1\ldots x_0^{s_r-1}x_1$ to $y_{s_1}\ldots y_{s_r}$. Then
	\begin{eqnarray*}
		\lim_{z\rightarrow 1} e^{y_1\log(1-z)}\pi_Y(\L(z))
		=\lim_{n\rightarrow\infty}e^{\sum_{k\ge1}\H_{y_k}(n){(-y_1)^k}/k}\H(n)
		=\pi_Y({Z}_{\minishuffle}).
	\end{eqnarray*}
\end{theorem}

Thus, the coefficients $\{\langle Z_{\shuffle}\vert u\rangle\}_{u\in X^*}$
(\textit{i.e.} $\{\zeta_{\shuffle}(u)\}_{u\in X^*}$) and
$\{\langle Z_{\stuffle}\vert v\rangle\}_{v\in Y^*}$
(\textit{i.e.} $\{\zeta_{\stuffle}(v)\}_{v\in Y^*}$) represent, respectively,
\begin{eqnarray}
	\mathrm{f.p.}_{z\rightarrow1}\Li_w(z)=\zeta_{\shuffle}(w),
	&&\{(1-z)^a\log^b((1-z)^{-1})\}_{a\in\Z,b\in\N},\cr
	\mathrm{f.p.}_{n\rightarrow+\infty}\H_w(n)=\zeta_{\stuffle}(w),
	&&\{n^a\H_1^b(n)\}_{a\in\Z,b\in\N}.
\end{eqnarray}
On the other hand, for any $w\in Y^*$, by a transfer theorem, let
\begin{eqnarray}
\gamma_w:=\mathrm{f.p.}_{n\rightarrow+\infty}\H_w(n),&\{n^a\log^b(n)\}_{a\in\Z,b\in\N}.
\end{eqnarray}

\begin{example}[\cite{Daresbury,JSC}]
In convergence case, one has
	\begin{eqnarray*}
		\Li_{2,1}(z)&=&\zeta(3)+(1-z)\log(1-z)-(1-z)^{-1}-(1-z)\log^2(1-z)/2\cr
		&+&(1-z)^2(-\log^2(1-z)+\log(1-z))/4+\ldots,\cr
		\H_{2,1}(n)&=&\zeta(3)-(\log(n)+1+\gamma)/n+\log(n)/{2n}+\ldots,
	\end{eqnarray*}
	one has $\mathrm{f.p.}_{z\rightarrow1}\Li_{2,1}(z)=\mathrm{f.p.}_{n\rightarrow+\infty}\H_{2,1}(n)=\zeta(2,1)=\zeta(3)$.

In divergence case, one has
	\begin{eqnarray*}
		\Li_{1,2}(z)&=&2-2\zeta(3)-\zeta(2)\log(1-z)-2(1-z)\log(1-z)\cr
		&+&(1-z)\log^2(1-z)+(1-z)^2((\log^2(1-z)-\log(1-z))/2+\ldots,\cr
		\H_{1,2}(n)&=&\zeta(2)\gamma-2\zeta(3)+\zeta(2)\log(n)+(\zeta(2)+2)/{2n}+\ldots,
	\end{eqnarray*}
	since numerically, $\zeta(2)\gamma=0.94948171111498152454556410223170493364000\ldots$, 	then
	\begin{eqnarray*}
		\mathrm{f.p.}_{z\rightarrow1}\Li_{1,2}(z)=2-2\zeta(3)&\neq&
		\mathrm{f.p.}_{n\rightarrow+\infty}\H_{1,2}(n)=\zeta(2)\gamma-2\zeta(3).
	\end{eqnarray*}
\end{example}

Let $Z_{\gamma}$ be the generating series of $\{\gamma_w\}_{w\in Y^*}$. For convenience, we denote
\begin{eqnarray}
	\mathrm{F.P.}_{z\rightarrow1}\L(z)=Z_{\shuffle},
	&\mathrm{F.P.}_{n\rightarrow+\infty}\H(n)=Z_{\stuffle},
	&\mathrm{F.P.}_{n\rightarrow+\infty}\H(n)=Z_{\gamma}.
\end{eqnarray}

\begin{proposition}[\cite{Daresbury,JSC,VJM}]\label{factorization}
	The morphism $\gamma_{\bullet}:(\Q\pol{Y},\stuffle,1_{Y^*})\rightarrow({\mathcal Z}[\gamma],\times,1)$,
	mapping $w$ to $\gamma_w$, is a character. Hence, $\scal{Z_{\gamma}}{1_{Y^*}}=1$ and
	\begin{eqnarray*}
	\Delta_{\stuffle}(Z_{\gamma})=Z_{\gamma}\otimes Z_{\gamma},
	&\Delta_{\stuffle}(\log(Z_{\gamma}))=\log(Z_{\gamma})\otimes 1_{Y^*}+1_{Y^*}\otimes\log(Z_{\gamma}).
	\end{eqnarray*}
	Moreover, 
	\begin{eqnarray*}
		Z_{\gamma}=e^{\gamma y_1}Z_{\stuffle}&\mbox{and}&
		\log(Z_{\gamma})=\sum_{k\ge1}\frac{(-1)^{k-1}}k
		\sum_{u_1,\ldots,u_k\in Y^+}\gamma_{u_1\stuffle\ldots\stuffle u_k}u_1\ldots u_k.
	\end{eqnarray*}
\end{proposition}

\begin{corollary}[\cite{VJM}]\label{pont}
	One has
	$Z_{\gamma}=B(y_1)\pi_Y({Z}_{\shuffle})\iff Z_{\stuffle}=B'(y_1)\pi_Y({Z}_{\shuffle})$,
	where $B(y_1)=e^{\gamma y_1-\sum_{k\ge2}{\zeta(k)}(-y_1)^k/{k}}$
	and $B'(y_1)=e^{-\sum_{k\ge2}{\zeta(k)}(-y_1)^k/{k}}$.
\end{corollary}

The left side of Corollary \ref{pont} leads to the finite parts
of divergent harmonic sums $\{\H_w\}_{w\in y_1Y^*}$ \cite{Daresbury,JSC}.

\begin{example}[generalized Euler's constants, \cite{Daresbury,JSC}]
	\begin{eqnarray*}
		\gamma_{1,1}&=&(\gamma^2-\zeta(2))/2,\cr
		\gamma_{1,1,1}&=&(\gamma^3-3\zeta(2)\gamma+2\zeta(3))/6,\cr
		\gamma_{1,1,1,1}&=&(80\zeta(3)\gamma-60\zeta(2)\gamma^2+6\zeta(2)^2+10\gamma^4)/240,\cr
		\gamma_{1,7}&=&\zeta(7)\gamma+\zeta(3)\zeta(5)-{54}\zeta(2)^4/{175},\cr
		\gamma_{1,1,6}&=&{4}\zeta(2)^3\gamma^2/{35}+(\zeta(2)\zeta(5)+2\zeta(3)\zeta(2)^2/5-4\zeta(7))\gamma\\
		&+&\zeta(6,2)+{19}\zeta(2)^4/{35}+\zeta(2)\zeta(3)^2/2-4\zeta(3)\zeta(5),\cr
		\gamma_{1,1,1,5}&=&3\zeta(6,2)/4-{14}\zeta(3)\zeta(5)/3+3\zeta(2)\zeta(3)^2/4+{809}\zeta(2)^4/{1400}+\zeta(5)\gamma^3/6\\
		&+&(\zeta(3)^2.4-\zeta(2)^3/5)\gamma^2-(2\zeta(7)-3\zeta(2)\zeta(5)/2+\zeta(3)\zeta(2)^2/10)\gamma.
	\end{eqnarray*}
\end{example}

\begin{lemma}[\cite{VJM,CM}]
	Since $\zeta_{\shuffle}(x_0)=\Li_{x_0}(1)=0$ and
	\begin{eqnarray*}
		\mathrm{f.p.}_{z\rightarrow1}\Li_{x_1}(z)=0,&&\{(1-z)^a\log^b((1-z)^{-1})\}_{a\in\Z,b\in\N},\cr
		\mathrm{f.p.}_{n\rightarrow+\infty}\H_{y_1}(n)=0,&&\{n^a\H_1^b(n)\}_{a\in\Z,b\in\N},\cr
		\mathrm{f.p.}_{n\rightarrow+\infty}\H_{y_1}(z)=\gamma,&&\{n^a\log^b(n)\}_{a\in\Z,b\in\N}
	\end{eqnarray*}
	then the polymorphism $\zeta$ in \eqref{zeta} can be extended as characters as follows
	\begin{eqnarray*}
		\zeta_{\shuffle}:(\QX,\shuffle,1_{X^*})\longrightarrow({\mathcal Z},\times,1),& \
		\zeta_{\stuffle}:( {\Q\pol{Y}} , \stuffle, 1_{Y^*})\longrightarrow({\mathcal Z},\times,1)
	\end{eqnarray*}
	according to its products and satisfying, for generators of length (resp. weight) one,
	$\zeta_{\shuffle}(x_0)=\zeta_{\shuffle}(x_1)=\zeta_{\stuffle}(y_1)=0$ and
	$\zeta_{\shuffle}(l)=\zeta_{\stuffle}(\pi_Y(l)=\gamma_{l}=\zeta(l)$, for $l\in\Lyn X\setminus X$.
\end{lemma}

Similarly to the character $\gamma_{\bullet}$, \eqref{zeta} yields immediately
\begin{proposition}[\cite{Daresbury,JSC,VJM}]\label{to0}
	$\Delta_{\stuffle}(Z_{\stuffle})=Z_{\stuffle}\otimes Z_{\stuffle}$
	and $\Delta_{\shuffle}(Z_{\shuffle})=Z_{\shuffle}\otimes Z_{\shuffle}$ with
	$\scal{Z_{\stuffle}}{1_{Y^*}}=\scal{Z_{\shuffle}}{1_{X^*}}=1$,
	$\Delta_{\stuffle}(\log(Z_{\stuffle}))=\log(Z_{\stuffle})\otimes1_{Y^*}+1_{Y^*}\otimes\log(Z_{\stuffle})$,
	$\Delta_{\shuffle}(\log(Z_{\shuffle}))=\log(Z_{\shuffle})\otimes1_{X^*}+1_{X^*}\otimes\log(Z_{\shuffle})$
	and then
	\begin{eqnarray*}
		\log(Z_{\stuffle})&=&\sum_{k\ge1}\frac{(-1)^{k-1}}k\sum_{u_1,\ldots,u_k\in Y^+}
		\zeta_{\stuffle}(u_1\stuffle\ldots\stuffle u_k)u_1\ldots u_k,\\
		\log(Z_{\shuffle})&=&\sum_{k\ge1}\frac{(-1)^{k-1}}k\sum_{u_1,\ldots,u_k\in X^+}
		\zeta_{\shuffle}(u_1\shuffle\ldots\shuffle u_k)u_1\ldots u_k.
	\end{eqnarray*}
\end{proposition}

Similarly, asymptotic behaviors of $\{\Li^-_w\}_{w\in Y^*_0},\{\H^-_w\}_{w\in Y^*_0}$ are analyzed by

\begin{proposition}[\cite{Ngo}]\label{polynomes}
	For any $n\in\mathbb{N}_+$ and $z\in\C,\abs{z}<1$ and $w\in Y_0^*$,
	$\H^-_w$ and $\Li^-_w$ are polynomial, of degree $(w)+|w|$ in
	$\Q[n]$ and $\Z[(1-z)^{-1}]$, respectively.
	
	Hence, for any $w\in Y_0^*$, there is $C^-_w\in\Q$ and $B^-_w\in\N$,
	such that
	\begin{eqnarray*}
		\H^-_{w}(n)\sim_{+\infty}n^{(w)+\abs{w}}C^-_w
		&\mbox{and}&\Li^-_{w}(z)\sim_1(1-z)^{-(w)-\abs{w}}B^-_w,\\
		C^-_w=\prod_{w=uv, v\neq 1_{Y_0^*}}((v)+\abs{v})^{-1}
		&\mbox{and}&B^-_w=((w)+\abs{w})!C^-_w.
	\end{eqnarray*}
\end{proposition}

\begin{example}[\cite{Ngo}]
	$$\begin{array}{@{}rcl@{}}
		\Li^-_{y_1y_1}(z)&=&-(1-z)^{-1}+3(1-z)^{-2}+3(1-z)^{-3}-(1-z)^{-4},\\
		\Li^-_{y_2y_1}(z)&=&(1-z)^{-1}-9(1-z)^{-2}+17(1-z)^{-3}-23(1-z)^{-4}-14(1-z)^{-5},\\
		\Li^-_{y_1y_2}(z)&=&(1-z)^{-1}-7(1-z)^{-2}+9(1-z)^{-3}-13(1-z)^{-4}-18(1-z)^{-5}.\\
		\H^-_{y_2y_1}(n)&=&{n(n^2-1)(10n^2 +15n +2)}/{120},\\
		\H^-_{y_2y_2}(n)&=&{n(10n^5 +12n^4 -10n^3 - 35n^2 +5n+3)}/{180},\\
		\H^-_{y_2y_3}(n)&=&n^2(n-1)(2n^2+n-2)(n+1)^2/8.
	\end{array}$$
	$$\begin{tabular}{@{}cccccc@{}}
		\hline
		$w$& $C^-_w$ & $B^-_w$ & $w$& $C^-_w$ & $B^-_w$\\ 		
		\hline
		$y_n$ & $\frac{1}{n+1}$ & $n!$ & $y_my_n$ & $\frac{1}{(m+1)(m+n+2)}$ & $n!\binom{m+n+2}{m+1}$\\ 		
		$y_0^2$ &$\frac{1}{2}$ & $1$ & $y_2y_2y_3$ & $\frac{1}{280}$ & $12960$\\ 		
		$y_0^n$ &$\frac{1}{(n+1)!}$ & $1$ & $y_2y_{10}y_1^2$ & $\frac{1}{2160}$ & $9686476800$\\ 
		$y_1^2$ & $\frac{1}{8}$ & $3$ & $y_2^2y_4y_3y_{11}$ & $\frac{1}{2612736}$ & $4167611825465088000000$\\ 
		\hline
	\end{tabular}$$
\end{example}

\begin{theorem}[second Abel like theorem, \cite{Ngo}]\label{renormalization2}
	Let us consider
	\begin{eqnarray*}
		\L^-:=\sum_{w\in Y_0^*}\Li^-_ww,&\H^-:=\displaystyle\sum_{w\in Y_0^*}\H^-_ww,&C^-:=\sum_{w\in Y_0^*}C^-_ww.
	\end{eqnarray*}
	Then $\scal{\H^-}{1_{Y_0^*}}=\scal{C^-}{1_{Y_0^*}}=1,\Delta_{\stuffle}(\H^-)=\H^-\otimes\H^-$,
	$\Delta_{\shuffle}(C^-)=C^-\otimes C^-$ and
	\begin{eqnarray*}
		\lim\limits_{z\to1}h^{\odot-1}((1-z)^{-1})\odot\L^-(z)=\lim\limits_{n\to+\infty}g^{\odot-1}(n)\odot\H^-(n)=C^-,
	\end{eqnarray*}
	where $\odot$ denotes the Hadamard product and
	\begin{eqnarray*}
		h(t):=\sum_{w\in Y_0^*}{((w)+\abs w)!}{t^{(w)+\abs w}}w&\mbox{and}&g(t):=(\sum_{y\in Y_0}t^{(y)+1}y)^*.
	\end{eqnarray*}
\end{theorem}

\begin{corollary}[\cite{CM}]\label{Limoins0Hmoins0}
	For any $w\in Y_0^*\setminus\{1_{Y_0^*}\}$, there exists a unique polynomial $p\in(\Z[t],\times,1)$
	of degree $(w)+\abs{w}$ such that, for $n\in\mathbb{N}_+$ and $z\in\C,\abs{z}<1$,
	\begin{eqnarray*}
		\Li^-_w(z)=&\displaystyle\sum_{k=0}^{(w)+\abs{w}}\frac{p_k}{(1-z)^k}
		=\sum_{k=0}^{(w)+\abs{w}}p_ke^{-k\log(1-z)}
		&\in(\Z[e^{-\log(1-z)}],\times,1_{{B}}),\\
		\H^-_w(n)=&\displaystyle\sum_{k=0}^{(w)+\abs{w}}p_k{n+k\choose k}
		=\sum_{k=0}^{(w)+\abs{w}}\frac{p_k}{k!}(n+k)_n
		&\in(\Q[(n+\bullet)_n],\times,1),
	\end{eqnarray*}
	where $(n+\bullet):\N\rightarrow\Q$ mapping $i$ to  $(n+i)_{n}=(n+i)!/n!$
	and $\Q[(n+\bullet)_n]$ denotes the set of polynomials, on $n$, expanded as follows
	\begin{eqnarray*}
		\forall\pi\in\Q[(n+\bullet)_n],
		&\displaystyle\pi=\sum_{i=0}^{d}\pi_k(n+i)_n=\sum_{i=0}^{d}\pi_k\frac{(n+i)!}{n!},
		&\deg(\pi)=d.
	\end{eqnarray*}
\end{corollary}

Note that, in Corollary \ref{Limoins0Hmoins0}, for any $w\in Y_0^*\setminus\{1_{Y_0^*}\}$,
the coefficients $\{\pi_k\}_{k\ge0}$ depend on $w$. One should denote $\pi_k=\pi_k^w$,
for $k\ge0$, and algorithmically, for $w=y_iu$, one can compute the coefficients
$\{\pi_k^{y_iu}\}_{k\ge0}$ as follows \cite{Ngo2}
\begin{enumerate}
	\item If $i=0$ then
	\begin{eqnarray*}
		\pi^{y_0u}_k:=\begin{cases}\pi_{k-1}^u&\mbox{for }k=(u)+|u|+1,\\
			\pi_{k-1}^u\pi_k^u&\mbox{for }1\le k\le(u)+|u|,\\
			-\pi_k^u&\mbox{for }k=0,\\
			0&\mbox{otherwise}.\\
		\end{cases} 
	\end{eqnarray*}
	\item If $i>0$ then 
	\begin{eqnarray*}
		\pi^{y_iu}_k:=\begin{cases}(k-1)\pi_{k-1}^{y_{i-1}u}&\mbox{for }k=(u)+|u|+i+1,\\
			(k-1)\pi_{k-1}^{y_{i-1}u}-k\pi_k^{y_{i-1}u}&\mbox{for }2\le k\le(u)+|u|+i,\\
			-\pi_k^{y_{i-1}u}&\mbox{for }k=1,\\
			0&\mbox{otherwise}.\\
		\end{cases} 
	\end{eqnarray*}
\end{enumerate}

\begin{example}[\cite{Ngo2}]
	We have 
	\begin{eqnarray*}
		\Li^-_{y_0}(z)&=&(1-z)^{-1}-1,\cr
		\Li^-_{y_1}(z)&=&(1-z)^{-2}-\frac{1}{1-z},\cr
		\Li^-_{y_2}(z)&=&(1-z)^{-3}-3(1-z)^{-2}+(1-z)^{-1},\cr
		\Li^-_{y_3}(z)&=&6(1-z)^{-4}-12(1-z)^{-3}+7(1-z)^{-2}-(1-z)^{-1},\cr
		\Li^-_{y_0^2}(z)&=&(1-z)^{-2}-2(1-z)^{-1}+1,\cr
		\Li^-_{y_1y_0}(z)&=&2(1-z)^{-3}-4(1-z)^{-2}+2(1-z)^{--1},\cr
		\Li^-_{y_2y_0}(z)&=&6(1-z)^{-4}-14(1-z)^{-3}+10(1-z)^{-2}-2(1-z)^{--1},\cr
		\H^-_{y_0}(N)&=&\binom{N+1}{1}-\binom{N}{0},\cr
		\H^-_{y_1}(N)&=&\binom{N+2}{2}-\binom{N+1}{1},\cr
		\H^-_{y_2}(N)&=&2\binom{N+3}{3}-3\binom{N+2}{2}+\binom{N+1}{1},\cr
		\H^-_{y_3}(N)&=&6\binom{N+4}{4}-12\binom{N+3}{3}+7\binom{N+2}{2}-\binom{N+1}{1},\cr
		\H^-_{y_0^2}(N)&=&\binom{N+2}{2}-2\binom{N+1}{1}+\binom{N}{0},\cr
		\H^-_{y_1y_0}(N)&=&2 \binom{N+3}{3}-4\binom{N+2}{2}+2\binom{N+1}{1},\cr
		\H^-_{y_2y_0}(N)&=&6\binom{N+4}{4}-14\binom{N+3}{3}+10\binom{N+2}{2}-2\binom{N+1}{1}.
	\end{eqnarray*}
\end{example}

By Corollary \ref{Limoins0Hmoins0}, denoting $\hat p$ the exponential transformed of $p$, one also has
$\Li^-_w(z)=p(e^{-\log(1-z)})$ and $\H^-_w(n)=\hat p((n+\bullet)_n)$, with
\begin{eqnarray}\label{LimoinsHmoins}
	p(t)=\sum_{k=0}^{(w)+\abs{w}}p_kt^{k}\in(\Z[t],\times,1)&\mbox{and}&
	\hat p(t)=\sum_{k=0}^{(w)+\abs{w}}\Frac{p_k}{k!}t^{k}\in(\Q[t],\times,1).
\end{eqnarray}
Let us then associate also $p$ and $\hat p$ with the polynomial $\check p$ obtained as follows
\begin{eqnarray}\label{pcheck}
	\check p(t)=\sum_{k=0}^{(w)+\abs{w}}k!p_kt^{k}=\sum_{k=0}^{(w)+\abs{w}}p_kt^{\shuffle k}\in(\Z[t],\shuffle,1).
\end{eqnarray}

Next, the previous polynomials $p,\hat p$ and $\check p$ given
in \eqref{LimoinsHmoins}--\eqref{pcheck} can be determined explicitly
thanks to Lemma \ref{stability} and to

\begin{proposition}[\cite{Ngo,CM}]\label{explicit}
	\begin{enumerate}
		\item The following morphism is bijective
		\begin{eqnarray*}
		\chi:(\Q[y_1^*],\stuffle,1_{Y^*})\longrightarrow(\Q[(n+\bullet)_n],\times,1),&S\longmapsto\H_S.
		\end{eqnarray*}
		
		\item For any $w=y_{s_1},\ldots y_{s_r}\in Y_0^*$,
		there exists a unique polynomial $R_w$ belonging to $(\Z[x_1^*],\shuffle,1_{X^*})$
		of degree $(w)+\abs{w}$, such that
		\begin{eqnarray*}
			\Li_{R_w}(z)=\Li^-_w(z)=&p(e^{-\log(1-z)})&\in(\Z[e^{-\log(1-z)}],\times,1_{\Omega}),\cr
			\H_{\pi_Y(R_w)}(n)=\H^-_w(n)=&\hat p((n+\bullet)_n)&\in(\Q[(n+\bullet)_n],\times,1).
		\end{eqnarray*}
		In particular, via the extension by linearity,
		of $R_{\bullet}$ over $\Q\pol{Y_0}$ and via Theorem \ref{structure2},
		$\{\Li_{R_{y_k}}\}_{k\ge0}$ is linear independent in
		$\Q\{\Li_{R_w}\}_{w\in Y_0^*}$ and, for any $k,l\in\N$,
		$\Li_{R_{y_k}\shuffle R_{y_l}}=\Li_{R_{y_k}}\Li_{R_{y_l}}
		=\Li^-_{y_k}\Li^-_{y_l}=\Li^-_{y_k\top y_l}=\Li_{R_{y_k\top y_l}}$.
		
		\item For any $w\in Y_0^*$, there exists a unique polynomial
		$R_w\in(\Z[x_1^*],\shuffle,1_{X^*})$ of degree $(w)+\abs{w}$ such that
		$\check p(x_1^*)=R_w$.
		
		\item More explicitly, for any $w=y_{s_1},\ldots y_{s_r}\in Y_0^*$,
		there exists a unique polynomial $R_w$ belonging to $(\Z[x_1^*],\shuffle,1_{X^*})$
		of degree $(w)+\abs{w}$, given by
		\begin{eqnarray*}
			R_{y_{s_1}\ldots y_{s_r}}
			=\sum_{k_1=0}^{s_1}\sum_{k_2=0}^{s_1+s_2-k_1}\ldots
			\sum_{k_r=0}^{(s_1+\ldots+s_r)-\atop(k_1+\ldots+k_{r-1})}
			\binom{s_1}{k_1}\binom{s_1+s_2-k_1}{k_2}\ldots\cr
			\binom{s_1+\ldots+s_r-k_1-\ldots-k_{r-1}}{k_r}
			\rho_{k_1}\shuffle\ldots\shuffle\rho_{k_r},
		\end{eqnarray*}
		where, for any $\forall i=1,\ldots,r$, if $k_i=0$ then $\rho_{k_i}=x_1^*-1_{X^*}$ else
		($S_2(k,j)$'s denote the Stirling numbers of second kind)
		\begin{eqnarray*}
			\rho_{k_i}=\sum_{j=1}^{k_i}S_2({k_i},j)(j!)^2\sum_{l=0}^j
			\frac{(-1)^{l}}{l!}\frac{(x_1^*)^{\shuffle(j-l+1)}}{(j-l)!}.
		\end{eqnarray*}
	\end{enumerate}
\end{proposition}

By Proposition \ref{explicit} and Lemma \ref{stability}, in particular, the bijectivity of the restriction
$\Li_{\bullet}:(\Z[x_1^*],\shuffle,1_{X^*})\longrightarrow(\Z[e^{-\log(1-z)}],.,1_{\Omega})$,
one deduces

\begin{corollary}[\cite{CM}]\label{letter}
	\begin{enumerate}
		\item $R_{\bullet}:({\Z}\langle{Y}\rangle,\top,1_{Y^*})\longrightarrow(\Z[x_1^*],\shuffle,1_{X^*})$ is bijective.
		
		\item Denoting $S_1(k,i)$ and $S_2(k,j)$ the Stirling numbers of first and second kind,
		one has, for any $k\ge1$, $\Li_{-k}=\Li_{R_{y_k}}$, where $R_{y_k}=x_1^*\shuffle R'_{y_k}$ and
		\begin{eqnarray*}
			R'_{y_k}
			=\sum_{i=0}^{k}i!S_2(k,i)(x_1^*-1)^{\shuffle i}
			=\sum_{i=0}^{k}\sum_{j=0}^{i}i!S_2(k,i){i\choose j}(-1)^{i-j}(x_1^*)^{\shuffle j}.
		\end{eqnarray*}
		($R'_{\bullet}$ is extended over $\Z\langle{Y}\rangle$). Conversely,
		with $R_{y_0}=x_1^*-1_{X^*}$, one has
		\begin{eqnarray*}
			\forall k\ge1,&\displaystyle(kx_1)^*=1_{X^*}+R_{y_0}+\sum_{j=2}^k\frac{S_1(k,j)}{(k-1)!}R_{y_{j+1}}.
		\end{eqnarray*}
		
		\item For any $k,l\ge1$, one has $\Li_{R_{y_k}}\odot\Li_{R_{y_l}}=\Li_S$, where
		\begin{eqnarray*}
			S=x_1^*\shuffle{R'_{y_k\stuffle y_l}}=(1_{X^*}+R_{y_0})\shuffle(R'_{y_{k+l}}+R'_{y_k\shuffle y_l}).
		\end{eqnarray*}
	\end{enumerate}
\end{corollary}

\begin{theorem}[extension of $\H_{\bullet}$, \cite{QPL}]\label{Indexation2}
For any $r\ge1,t\in\C,\abs{t}<1$, one has
\begin{eqnarray*}
\H_{(t^ry_r)^*}=\sum_{k\ge0}\H_{y_r^k}t^{kr}
=\exp\biggl(\sum_{k\ge1}\H_{y_{kr}}\frac{(-t^r)^{k-1}}k\biggr).
\end{eqnarray*}
Moreover, for $\abs{a_s}<1,\abs{b_s}<1$ and $\abs{a_s+b_s}<1$,
\begin{eqnarray*}
\H_{(\sum_{s\ge1}(a_s+b_s)y_s+\sum_{r,s\ge1}a_sb_ry_{s+r})^*}
=\H_{(\sum_{s\ge1}a_sy_s)^*}\H_{(\sum_{s\ge1}b_sy_s)^*}.
\end{eqnarray*}
Hence, $\H_{(a_sy_s+a_ry_r+a_sa_ry_{s+r})^*}=\H_{(a_sy_s)^*}\H_{(a_ry_r)^*}$
and $\H_{(-a_s^2y_{2s})^*}=\H_{(a_sy_s)^*}\H_{(-a_sy_s)^*}$.
\end{theorem}

\begin{proof}
For $t\in\C,\abs{t}<1$, since $\Li_{(tx_1)^*}$ is well defined then
so is the arithmetic function, expressed via Newton-Girard formula, for $n\ge0$, by
\begin{eqnarray*}
\H_{(ty_1)^*}(n)
=\sum_{k\ge0}\H_{y_1^k}(n)t^{k}
=\exp\Bigl(-\sum_{k\ge1}\H_{y_k}(n)\frac{(-t)^{k}}k\Bigr)
=\prod_{l=1}^n\Bigl(1+\frac{t}{l}\Bigr).
\end{eqnarray*}
Similarly, $\H_{(t^ry_r)^*}$ ($r\ge2$) can be expressed, via Newton-Girard formula
\cite{lascoux} once again and via Adam's transform (see \cite{JSC}), by
\begin{eqnarray*}
\H_{(t^ry_r)^*}(n)
=\sum_{k\ge0}\H_{y_r^k}(n)t^{kr}
=\exp\Bigl(-\sum_{k\ge1}\H_{y_{kr}}(n)\frac{(-t^r)^{k}}k\Bigr)
=\prod_{l=1}^N\Bigl(1-\frac{(-t^r)}{l^r}\Bigr).
\end{eqnarray*}
Since $\absv{\H_{y_r}}_{\infty}\le\zeta(r)$ then $-\sum_{k\ge1}{\H_{kr}{(-t^r)^{k}}}/{k}$ is termwise dominated by $\absv{f_r}_{\infty}$,
where $f_r(t)=-\sum_{k\ge1}\zeta(kr){(-t^r)^k}/k$, and then $\H_{(t^ry_r)^*}$ by $e^{f_r}$ (see also proposition \ref{reg_alg} bellow).
By Corollary \ref{Kleene}, it follows then last results.
\end{proof}

\begin{corollary}
For any $r\ge1$
\begin{eqnarray*}
y_r^*=\exp_{\stuffle}\Bigl(\sum_{k\ge1}y_{kr}\frac{(-1)^{k-1}}k\Bigr).
\end{eqnarray*}
Hence, for any $k\ge0$, one has
\begin{eqnarray*}
y_r^n=\frac{(-1)^n}{n!}\sum_{s_1,\ldots,s_n>0\atop s_1+\ldots+ns_n=n}
\frac{(-y_r)^{\stuffle s_1}}{1^{s_1}}\stuffle\ldots\stuffle\frac{(-y_{nr})^{\stuffle s_n}}{n^{s_n}}
\end{eqnarray*}
and, for any $r,s\ge1$, one also has
\begin{eqnarray*}
y_r^*\stuffle y_r^*&=&\sum_{k=0}^r{r+s-k\choose s}{s\choose k}y_{r+s-k}.
\end{eqnarray*}
\end{corollary}

By Weierstrass factorization \cite{Dieudonne} and Newton-Girard identity, one obtains
\begin{eqnarray}\label{Weierstrass}
\frac1{\Gamma(z+1)}
=e^{\gamma z}\prod_{n\ge1}\Bigl(1+\frac{z}{n}\Bigr)e^{-\frac{z}{n}}
=\exp\Bigl(\gamma z-\sum_{k\ge2}\zeta(k)\frac{(-z)^k}{k}\Bigr).
\end{eqnarray}
From the estimations from above of the proof of Theorem \ref{Indexation2}, it follows then

\begin{corollary}
For any $r\ge2$, one has
\begin{eqnarray*}
\sum_{k\ge0}\zeta(\underbrace{{r,\ldots,r}}_{k{\tt\ times}})t^{kr}
=\exp\Bigl(-\sum_{k\ge1}\zeta(kr)\frac{(-t^r)^{k}}k\Bigr)
=\prod_{n\ge1}\Bigl(1-\frac{(-t^r)}{n^r}\Bigr).
\end{eqnarray*}
\end{corollary}

Using the following functional equation and Euler's complement formula, {\it i.e.}
\begin{eqnarray}\label{complement_formula}
\Gamma(1+z)=z\Gamma(z)&\mbox{and}&\Gamma(z)\Gamma(1-z)=\frac{\pi}{\sin(z\pi)},
\end{eqnarray}
and also introducing the {\it partial beta function} defined by
\begin{eqnarray}\label{beta}
\forall a,b,z\in\C,&\Re a>0,\Re b>0,\abs{z}<1,&\mathrm{B}(z;a,b):=\int_0^zdt\;t^{a-1}(1-t)^{b-1}
\end{eqnarray}
and then, classically, $\mathrm{B}(a,b):=\mathrm{B}(1;a,b)={\Gamma(a)\Gamma(b)}/{\Gamma(a+b)}$.
One also has \cite{CASC2018,CM}
\begin{eqnarray}
\Li_{x_0[(ax_0)^*\shuffle((1-b)x_1)^*]}(z)=\Li_{x_1[((a-1)x_0)^*\shuffle(-bx_1)^*]}(z)=\mathrm{B}(z;a,b).
\end{eqnarray}
Since $\Li_{(tx_1)^*}(z)=(1-z)^{-t}$, for $\abs{t}<1$, then by Lemma \ref{rat_ext} and Proposition \ref{to0}, the characters
$\zeta_{\shuffle}$ and $\gamma_{\bullet}$ can be extended algebraically as follows

\begin{proposition}[\cite{QPL,CASC2018,CM}]\label{reg_alg}
The characters $\zeta_{\shuffle}$ and $\gamma_{\bullet}$ are extended algebraically as follows:
\begin{eqnarray*}
\zeta_{\shuffle}:(\CX\shuffle\ncs{\C_{\mathrm{exc}}^{\mathrm{rat}}}{X},\shuffle,1_{X^*})&\longrightarrow&(\C,.,1),\\
\forall t\in\C,\abs{t}<1,(tx_0)^*,(tx_1)^*&\longmapsto &{1_{\C}},\\
\gamma_{\bullet}:(\CY\stuffle\ncs{\C_{\mathrm{exc}}^{\mathrm{rat}}}{Y},\stuffle,1_{Y^*})&\longrightarrow&(\C,.,1),\\
\forall t\in\C,\abs{t}<1,(ty_1)^*&\longmapsto&\exp\Bigl(\gamma t-\sum\limits_{k\ge2}\zeta(k)\frac{(-t)^{k}}k\Bigr),\\
\forall t\in\C,\abs{t}<1,\forall r>1,(t^ry_r)^*&\longmapsto&\exp\Bigl(-\sum\limits_{k\ge1}\zeta(kr)\frac{(-t^r)^{k}}k\Bigr).
\end{eqnarray*}
\end{proposition}

\begin{example}[\cite{Ngo2,CM}]
	$$\begin{array}{l}
		\Li_{-1,-1}=-\Li_{x_1^*}+5\Li_{(2x_1)^*}-7\Li_{(3x_1)^*}+3\Li_{(4x_1)^*},\cr
		\Li_{-2,-1}=\Li_{x_1^*}-11\Li_{(2x_1)^*}+31\Li_{(3x_1)^*}-33\Li_{(4x_1)^*}+12\Li_{(5x_1)^*},\cr
		\Li_{-1,-2}=\Li_{x_1^*}-9\Li_{(2x_1)^*}+23\Li_{(3x_1)^*}-23\Li_{(4x_1)^*}+8\Li_{(5x_1)^*},\cr
		\H_{-1,-1}=-\H_{\pi_Yx_1^*}+5\H_{\pi_Y(2x_1)^*}-7\H_{\pi_Y(3x_1)^*}+3\H_{\pi_Y(4x_1)^*},\cr
		\H_{-2,-1}=\H_{\pi_Yx_1^*}-11\H_{\pi_Y(2x_1)^*}+31\H_{\pi_Y(3x_1)^*}-33\H_{\pi_Y(4x_1)^*}+12\H_{\pi_Y(5x_1)^*},\cr
		\H_{-1,-2}=\H_{\pi_Yx_1^*}-9\H_{\pi_Y(2x_1)^*}+23\H_{\pi_Y(3x_1)^*}-23\H_{\pi_Y(4x_1)^*}+8\H_{\pi_Y(5x_1^*)}.
	\end{array}$$
	Therefore,
	$\begin{array}{lll}
		\zeta_{\shuffle}(-1,-1)=0,&\zeta_{\shuffle}(-2,-1)=1,&\zeta_{\shuffle}(-1,-2)=0,
	\end{array}$
	and
	$$\begin{array}{l}
		\gamma_{-1,-1}=-\Gamma^{-1}(2)+5\Gamma^{-1}(3)-7\Gamma^{-1}(4)+3\Gamma^{-1}(5)={11}/{24},\cr
		\gamma_{-2,-1}=\Gamma^{-1}(2)-11\Gamma^{-1}(3)+31 \Gamma^{-1}(4)-33\Gamma^{-1}(5)+12\Gamma^{-1}(6)=-{73}/{120},\cr
		\gamma_{-1,-2}=\Gamma^{-1}(2)-9\Gamma^{-1}(3)+23 \Gamma^{-1}(4)-23\Gamma^{-1}(5)+8\Gamma^{-1}(6)=-{67}/{120}.
	\end{array}$$
\end{example}

\begin{example}\label{example1}\cite{words03,orlando}
By Corollary \ref{Kleene}, one has $(-t^2y_2)^*=(ty_1)^*\stuffle(-ty_1)^*$ and then, by Proposition \ref{reg_alg},
$\gamma_{(ty_1)^*}\gamma_{(-ty_1)^*}=\gamma_{(-t^2y_2)^*}=\zeta((-t^2x_0x_1)^*)$ and then,
by \eqref{Weierstrass}--\eqref{complement_formula}, one deduces
\begin{eqnarray*}
\exp\Bigl(-\sum_{k\ge2}\zeta(2k)\frac{t^{2k}}{k}\Bigr)
=\sum_{k\ge1}(-1)^k\frac{(t\pi)^{2k}}{(2k+1)!}
=\sum_{k\ge2}\zeta(\overbrace{{2,\ldots,2}}^{k{\tt\ times}})(-1)^kt^{2k}.
\end{eqnarray*}
On the one hand, identifying the coefficients of $t^{2k}$ on the second identity, one gets
\begin{eqnarray*}
\frac{\zeta(\overbrace{{2,\ldots,2}}^{k{\tt\ times}})}{\pi^{2k}}=\frac{1}{(2k+1)!}\in\Q.
\end{eqnarray*}
On the other hand, taking the logarithms and considering Taylor expansions on the first identity,
one also obtains\footnote{Note that Euler gave another explicit formula using Bernoulli numbers.}
\begin{eqnarray*}
\frac{{\zeta(2k)}}{\pi^{2k}}=k\sum_{l=1}^k\frac{(-1)^{k+l}}{l}
\sum_{n_1,\ldots,n_l\ge1\atop n_1+\ldots+n_l=k}\prod_{i=1}^l\frac{1}{{\Gamma(2n_i+2)}}\in\Q.
\end{eqnarray*}
For example,
\begin{eqnarray*}
\frac{\zeta(2)}{\pi^2}=&\displaystyle1\frac{(-1)^{1+1}}{1}\frac{1}{\Gamma(4)}&=\frac{1}{6},\cr
\frac{\zeta(4)}{\pi^4}=&\displaystyle2\Bigl(\frac{(-1)^{2+1}}{1}\frac{1}{\Gamma(6)}+\frac{(-1)^{2+2}}{2}\frac{1}{\Gamma(4)\Gamma(4)}\Bigr)&=\frac{1}{90},\cr
\frac{\zeta(6)}{\pi^6}=&\displaystyle3\sum_{l=1}^3\frac{(-1)^{3+l}}{l}\sum_{n_1,\ldots,n_l\ge1\atop n_1+\ldots+n_l=3}\prod_{i=1}^l\frac{1}{{\Gamma(2n_i+2)}}&=\frac{1}{945},\cr
\frac{\zeta(8)}{\pi^8}=&\displaystyle4\sum_{l=1}^4\frac{(-1)^{4+l}}{l} \sum_{n_1,\ldots,n_l\ge1\atop n_1+\ldots+n_l=4}\prod_{i=1}^l\frac{1}{{\Gamma(2n_i+2)}}&=\frac{1}{9450},\cr
\frac{\zeta(10)}{\pi^{10}}=&\displaystyle5\sum_{l=1}^5\frac{(-1)^{5+l}}{l}\sum_{n_1,\ldots,n_l\ge1\atop n_1+\ldots+n_l=5}\prod_{i=1}^l\frac{1}{{\Gamma(2n_i+2)}}&=\frac{1}{93555}.
\end{eqnarray*}

Similarly, by Proposition \ref{reg_alg} and the identity on rational series in Example \ref{quiplaitbis}, one has
$\gamma_{(-t^4y_4)^*}=\gamma_{(t^2y_2)^*}\gamma_{(-t^2y_2)^*}$,
or equivalently,
\begin{eqnarray*}
\exp\Bigl(-\sum_{k\ge1}\zeta(4k)\frac{t^{4k}}{k}\Bigr)
=\frac{\sin(\mathrm{i}t\pi)}{\mathrm{i}t\pi}\frac{\sin(t\pi)}{t\pi}
=\sum_{k\ge1}\frac{2(-4t\pi)^{4k}}{(4k+2)!}.
\end{eqnarray*}
Since $\gamma_{(-t^4y_4)^*}=\zeta((-t^4y_4)^*),\gamma_{(-t^2y_2)^*}=\zeta((-t^2y_2)^*),\gamma_{(t^2y_2)^*}=\zeta((t^2y_2)^*)$
then, by Example \ref{quiplait}, one has $\zeta((-4t^4x_0^2x_1^2)^*)=\zeta((-t^2x_0x_1)^*)\zeta((t^2x_0x_1)^*))$ and then
\begin{eqnarray*}
\sum_{k\ge0}\zeta(\overbrace{{3,1,\ldots,3,1}}^{k{\tt\ times}})(-4t^4)^k
=\sum_{k\ge0}\zeta(\overbrace{{4,\ldots,4}}^{k{\tt\ times}})(-t\pi)^{4k}
=\sum_{k\ge1}\frac{2(-4t\pi)^{4k}}{(4k+2)!}.
\end{eqnarray*}
Thus, by identification the coefficients of $t^{4k}$ and previously, we obtain
\begin{eqnarray*}
\frac{\zeta(\overbrace{{4,\ldots,4}}^{k{\tt\ times}})}{4^k\pi^{4k}}
=\frac{\zeta(\overbrace{{3,1,\ldots,3,1}}^{k{\tt\ times}})}{\pi^{4k}}
=\frac{2}{(4k+2)!}\in\Q.
\end{eqnarray*}
\end{example}
By Corollary \ref{Limoins0Hmoins0}, $p(1)$ and $\hat p(1)$ (see \eqref{LimoinsHmoins}) represent the following finite parts
\begin{lemma}[\cite{CASC2018,CM}]\label{reg_anal}
	\begin{enumerate}
		\item For any $w\in Y^*$, let $R_w$ be explicitly determined as in Proposition \ref{explicit}.
		There exists a unique polynomial $p\in\Z[t]$ of valuation $1$
		and of degree $(w)+\abs{w}$ such that $R_w=\check p(x_1^*)$ and
		\begin{eqnarray*}
			\mathrm{f.p.}_{z\rightarrow1}\Li_{R_w}(z)=p(1)\in\Z,
			&&\{(1-z)^{a}\log^b((1-z)^{-1})\}_{a\in\Z,b\in\N},\cr
			\mathrm{f.p.}_{n\rightarrow+\infty}\H_{\pi_Y(R_w)}(n)=\hat p(1)\in\Q,
			&&\{n^{a}\log^b(n)\}_{a\in\Z,b\in\N}.
		\end{eqnarray*}			
		\item For any $Q\in(\Z[x_0^*,(-x_0)^*,x_1^*],\shuffle,1_{X^*})/(x_0^*\shuffle {x_{1}}^*-{x_1}^*+1)$,
		let $\P_{Q}(z):=e^{-\log(1-z)}\allowbreak\Li_{Q}(z)$.
		Then $\P_{Q}=\Li_{x_1^*\shuffle Q}$ and then $\Li_{Q}$ and $\P_{Q}$ belong to $\Z[z,z^{-1},e^{-\log(1-z)}]$.
		By Lemma \ref{stability}, the converse holds. Moreover,
		\begin{eqnarray*}
			\mathrm{f.p.}_{z\rightarrow1}\P_{Q}(z)
			=\mathrm{f.p.}_{z\rightarrow1}\Li_{Q}(z)\in\Z,\ 
			\{(1-z)^{a}\log^b((1-z)^{-1})\}_{a\in\Z,b\in\N},\cr
			\mathrm{f.p.}_{n\rightarrow+\infty}\scal{\P_Q}{z^n}
			=\mathrm{f.p.}_{n\rightarrow+\infty}\H_{\pi_Y(Q)}(n)\in\Q,\ 
			\{n^{a}\log^b(n)\}_{a\in\Z,b\in\N}.
		\end{eqnarray*}
	\end{enumerate}
\end{lemma}

As determined in Proposition \ref{polynomes}, $B^-_{\bullet}$ and $C^-_{\bullet}$
do not realize characters for $(\QX,\shuffle,1_{X^*})$ and $(\Q\pol{Y},\stuffle,1_{Y^*})$,
respectively \cite{Ngo}. Hence, instead of to regularize the divergent sums
$\zeta_{\shuffle}(R_w)$ and $\zeta_{\gamma}(\pi_Y(R_w))$ by $B^-_w$ and $C^-_w$,
one can use, respectively, $p(1)$ and $\hat p(1)$ (depending on $w$)
as showed Theorem \ref{fin} below which is a consequence of Lemma \ref{reg_anal},
Propositions \ref{explicit}, \ref{reg_alg} and Corollary \ref{letter}.

Now, let $\Upsilon\in\serie{\Q[(n)_{\bullet}]}{Y}$ and $\Lambda\in\serie{\Q[e^{-\log(1-z)}][\log(z)]}{X}$
be the noncommutative generating series of, respectively, $\{\H_{\pi_Y(R_w)}\}_{w\in Y^*}$ and
$\{\Li_{R_{\pi_Y(w)}}\}_{w\in X^*}$ (with $\scal{\Lambda(z)}{x_0}=\log(z)$) and similarly, let
$Z^-_{\gamma}\in\serie{\Q}{Y}$ and $Z^-_{\shuffle}\in\serie{\Z}{X}$ be the generating series of,
respectively, $\{\gamma_{\pi_Y(R_w)}\}_{w\in Y^*}$ and $\{\zeta_{\shuffle}(R_{\pi_Y(w)})\}_{w\in X^*}$:
\begin{eqnarray}\label{definition}
	\Upsilon:=\sum_{w\in Y^*}\H_{\pi_Y(R_w)}w
	&\mbox{and}&
	\Lambda:=\sum_{w\in X^*}\Li_{R_{\pi_Y(w)}}w,\\
	Z^-_{\gamma}:=\sum_{w\in Y^*}\gamma_{\pi_Y(R_w)}w
	&\mbox{and}&
	Z^-_{\shuffle}:=\sum_{w\in X^*}\zeta_{\shuffle}(R_{\pi_Y(w)})w.
\end{eqnarray}
In particular, $\scal{Z^-_{\gamma}}{y_1}=-1/2$ and
$\scal{Z^-_{\shuffle}}{x_1}=\scal{Z^-_{\shuffle}}{x_0}=0$.

Using the bijectivity of $R_{\bullet}:({\C[x_0]}\langle{Y_0}\rangle,{\top},1_{Y_0^*})\longrightarrow(\C[x_0][x_1^*],\shuffle,1_{X^*})$ and
the expression of ${\mathcal D}_{X}$ (resp. ${\mathcal D}_{\stuffle}$), given in \eqref{diagonalX} (resp. \eqref{diagonalY}), one also has

\begin{lemma}[\cite{CASC2018,CM}]
	\begin{enumerate}
		\item Let $\hat{\pi}_Y$ be the morphism of algebras defined, over an algebraic basis, by, for any
		$l\in\Lyn X-\{x_0\},\hat{\pi}_YS_l=\pi_YS_l$ and $\hat{\pi}_Y(x_0)=x_0$ (such that
		$\Li_{R_{\hat{\pi}_Yx_0}}(z)=\log(z)$ and then $\zeta(R_{\hat{\pi}_Yx_0})=0$). Then
		\begin{eqnarray*}
			\Upsilon=((\H_{\bullet}\circ{\pi}_Y\circ R_{\bullet})\otimes\mathrm{Id}){{\mathcal D}_Y}&\mbox{and}&
			\Lambda=((\Li_{\bullet}\circ R_{\bullet}\circ\hat{\pi}_Y)\otimes\mathrm{Id}){{\mathcal D}_X},\\
			Z^-_{\gamma}=((\gamma_{\bullet}\circ{\pi}_Y\circ R_{\bullet})\otimes\mathrm{Id}){{\mathcal D}_Y}&\mbox{and}&
			Z^-_{\shuffle}=((\zeta_{\shuffle}\circ R_{\bullet}\circ\hat{\pi}_Y)\otimes\mathrm{Id}){{\mathcal D}_X}.
		\end{eqnarray*}
		
		\item One has $Z^-_{\gamma}=\mathrm{F.P.}_{n\rightarrow+\infty}\Upsilon(n)$ and	$Z^-_{\shuffle}=\mathrm{F.P.}_{z\rightarrow1}\Lambda(z)$, \textit{i.e.}
		\begin{eqnarray*}
			\forall u\in X^*,&\mathrm{f.p.}_{z\rightarrow1}\scal{\Lambda(z)}{u}=\scal{Z^-_{\shuffle}}{u},&\{(1-z)^{a}\log^b((1-z)^{-1})\}_{a\in\Z,b\in\N},\\
			\forall v\in Y^*,&\mathrm{f.p.}_{n\rightarrow+\infty}\scal{\Upsilon(n)}{v}=\scal{Z^-_{\gamma}}{v},&\{n^{a}\log^b(n)\}_{a\in\Z,b\in\N},
		\end{eqnarray*}
	\end{enumerate}
\end{lemma}

Hence, by Propositions \ref{explicit} and \ref{reg_alg}, Lemmas \ref{rat_ext} and \ref{reg_anal}, one deduces then

\begin{theorem}[\cite{CASC2018,CM}]\label{fin}
	\begin{enumerate}
		\item Associating $l\in\Lyn Y$ with $(s_1,\ldots,s_r)\in\N_{\ge1}^r$, there exists
		a unique $p\in\Z[t]$ of valuation $1$ and of degree $(l)+\abs{l}$ such that 
		$$\begin{array}{@{}rcll@{}}
			\check p(x_1^*)&=&R_l&\in(\Z[x_1^*],\shuffle,1_{X^*}),\cr
			p(e^{-\log(1-z)})&=&\Li_{R_l}(z)&\in(\Z[e^{-\log(1-z)}],\times,1_{\Omega}),\cr
			\hat p((n+\bullet)_n)&=&\H_{\pi_Y(R_l)}(n)&\in(\Q[(n+\bullet)_n],\times,1),\cr
			\zeta(-s_1,\ldots,-s_r)=p(1)&=&\zeta_{\shuffle}(R_l)&\in(\Z,\times,1),\cr
			\gamma_{-s_1,\ldots,-s_r}=\hat p(1)&=&\gamma_{\pi_Y(R_l)}&\in(\Q,\times,1).
		\end{array}$$
		where $\hat p\in\Q[t]$ is the exponential transformed of $p$ and $p$
		is obtained as the exponential transformed of $\check p\in\Z[t]$.
		
		\item $\Delta_{\stuffle}(Z^-_{\gamma})=Z^-_{\gamma}\otimes Z^-_{\gamma}, 
		\Delta_{\shuffle}(Z^-_{\shuffle})=Z^-_{\shuffle}\otimes Z^-_{\shuffle},
		\scal{Z^-_{\gamma}}{1_{Y^*}}=\scal{Z^-_{\shuffle}}{1_{X^*}}=1$.
		$\Delta_{\stuffle}(\Upsilon)=\Upsilon\otimes\Upsilon,
		\Delta_{\shuffle}(\Lambda)=\Lambda\otimes\Lambda
		\scal{\Upsilon}{1_{Y^*}}=\scal{\Lambda}{1_{X^*}}=1$.
		Moreover,
		\begin{eqnarray*}
			Z^-_{\gamma}=\prod_{l\in\Lyn Y}^{\searrow}e^{\gamma_{\pi_Y(R_{\Sigma_l})}\Pi_l}	
			&\text{and}&
			Z^-_{\shuffle}=\prod_{l\in\Lyn X}^{\searrow}e^{\zeta_{\shuffle}(\pi_Y(S_l))P_l},\cr
			\Upsilon=\prod_{l\in\Lyn Y}^{\searrow}e^{\H_{\pi_Y(R_{\Sigma_l})}\Pi_l}
			&\text{and}&
			\Lambda=\prod_{l\in\Lyn X}^{\searrow}e^{\Li_{R_{\pi_Y(S_l)}}P_l}.			
		\end{eqnarray*}
		
		\item Under the action of $\calG$, as for $\L$ \cite{FPSAC98}, for $g\in\calG$, there is a substitution $\sigma_g$
		and $C\in\LXX$ such that $\Lambda(g(z))=\sigma_g(\Lambda(z))e^C$ and $\Lambda(z)\sim_0e^{x_0\log(z)}$.
		\end{enumerate}
\end{theorem}

\subsection{Relations among polyzetas}
In the right side of Corollary \ref{pont}, identifying of local coordinates one gets polynomials homogeneous
in weight yielding the graded kernel of $\zeta$ \eqref{zeta}, \textit{i.e.} algebraic relations on the local
coordinates of $Z_{\stuffle}$ (resp. $Z_{\shuffle}$), \textit{i.e.} the values
$\{\zeta(\Sigma_l)\}_{l\in\Lyn Y\setminus\{y_1\}}$ (resp. $\{\zeta(S_l)\}_{l\in\Lyn X\setminus X}$). For that,
let us use the following algorithm to identify locale coordinates, via identities in Corollary \ref{pont}.

The identification of local coordinates in $Z_{\gamma}=B(y_1)\pi_YZ_{\shuffle}$, leads to
\begin{itemize}
	\item A family of algebraic generators $\calZ_{irr}^{\infty}({\mathcal X})$ of ${\calZ}$ (see Table \ref{E2} in Appendix)
	and their inverse image by a section\footnote{By \eqref{zeta}, the polymorphism $\zeta$ from $\Q[\Lyn(X)\setminus X]$
	(resp. $\Q[\Lyn(Y)\setminus\{y_1\}]$) to $\calZ$ is surjective.} of $\zeta$ (see Table \ref{E4} in Appendix)
	\begin{eqnarray}
		\calZ_{irr}^{\le 2}(\calX)\subset\cdots\subset\calZ_{irr}^{\le p}(\calX)\subset\cdots\subset
		\calZ_{irr}^{\infty}(\calX)=\displaystyle\bigcup_{p\ge2}\calZ_{irr}^{\le p}(\calX),\\
		\calL_{irr}^{\le 2}(\calX)\subset\cdots\subset\calL_{irr}^{\le p}(\calX)\subset\cdots\subset
		\calL_{irr}^{\infty}(\calX)=\displaystyle\bigcup_{p\ge2}\calL_{irr}^{\le p}(\calX),
	\end{eqnarray}
	such that the restriction $\zeta:\Q[\calL_{irr}^{\infty}({\mathcal X})]\longrightarrow\calZ$ is bijective.
	Moreover, one has $\calZ=\Q[{ \calZ_{irr}^{\infty}(\calX)}]=\Q[\{\zeta(p)\}_{p\in \calL_{irr}^{\infty}(\calX)}]$.

	\item The polynomials $\{Q_l\}_{l\in\Lyn{\calX},\atop l\notin\{y_1,x_0,x_1\}}$
	homogenous in weight ($=(l)$) generate the following sub ideals of $\ker\zeta$ (see Table \ref{E3} in Appendix)
	\begin{eqnarray}
		{\mathcal R}_Y&:=&(\mathrm{span}_{\Q}\{Q_l\}_{l\in\Lyn Y\setminus\{y_1\}},\stuffle,1_{Y^*}),\\
		{\mathcal R}_X&:=&(\mathrm{span}_{\Q}\{Q_l\}_{l\in\Lyn X\setminus X},\shuffle,1_{X^*})
	\end{eqnarray}
	such that the following assertions are equivalent (see Appendix)
	\begin{itemize}
		\item ${Q_l}=0$,
		\item ${\Sigma_l\rightarrow\Sigma_l}$ (resp. ${S_l\rightarrow S_l}$), 
		\item ${\Sigma_l}\in{\calL_{irr}^{\infty}(Y)}$ (resp. ${S_l}\in{\calL_{irr}^{\infty}(X)}$).
	\end{itemize}
	
	Any ${Q_l}$, not equal $0$, is led by ${\Sigma_l}$ (resp. ${S_l}$), as being transcendent over the
	sub algebra $\Q[{\calL_{irr}^{\infty}(\calX)}]$. Moreover, as being homogenous polynomial of weight,
	${\Sigma_l}$ (resp. ${S_l}$) belongs to $\Q[\calL_{irr}^{\le p}(\calX))$ and
	$\Sigma_l\rightarrow\Upsilon_l$ (resp. $S_l\rightarrow U_l$).
	In other terms, $\Sigma_l=Q_l+\Upsilon_l$ (resp. $S_l={Q_l}+U_l$), \textit{i.e.}
$\mathrm{span}_{\Q}\{S_l\}_{l\in\Lyn X\setminus X}={\mathcal R}_{X}\oplus\mathrm{span}_{\Q}{\calL_{irr}^{\infty}(X)}$
(resp. $\mathrm{span}_{\Q}\{\Sigma_l\}_{l\in\Lyn Y\setminus\{y_1\}}={\mathcal R}_{Y}\oplus\mathrm{span}_{\Q}{\calL_{irr}^{\infty}(Y)})$.
\end{itemize}

\bigskip\noindent
	{\bf LocaleCordonateIdentification($N$)}{\it \% $N\in\N_{\ge1}$  \%}\\
  $\calZ_{irr}^{\infty}(\calX):=\{\},\calL_{irr}^{\infty}(\calX):=\{\},\calR_{irr}(\calX):=\{\}$;\\
	for $p$ ranges in $2,\ldots,N$ do \%$\calL yn^p(\calX)$ is the set of Lyndon words of weight $p$\%
	\begin{quotation}\noindent
		by elimination, obtain the system of equations in $\{\zeta(\Sigma_l)\}_{l\in\Lyn^p(X)}$;\\
		by elimination, obtain the system of equations in $\{\zeta(S_l)\}_{l\in\Lyn^p(Y)}$;\\
		for $l$ ranges in the totally ordered $\Lyn^p(\calX)$ do
		\begin{quotation}\noindent
			identify the coefficient of $\Pi_l$ in $Z_{\gamma}=B(y_1)\pi_YZ_{\shuffle}$;\\
			identify the coefficient of $P_l$ in $\pi_XZ_{\gamma}=B(x_1)Z_{\shuffle}$
		\end{quotation}
		end\_for;\\
		for $l$ ranges in the totally ordered $\calL yn^p(\calX)$ do
		\begin{quotation}\noindent
			express the local coordinate $\zeta(\Sigma_l)$ as rewriting rule;\\
			if $\zeta(\Sigma_l)\rightarrow\zeta(\Sigma_l)$
				then $\calZ_{irr}^{\infty}(Y)=\calZ_{irr}^{\infty}(Y)\cup\{\zeta(\Sigma_l)\}$ and $\calL_{irr}^{\infty}(Y)=\calL_{irr}^{\infty}(Y)\cup\{\Sigma_l\}$\\
				else ${\mathcal R}_{irr}(Y)={\mathcal R}_{irr}(Y)\cup\{\Sigma_l\rightarrow\Upsilon_l\}$;\\
			express the local coordinate $\zeta(S_l)$ as rewriting rule;\\
			if $\zeta(S_l)\rightarrow\zeta(S_l)$
				then $\calZ_{irr}^{\infty}(X)=\calZ_{irr}^{\infty}(X)\cup\{\zeta(S_l)\}$ and $\calL_{irr}^{\infty}(X)=\calL_{irr}^{\infty}(X)\cup\{S_l\}$\\
				else ${\mathcal R}_{irr}(X)={\mathcal R}_{irr}(X)\cup\{S_l\rightarrow U_l\}$
		\end{quotation}
		end\_for
	\end{quotation}
	end\_for

\begin{proposition}[\cite{VJM,CM}]
\begin{eqnarray*}
\Q[\{S_l\}_{l\in\Lyn X\setminus X}]={\mathcal R}_X\oplus\Q[{\calL_{irr}^{\infty}(X)}],&
\Q[\{\Sigma_l\}_{l\in\Lyn Y\setminus\{y_1\}}]={\mathcal R}_Y\oplus\Q[\calL_{irr}^{\infty}(Y)].
\end{eqnarray*}
\end{proposition}

\begin{proof}
For any $w\in x_0X^*x_1$ (resp. $(Y\setminus\{y_1\})Y^*$), using ${\mathcal R}_{\calX}$
and decomposing in the basis $\{S_l\}_{l\in\Lyn X\setminus X}$ (resp. $\{\Sigma_l\}_{l\in\Lyn Y\setminus\{y_1\}}$),
$\zeta(w)\in\Q[{\calZ_{irr}^{\infty}(\calX)}]$.
Hence, for any $P\in\Q[\{S_l\}_{l\in\Lyn X\setminus X}]$ (resp. $\Q[\{\Sigma_l\}_{l\in\Lyn Y\setminus\{y_1\}}]$) such that
$P\notin\ker\zeta\supseteq{\mathcal R}_{\calX}$, one obtains, by linearity, $\zeta(P)\in\Q[\calZ_{irr}^{\infty}(\calX)]$. Next, let
$Q\in{\mathcal R}_{\calX}\cap\Q[\calL_{irr}^{\infty}(\calX)]$. Since ${\mathcal R}_{\calX}\subseteq\ker\zeta$ then $\zeta(Q)=0$. Moreover,
restricted on $\Q[\calL_{irr}^{\infty}(\calX)]$, the polymorphism $\zeta$ is bijective and then $Q=0$. It follows then the expected result.
\end{proof}

\begin{example}[irreducible polyzetas up to weight $12$, \cite{Bui}]\label{E0}
	\begin{eqnarray*}
		\calZ_{irr}^{\le12}(X)&=&
		\{\zeta(S_{x_0x_1}),\zeta(S_{x_0^2x_1}),\zeta(S_{x_0^4x_1}),\zeta(S_{x_0^6x_1}),\zeta(S_{x_0x_1^2x_0x_1^4}),\zeta(S_{x_0^8x_1}),\\
		&&\zeta(S_{x_0x_1^2x_0x_1^6}),\zeta(S_{x_0^{10}x_1}),\zeta(S_{x_0x_1^3x_0x_1^7}),\zeta(S_{x_0x_1^2x_0x_1^8}),\zeta(S_{x_0x_1^4x_0x_1^6})\}.\cr
		\calL_{irr}^{\le12}(X)&=&\{S_{x_0x_1},S_{x_0^2x_1},S_{x_0^4x_1},S_{x_0^6x_1},S_{x_0x_1^2x_0x_1^4},S_{x_0^8x_1},S_{x_0x_1^2x_0x_1^6},S_{x_0^{10}x_1},\\
		&&S_{x_0x_1^3x_0x_1^7},S_{x_0x_1^2x_0x_1^8},S_{x_0x_1^4x_0x_1^6}\}.\\
		\calZ_{irr}^{\le12}(Y)&=&
		\{\zeta(\Sigma_{y_2}),\zeta(\Sigma_{y_3}),\zeta(\Sigma_{y_5}),\zeta(\Sigma_{y_7}),\zeta(\Sigma_{y_3y_1^5}),\zeta(\Sigma_{y_9}),\zeta(\Sigma_{y_3y_1^7}),\\
		&&\zeta(\Sigma_{y_{11}}),\zeta(\Sigma_{y_2y_1^9}),\zeta(\Sigma_{y_3y_1^9}),\zeta(\Sigma_{y_2^ 2y_1^8})\}.\cr
		\calL_{irr}^{\le12}(Y)&=&\{\Sigma_{y_2},\Sigma_{y_3},\Sigma_{y_5},\Sigma_{y_7},\Sigma_{y_3y_1^5},\Sigma_{y_9},\Sigma_{y_3y_1^7},\Sigma_{y_{11}},\Sigma_{y_2y_1^9},\Sigma_{y_3y_1^9},\Sigma_{y_2^ 2y_1^8}\}.
	\end{eqnarray*}
\end{example}

\begin{corollary}[\cite{VJM,CM}]
One has
\begin{enumerate}
	\item $\Q[\{\zeta(p)\}_{p\in\calL_{irr}^{\infty}}(\calX)]=\calZ=\mathrm{Im}\,\zeta$ and ${\mathcal R}_{\calX}=\ker\zeta$.
	\item $\calZ=\Q1\oplus\displaystyle\bigoplus_{k\ge2}\calZ_k$.
	\item If $P\in\calL_{irr}^{\infty}(\calX)$, homogenous in weight, then $\zeta(P)$ is transcendent number.
\end{enumerate}
\end{corollary}

\begin{proof}
\begin{enumerate}
	\item Let $Q\in\ker\zeta$ such that $\scal{Q}{1_{\calX^*}}=0$. Then $Q=Q_1+Q_2$ with $Q_2\in\Q[\calL_{irr}^{\infty}(\calX)]$
	and $Q_1\in{\mathcal R}_{\calX}$. Thus, reducing by ${\mathcal R}_{\calX}$, one obtains
	$Q\equiv_{{\mathcal R}_{\calX}}Q_1\in{\mathcal R}_{\calX}$ and it follows then the	expected result.
	\item As an ideal generated by homogenous in weight polynomials, $\ker\zeta$ is graded. 
	Since $\calZ\cong\Q1_{Y^*}\oplus(Y\setminus\{y_1\})\ncp{\Q}{Y}/\ker\zeta\cong\Q1_{X^*}\oplus x_0\QX x_1/\ker\zeta$
	then it follows that, as a quotient, $\calZ$ is also graded.
	\item Let $\xi=\zeta(P)$, where $P\in\ncp{\Q}{\calX}$ and $P\notin{\ker\zeta}$, homogenous in weight.
	Since, for any $p$ and $n\ge1$, one has $\calZ_p\calZ_n\subset\calZ_{p+n}$ then each monomial $\xi^n$,
	for $n\ge1$, is of different weight. Thus ${\xi}$ could not satisfy $\xi^n+a_{n-1}{\xi}^{n-1}+\ldots=0$
	(with $a_{n-1},\ldots\in\Q$). It follows then the expected result.
\end{enumerate}
\end{proof}

\section{Conclusion}
In this paper, we surveyed our results concerning the renormalization and regularization
of the zeta functions, $\{\zeta(s_1,\ldots,s_r)\}^{r\ge1}_{(s_1,\ldots,s_r)\in\C^r}$
via noncommutative symbolic computation and generating series of polylogarithms,
$\{\Li_{s_1,\ldots,s_r}\}^{r\ge1}_{(s_1,\ldots,s_r)\in\C^r}$, and of harmonic sums,
$\{\H_{s_1,\ldots,s_r}\}^{r\ge1}_{(s_1,\ldots,s_r)\in\C^r}$. This is based mainly on
the combinatorics on the shuffle bialgebras and their diagonal series, \textit{i.e.}
${\mathcal D}_{\shuffle},{\mathcal D}_{\stuffle}$ and ${\mathcal D}_X$. In particular,
it operated with
\begin{enumerate}
	\item The construction of the pairs of bases (Lie algebra bases and transcendence bases)
	in duality (Theorem \ref{isomorphy}) to factorize the noncommutative rational series
	and to get algebraic structure of
	$\{\zeta(s_1,\ldots,s_r)\}^{r\ge1}_{(s_1,\ldots,s_r)\in\N_{\ge1}^r}$
	by identification of locale coordinates, in infinite dimension (Corollary \ref{pont}).
	
	\item The algebraic structures (Theorems \ref{structure1}, \ref{structure2} and \ref{Indexation2})
	and the singularity analysis (Theorems \ref{renormalization1}, \ref{renormalization2} and \ref{fin})
	of the polylogarithms and the harmonic sums, for which the global renormalizations has been obtained via
	the Abel like theorems, for the pairs of generating series $\L,\H$ and $\L^-,\H^-$.
	
	In particular, the series $\L$ corresponds to the actual solution of \eqref{DE},
	satisfying the standard asymptotic behaviors,
	and the series $Z_{\shuffle}$ corresponds to the associator $\Phi_{KZ}$. 
	
	Nontrivial expressions for Drindfel'd series with rational coefficients,
	\textit{i.e.} $Z^-_{\shuffle}$ and $Z^-_{\gamma}$, were also explicitly provided
	thanks to various process of regularization via the noncommutative generating series
	$\Lambda$ and $\Upsilon$, which are group-like series, for respectively,
	$\Delta_{\shuffle}$ and $\Delta_{\stuffle}$ (Theorem \ref{fin}).
	
	\item Via the locale coordinates of the power series	$Z_{\shuffle},Z^-_{\shuffle},Z_{\gamma},Z^-_{\gamma}$ and $Z_{\stuffle}$,
	the regularization maps for divergent zeta were constructed 	(Propositions \ref{to0} and \ref{reg_alg}) over algebraic bases
	matching with analytical meaning: on the one hand, the character $\zeta_{\shuffle}$ corresponds to the regularization,
	obtained as the finite parts of the singular expansions of $\{\Li_{s_1,\ldots,s_r}\}^{r\ge1}_{(s_1,\ldots,s_r)\in\Z^r}$;
	on the other hand, the characters $\zeta_{\stuffle}$ and $\gamma_{\bullet}$ correspond to the regularizations obtained as the
	finite parts of the asymptotic expansions of $\{\H_{s_1,\ldots,s_r}\}^{r\ge1}_{(s_1,\ldots,s_r)\in\Z^r}$, in different
	comparison scales.
	
	\item After that, thanks to the Abel like results, we obtained equations bridging the algebraic structures of converging polyzetas
	$\{\zeta(S_l)\}_{l\in\Lyn X\setminus X}$ and $\{\zeta(\Sigma_l)\}_{l\in\Lyn Y\setminus \{y_1\}}$ (Corollary \ref{pont}) and, by
	local coordinates identification, leading to
	\begin{enumerate}
		\item the polynomial relations, homogenous in weight, among converging polyzetas
		(Table \ref{E1} in Appendix),
		\item the algebraic generators for polyzetas (Example \ref{E0}),
		$\calL_{irr}^{\infty}(X)$ and $\calL_{irr}^{\infty}(Y)$.
	\end{enumerate}
	Then, using the algorithm	{\bf LocaleCordonateIdentification} (partially implemented in \cite{Bui}),
	the Zagier's dimension conjecture \cite{zagier} holds, up to weight $12$, meaning that
	the irreducible polyzetas in $\calZ_{irr}^{\le12}(\calX)$ (Example \ref{E0})
	is $\Q$-algebraically independent (see \cite{VJM}, for example, for a short discussion).
	
	\item To end, let $l\in\Lyn\calX$ such that $l\neq y_1$ and $l\neq x_0,x_1$. Then one has
	\begin{enumerate}
		\item $l>y_n$ and $l>x_0^{n-1}x_1$,
		\item $\Sigma_{y_n}=y_n\in\Lyn Y$ and $S_{x_0^{n-1}x_1}=x_0^{n-1}x_1\in\Lyn X$.
	\end{enumerate}
	Unfortunately, $\{\Sigma_{y_n}\}_{n>1}\not\subset\calL_{irr}^{\infty}(Y)$
	and $\{S_{x_0^{n-1}x_1}\}_{n>1}\not\subset\calL_{irr}^{\infty}(X)$.	Indeed,
	\begin{enumerate}
		\item $\zeta(2)=\zeta(\Sigma_{y_2})=\zeta(S_{x_0x_1})$ is then irreducible (see Example \ref{E0}).
		\item By a Euler's identity about the ratio\footnote{See also \cite{QPL}, for analogous study
			of the ratios $\zeta(s_1,\ldots,s_r)/\pi^{s_1+\ldots+s_r}$.} $\zeta(2k)/\pi^{2k},k>1$, one deduces that
		$\Sigma_{y_{2k}}=y_{2k}\notin\calL_{irr}^{\infty}(Y)$ and $S_{x_0^{2k-1}x_1}=x_0^{2k-1}x_1\notin\calL_{irr}^{\infty}(X)$.
		\item By Example \ref{E0}, one obtains $\zeta(3),\zeta(5),\zeta(7),\zeta(9),\zeta(11)\in\calZ_{irr}^{\le12}(\calX)$.
		\item It could remains that, for any $n\ge1$, $\Sigma_{y_{2n+1}}=y_{2n+1}\in\calL_{irr}^{\infty}(Y)$ and
		$S_{x_0^{2n}x_1}=x_0^{2n}x_1\in\calL_{irr}^{\infty}(X)$.
	\end{enumerate}
\end{enumerate}

In Appendix, we will explain how (\ref{DE}) arises in Knizhnik--Zamolodchikov differential equations $(KZ_n)$,
as proposed for $n=3$ in \cite{cartier1,drinfeld2}. Since the work presented in this text is quite extensive
then it will be continued in \cite{QTS} by treating the case of $n\ge4$, as application of a Picard-Vessiot
theory of noncommutative differential equations (see \cite{PVNC}).

\section{Appendix}
\subsection{$KZ_3:$ simplest non-trivial case of $KZ_n$}\label{KZ_3}
Let $\calV$ be the universal covering, $\widetilde{\C_*^n}$,
of the configuration space of $n$ points on the complex plane,
\begin{eqnarray}
\C_*^n:=\{z=(z_1,\ldots,z_n)\in\C^n|z_i\neq z_j\mbox{ for }i\neq j\}
\end{eqnarray}
and the following differential equation over $\ncs{\calH(\widetilde{\C_*^n})}{\calT_n}$ \cite{}
\begin{eqnarray}\label{KZn}
{\bf d}F=\Omega_nF,&\mbox{where}&\Omega_n(z):=\sum_{1\le i<j\le n}\frac{t_{i,j}}{2{\rm i}\pi}d\log(z_i-z_j),
\end{eqnarray}
so-called $KZ_n$ equation and $\Omega_n$ is called universal KZ connection form.

\begin{example}[trivial case]
$\calT_2=\{t_{1,2}\}$ and $\Omega_2(z)=(t_{1,2}/2{\rm i}\pi)d\log(z_1-z_2)$.
A solution of ${\bf d}F=\Omega_2F$ is
$F(z_1,z_2)=e^{(t_{1,2}/2{\rm i}\pi)\log(z_1-z_2)}=(z_1-z_2)^{t_{1,2}/2{\rm i}\pi}$
and it belongs to $\ncs{\calH(\widetilde{\C_*^2})}{\calT_2}$,
\end{example}

\begin{example}[simplest non-trivial case]\label{$KZ_3$bis}
Solution of ${\bf d}F=\Omega_3F$ can be computed as limit of the sequence $\{F_l\}_{l\ge0}$,
in $\ncs{\calH(\widetilde{\C_*^3})}{\calT_3}$, by convergent Picard's iteration as follows
\begin{eqnarray*}
F_0(z^0,z)=1_{\calH(\widetilde{\C_*^n})}&\mbox{and}&\forall i\ge1,F_i(z^0,z)=F_{i-1}(z^0,z)+\int_{z^0}^z\Omega_3(s)F_{i-1}(s).
\end{eqnarray*}
It can be also computed, by another way, by the sequence $\{V_l\}_{l\ge0}$,
in $\ncs{\calH(\widetilde{\C_*^3})}{\calT_3}$, satisfying the following recursion
\begin{eqnarray*}
V_0(z)&=&e^{(t_{1,2}/2{\rm i}\pi)\log(z_1-z_2)},\\
V_l(z)&=&V_0(z)\int_0^zV_0^{-1}(s)\Big(\frac{t_{1,3}}{2{\rm i}\pi}d\log(z_1-z_3)+\frac{t_{2,3}}{2{\rm i}\pi}d\log(z_2-z_3)\Big)V_{l-1}(s)\\
&=&e^{(t_{1,2}/2{\rm i}\pi)\log(z_1-z_3)}\int_0^ze^{-(t_{1,2}/2{\rm i}\pi)\log(s_1-s_2)}\bar{\Omega}_3(s)V_{l-1}(s),
\end{eqnarray*}
where $\bar{\Omega}_3=(t_{1,3}d\log(z_1-z_3)+t_{2,3}d\log(z_2-z_3))/{2{\rm i}\pi}$.
\end{example}

With the notations given in Example \ref{$KZ_3$bis}, solution of $KZ_3$ is explicit as $F=V_0G$,
where $V_0(z)=(z_1-z_2)^{t_{1,2}/2{\rm i}\pi}$ and $G$ is expanded as follows
\begin{eqnarray*}
G(z)=\sum_{m\ge0}\sum_{t_{i_1,j_1}\ldots t_{i_m,j_m}\in\{t_{1,3},t_{2,3}\}^*}
\int_0^z\omega_{i_1,j_1}(s_1)\varphi^{s_1}(t_{i_1,j_1})\ldots\int_0^{s_{m-1}}\\
\omega_{i_m,j_m}(s_m)\varphi^{s_m}(t_{i_m,j_m}),
\end{eqnarray*}
where $\omega_{1,3}(z)=d\log(z_1-z_3)$ and $\omega_{2,3}(z)=d\log(z_2-z_3)$ and
\begin{eqnarray*}
\varphi^z
=e^{\ad_{-(t_{1,2}/2{\rm i}\pi)\log(z_1-z_2)}}
=\sum_{k\ge1}\frac{\log^k(z_1-z_2)}{(-2{\rm i}\pi)^kk!}\ad^k_{t_{1,2}}.
\end{eqnarray*}
One also has
$\varphi^{(\varsigma,s_1)}(t_{i_1,j_1})\ldots\varphi^{(\varsigma,s_m)}(t_{i_m,j_m})
=V_0(z)^{-1}\hat\kappa_{t_{i_1,j_1}\ldots t_{i_m,j_m}}(z,s_1,\cdots,s_m)$.

Moreover, Example \ref{$KZ_3$bis} (equipping the ordering $t_{1,2}\prec t_{1,3}\prec t_{2,3}$), one has 
\begin{eqnarray*}
\varphi^z(t_{i,3})=\sum\limits_{k\ge0}\Frac{\log^k(z_1-z_2)}{(-2{\rm i}\pi)^kk!}P_{t_{1,2}^kt_{i,3}},&
\check\varphi^z(t_{i,3})=\sum\limits_{k\ge0}\Frac{\log^k(z_1-z_2)}{(-2{\rm i}\pi)^kk!}S_{t_{1,2}^kt_{i,3}},
\end{eqnarray*}
where $\check\varphi$ is the adjoint to $\varphi$ and is defined by
\begin{eqnarray*}
\check\varphi^z
=\sum_{k\ge0}\frac{\log^k(z_1-z_2)}{(-2{\rm i}\pi)^kk!}t_{1,2}^k
=e^{-(t_{1,2}/2{\rm i}\pi)\log(z_1-z_2)}.
\end{eqnarray*}

Hence, belonging to $\ncs{\calH(\widetilde{\C_*^3})}{\calT_3}$, $G$ satisfies ${\bf d}G(z)=\bar{\Omega}_2(z)G(z)$, where
\begin{eqnarray*}
\bar{\Omega}_2(z)=(\varphi^z(t_{1,3})d\log(z_1-z_3)+\varphi^z(t_{2,3})d\log(z_2-z_3))/2{\rm i}\pi.
\end{eqnarray*}
In the affine plan $(P_{1,2}):z_1-z_2=1$, one has $\log(z_1-z_2)=0$ and then $\varphi\equiv\mathrm{Id}$.
Changing $x_0=t_{1,3}/2{\rm i}\pi,x_1=-t_{2,3}/2{\rm i}\pi$ and setting $z_1=1,z_2=0,z_3=s$,
\begin{eqnarray*}
\bar{\Omega}_2(z)
=\frac{1}{2{\rm i}\pi}\Big(t_{1,3}\frac{d(z_1-z_3)}{z_1-z_3}+t_{2,3}\frac{d(z_2-z_3)}{z_2-z_3}\Big)
=x_1\omega_1(s)+x_0\omega_0(s),
\end{eqnarray*}
and ${\bf d}G(z)=\bar\Omega_2(z)G(z)$ admits the noncommutative generating series of polylogarithms
as the actual solution satisfying the asymptotic conditions in \eqref{asymcond}.
Thus, by $\L$ defined in \eqref{LH}, and the homographic substitution $g:z_3\longmapsto(z_3-z_2)/(z_1-z_2)$,
mapping\footnote{Generally, $s\mapsto(s-a)(c-b)(s-b)^{-1}(c-a)^{-1}$ maps the singularities $\{a,b,c\}$
in $\{0,+\infty,1\}$.} $\{z_2,z_1\}$ to $\{0,1\}$, a particular solution of $KZ_3$, in $(P_{1,2})$,
is $\L\Big(\Frac{z_3-z_2}{z_1-z_2}\Big)$. So does\footnote{
Note also that these solutions could not be obtained by Picard's iteration in Example \ref{$KZ_3$bis}.

$(z_1-z_2)^{(t_{1,2}+t_{2,3}+t_{1,3})/2{\rm i}\pi}=e^{((t_{1,2}+t_{2,3}+t_{1,3})/2{\rm i}\pi)\log(z_1-z_2)}$,
which is grouplike and independent on the variable $z_3=s$, and then belongs to the differential Galois group
of $KZ_3$.} $\L\Big(\Frac{z_3-z_2}{z_1-z_2}\Big)(z_1-z_2)^{(t_{1,2}+t_{1,3}+t_{2,3})/2{\rm i}\pi}$.

To end with $KZ_3$, the infinitesimal braid relations \cite{drinfeld2}, one has $[t_{1,2}+t_{2,3}+t_{1,3},t]=0$,
for $t\in\calT_3$, meaning that $t$ commutes with $(z_1-z_2)^{(t_{1,2}+t_{2,3}+t_{1,3})/2{\rm i}\pi}$ and then 
$(z_1-z_2)^{(t_{1,2}+t_{1,3}+t_{2,3})/2{\rm i}\pi}$ commutes with $\ncs{\calA}{\calT_3}$. Thus, $KZ_3$ also admits
$(z_1-z_2)^{(t_{1,2}+t_{1,3}+t_{2,3})/2{\rm i}\pi}\L\Big(\Frac{z_3-z_2}{z_1-z_2}\Big)$ as a particular solution
in $(P_{1,2})$. The last expression is quite in the form proposed in \cite{drinfeld2}.

\subsection{Computational examples}\label{examples}
Due to Corollary \ref{pont}, by identifying local coordinate in weight with the algorithm
{\bf LocaleCordonateIdentification}, we get the following polynomial relations on local coordinates:
\begin{table}[htbp]	
	\centering
	\caption{Polynomial relations on local coordinates.}\label{E1}
	$\footnotesize\allowdisplaybreaks
	\begin{array}{crclrcl}
		\hline
		&\mbox{Relations}\!\!&\text{on}&\!\!\{\zeta(\Sigma_l)\}_{l\in\Lyn Y-\{y_1\}}&\mbox{Relations}\!\!&\text{on}&\!\!\{\zeta(S_l)\}_{l\in\Lyn X-X}\\
		\hline
		3&\zeta(\Sigma_{y_2y_1})&=&\frac{3}{2}\zeta(\Sigma_{y_3})&\zeta(S_{x_0x_1^2})&=&\zeta(S_{x_0^2x_1})\\ 
		\hline
		&\zeta(\Sigma_{y_4})&=&\frac{2}{5}\zeta(\Sigma_{y_2})^2&\zeta(S_{x_0^3x_1})&=&\frac{2}{5}\zeta(S_{x_0x_1})^2\\               
		4&\zeta(\Sigma_{y_3y_1})&=&\frac{3}{10}\zeta(\Sigma_{y_2})^2&\zeta(S_{x_0^2x_1^2})&=&\frac{1}{10}\zeta(S_{x_0x_1})^2\\  
		&\zeta(\Sigma_{y_2y_1^2})&=&\frac{2}{3}\zeta(\Sigma_{y_2})^2&\zeta(S_{x_0x_1^3})&=&\frac{2}{5}\zeta(S_{x_0x_1})^2\\ 
		\hline
		&\zeta(\Sigma_{y_3y_2})&=&3\zeta(\Sigma_{y_3})\zeta(\Sigma_{y_2})-5\zeta(\Sigma_{y_5})&\zeta(S_{x_0^3x_1^2})&=&-\zeta(S_{x_0^2x_1})\zeta(S_{x_0x_1})+2\zeta(S_{x_0^4x_1})\\ 
		&\zeta(\Sigma_{y_4y_1})&=&-\zeta(\Sigma_{y_3})\zeta(\Sigma_{y_2})+\frac{5}{2}\zeta(\Sigma_{y_5})&\zeta(S_{x_0^2x_1x_0x_1})&=&-\frac{3}{2}\zeta(S_{x_0^4x_1})+\zeta(S_{x_0^2x_1})\zeta(S_{x_0x_1})\\   
		5&\zeta(\Sigma_{y_2^2y_1})&=&\frac{3}{2}\zeta(\Sigma_{y_3})\zeta(\Sigma_{y_2})-\frac{25}{12}\zeta(\Sigma_{y_5})&\zeta(S_{x_0^2x_1^3})&=&-\zeta(S_{x_0^2x_1})\zeta(S_{x_0x_1})+2\zeta(S_{x_0^4x_1})\\ 
		&\zeta(\Sigma_{y_3y_1^2})&=&\frac{5}{12}\zeta(\Sigma_{y_5})&\zeta(S_{x_0x_1x_0x_1^2})&=&\frac{1}{2}\zeta(S_{x_0^4x_1})\\ 
		&\zeta(\Sigma_{y_2y_1^3})&=&\frac{1}{4}\zeta(\Sigma_{y_3})\zeta(\Sigma_{y_2})+\frac{5}{4}\zeta(\Sigma_{y_5})&\zeta(S_{x_0x_1^4})&=&\zeta(S_{x_0^4x_1})\\ 
		\hline
		&\zeta(\Sigma_{y_6})&=&\frac{8}{35}\zeta(\Sigma_{y_2})^3& \zeta(S_{x_0^5x_1})&=&\frac{8}{35}\zeta(S_{x_0x_1})^3\\ 
		&\zeta(\Sigma_{y_4y_2})&=&\zeta(\Sigma_{y_3})^2-\frac{4}{21}\zeta(\Sigma_{y_2})^3& \zeta(S_{x_0^4x_1^2})&=&\frac{6}{35}\zeta(S_{x_0x_1})^3-\frac{1}{2}\zeta(S_{x_0^2x_1})^2\\ 
		&\zeta(\Sigma_{y_5y_1})&=&\frac{2}{7}\zeta(\Sigma_{y_2})^3-\frac{1}{2}\zeta(\Sigma_{y_3})^2& \zeta(S_{x_0^3x_1x_0x_1})&=&\frac{4}{105}\zeta(S_{x_0x_1})^3 \\ 
		&\zeta(\Sigma_{y_3y_1y_2})&=&-\frac{17}{30}\zeta(\Sigma_{y_2})^3+\frac{9}{4}\zeta(\Sigma_{y_3})^2&\zeta(S_{x_0^3x_1^3})&=&\frac{23}{70}\zeta(S_{x_0x_1})^3-\zeta(S_{x_0^2x_1})^2\\ 
		6&\zeta(\Sigma_{y_3y_2y_1})&=&3\zeta(\Sigma_{y_3})^2-\frac{9}{10}\zeta(\Sigma_{y_2})^3& \zeta(S_{x_0^2x_1x_0x_1^2})&=&\frac{2}{105}\zeta(S_{x_0x_1})^3\\ 
		&\zeta(\Sigma_{y_4y_1^2})&=&\frac{3}{10}\zeta(\Sigma_{y_2})^3-\frac{3}{4}\zeta(\Sigma_{y_3})^2&\zeta(S_{x_0^2x_1^2x_0x_1})&=&-\frac{89}{210}\zeta(S_{x_0x_1})^3+\frac{3}{2}\zeta(S_{x_0^2x_1})^2 \\ 
		&\zeta(\Sigma_{y_2^2y_1^2})&=&\frac{11}{63}\zeta(\Sigma_{y_2})^3-\frac{1}{4}\zeta(\Sigma_{y_3})^2& \zeta(S_{x_0^2x_1^4})&=&\frac{6}{35}\zeta(S_{x_0x_1})^3-\frac{1}{2}\zeta(S_{x_0^2x_1})^2\\ 
		&\zeta(\Sigma_{y_3y_1^3})&=&\frac{1}{21}\zeta(\Sigma_{y_2})^3&\zeta(S_{x_0x_1x_0x_1^3})&=&\frac{8}{21}\zeta(S_{x_0x_1})^3-\zeta(S_{x_0^2x_1})^2 \\ 
		&\zeta(\Sigma_{y_2y_1^4})&=&\frac{17}{50}\zeta(\Sigma_{y_2})^3+\frac{3}{16}\zeta(\Sigma_{y_3})^2&\zeta(S_{x_0x_1^5})&=&\frac{8}{35}\zeta(S_{x_0x_1})^3\\ 
	\hline
	\end{array}$
\end{table}

\begin{table}[htbp]	
	\centering
	\caption{Rewriting system on irreducible coordinates, \cite{Bui} (Replace ``$=$" by ``$\rightarrow$").}\label{E2}
	$\footnotesize \allowdisplaybreaks
	\begin{array}{crclrcl}
		\hline
		&\mbox{Relations}\!\!&\text{on}&\!\!\{\zeta(\Sigma_l)\}_{l\in\Lyn Y-\{y_1\}}&\mbox{Relations}\!\!&\text{on}&\!\!\{\zeta(S_l)\}_{l\in\Lyn X-X}\\
		\hline
		3&\zeta(\Sigma_{y_2y_1})&\to&\frac{3}{2}\zeta(\Sigma_{y_3})&\zeta(S_{x_0x_1^2})&\to&\zeta(S_{x_0^2x_1})\\ 
		\hline
		&\zeta(\Sigma_{y_4})&\to&\frac{2}{5}\zeta(\Sigma_{y_2})^2&\zeta(S_{x_0^3x_1})&\to&\frac{2}{5}\zeta(S_{x_0x_1})^2\\               
		4&\zeta(\Sigma_{y_3y_1})&\to&\frac{3}{10}\zeta(\Sigma_{y_2})^2&\zeta(S_{x_0^2x_1^2})&\to&\frac{1}{10}\zeta(S_{x_0x_1})^2\\  
		&\zeta(\Sigma_{y_2y_1^2})&\to&\frac{2}{3}\zeta(\Sigma_{y_2})^2&\zeta(S_{x_0x_1^3})&\to&\frac{2}{5}\zeta(S_{x_0x_1})^2\\ 
		\hline
		&\zeta(\Sigma_{y_3y_2})&\to&3\zeta(\Sigma_{y_3})\zeta(\Sigma_{y_2})-5\zeta(\Sigma_{y_5})&\zeta(S_{x_0^3x_1^2})&\to&-\zeta(S_{x_0^2x_1})\zeta(S_{x_0x_1})+2\zeta(S_{x_0^4x_1})\\ 
		&\zeta(\Sigma_{y_4y_1})&\to&-\zeta(\Sigma_{y_3})\zeta(\Sigma_{y_2})+\frac{5}{2}\zeta(\Sigma_{y_5})&\zeta(S_{x_0^2x_1x_0x_1})&\to&-\frac{3}{2}\zeta(S_{x_0^4x_1})+\zeta(S_{x_0^2x_1})\zeta(S_{x_0x_1})\\   
		5&\zeta(\Sigma_{y_2^2y_1})&\to&\frac{3}{2}\zeta(\Sigma_{y_3})\zeta(\Sigma_{y_2})-\frac{25}{12}\zeta(\Sigma_{y_5})&\zeta(S_{x_0^2x_1^3})&\to&-\zeta(S_{x_0^2x_1})\zeta(S_{x_0x_1})+2\zeta(S_{x_0^4x_1})\\ 
		&\zeta(\Sigma_{y_3y_1^2})&\to&\frac{5}{12}\zeta(\Sigma_{y_5})&\zeta(S_{x_0x_1x_0x_1^2})&\to&\frac{1}{2}\zeta(S_{x_0^4x_1})\\ 
		&\zeta(\Sigma_{y_2y_1^3})&\to&\frac{1}{4}\zeta(\Sigma_{y_3})\zeta(\Sigma_{y_2})+\frac{5}{4}\zeta(\Sigma_{y_5})&\zeta(S_{x_0x_1^4})&\to&\zeta(S_{x_0^4x_1})\\ 
		\hline
		&\zeta(\Sigma_{y_6})&\to&\frac{8}{35}\zeta(\Sigma_{y_2})^3& \zeta(S_{x_0^5x_1})&\to&\frac{8}{35}\zeta(S_{x_0x_1})^3\\ 
		&\zeta(\Sigma_{y_4y_2})&\to&\zeta(\Sigma_{y_3})^2-\frac{4}{21}\zeta(\Sigma_{y_2})^3& \zeta(S_{x_0^4x_1^2})&\to&\frac{6}{35}\zeta(S_{x_0x_1})^3-\frac{1}{2}\zeta(S_{x_0^2x_1})^2\\ 
		&\zeta(\Sigma_{y_5y_1})&\to&\frac{2}{7}\zeta(\Sigma_{y_2})^3-\frac{1}{2}\zeta(\Sigma_{y_3})^2& \zeta(S_{x_0^3x_1x_0x_1})&\to&\frac{4}{105}\zeta(S_{x_0x_1})^3 \\ 
		&\zeta(\Sigma_{y_3y_1y_2})&\to&-\frac{17}{30}\zeta(\Sigma_{y_2})^3+\frac{9}{4}\zeta(\Sigma_{y_3})^2&\zeta(S_{x_0^3x_1^3})&\to&\frac{23}{70}\zeta(S_{x_0x_1})^3-\zeta(S_{x_0^2x_1})^2\\ 
		6&\zeta(\Sigma_{y_3y_2y_1})&\to&3\zeta(\Sigma_{y_3})^2-\frac{9}{10}\zeta(\Sigma_{y_2})^3& \zeta(S_{x_0^2x_1x_0x_1^2})&\to&\frac{2}{105}\zeta(S_{x_0x_1})^3\\ 
		&\zeta(\Sigma_{y_4y_1^2})&\to&\frac{3}{10}\zeta(\Sigma_{y_2})^3-\frac{3}{4}\zeta(\Sigma_{y_3})^2&\zeta(S_{x_0^2x_1^2x_0x_1})&\to&-\frac{89}{210}\zeta(S_{x_0x_1})^3+\frac{3}{2}\zeta(S_{x_0^2x_1})^2 \\ 
		&\zeta(\Sigma_{y_2^2y_1^2})&\to&\frac{11}{63}\zeta(\Sigma_{y_2})^3-\frac{1}{4}\zeta(\Sigma_{y_3})^2& \zeta(S_{x_0^2x_1^4})&\to&\frac{6}{35}\zeta(S_{x_0x_1})^3-\frac{1}{2}\zeta(S_{x_0^2x_1})^2\\ 
		&\zeta(\Sigma_{y_3y_1^3})&\to&\frac{1}{21}\zeta(\Sigma_{y_2})^3&\zeta(S_{x_0x_1x_0x_1^3})&\to&\frac{8}{21}\zeta(S_{x_0x_1})^3-\zeta(S_{x_0^2x_1})^2 \\ 
		&\zeta(\Sigma_{y_2y_1^4})&\to&\frac{17}{50}\zeta(\Sigma_{y_2})^3+\frac{3}{16}\zeta(\Sigma_{y_3})^2&\zeta(S_{x_0x_1^5})&\to&\frac{8}{35}\zeta(S_{x_0x_1})^3\\ 
	\hline
	\end{array}$
\end{table}

\begin{table}[htbp]	
	\centering
	\caption{Homogeneous polynomials generating inside $\ker\zeta$, \cite{Bui}\label{E3}
		($\zeta$ is surjective).}
	$\footnotesize \allowdisplaybreaks
	\begin{array}{crr}
	\hline
		&\{Q_l\}_{l\in\Lyn Y-\{y_1\}}&\{Q_l\}_{l\in\Lyn X-X}\\
		\hline
		3&\zeta({\Sigma_{y_2y_1}-\frac{3}{2}\Sigma_{y_3}})=0&\zeta({S_{x_0x_1^2}-S_{x_0^2x_1}})=0\\ 
		\hline
		&\zeta({\Sigma_{y_4}-\frac{2}{5}\Sigma_{y_2}^{\stuffle2}})=0&\zeta({S_{x_0^3x_1}-\frac{2}{5}S_{x_0x_1}^{\shuffle2}})=0\\          
		4&\zeta({\Sigma_{y_3y_1}-\frac{3}{10}\Sigma_{y_2}^{\stuffle2}})=0&\zeta({S_{x_0^2x_1^2}-\frac{1}{10}S_{x_0x_1}^{\shuffle2}})=0\\ 
		&\zeta({\Sigma_{y_2y_1^2}-\frac{2}{3}\Sigma_{y_2}^{\stuffle2}})=0&\zeta({S_{x_0x_1^3}-\frac{2}{5}S_{x_0x_1}^{\shuffle2}})=0\\ 
		\hline
		&\zeta({\Sigma_{y_3y_2}-3\Sigma_{y_3}\stuffle\Sigma_{y_2}-5\Sigma_{y_5}})=0&\zeta({S_{x_0^3x_1^2}-S_{x_0^2x_1}\shuffle S_{x_0x_1}+2S_{x_0^4x_1}})=0\\ 
		&\zeta({\Sigma_{y_4y_1}-\Sigma_{y_3}\stuffle\Sigma_{y_2})+\frac{5}{2}\Sigma_{y_5}})=0&\zeta({S_{x_0^2x_1x_0x_1}-\frac{3}{2}S_{x_0^4x_1}+S_{x_0^2x_1}\shuffle S_{x_0x_1}})=0\\ 
		5&\zeta({\Sigma_{y_2^2y_1}-\frac{3}{2}\Sigma_{y_3}\stuffle\Sigma_{y_2}-\frac{25}{12}\Sigma_{y_5}})=0&\zeta({S_{x_0^2x_1^3}-S_{x_0^2x_1}\shuffle S_{x_0x_1}+2S_{x_0^4x_1}})=0\\ 
		&\zeta({\Sigma_{y_3y_1^2}-\frac{5}{12}\Sigma_{y_5}})=0&\zeta({S_{x_0x_1x_0x_1^2}-\frac{1}{2}S_{x_0^4x_1}})=0\\ 
		&\zeta({\Sigma_{y_2y_1^3}-\frac{1}{4}\Sigma_{y_3}\stuffle\Sigma_{y_2})+\frac{5}{4}\Sigma_{y_5}})=0&\zeta({S_{x_0x_1^4}-S_{x_0^4x_1}})=0\\ 
		\hline
		&\zeta({\Sigma_{y_6}-\frac{8}{35}\Sigma_{y_2}^{\stuffle3}})=0&\zeta({S_{x_0^5x_1}-\frac{8}{35}S_{x_0x_1}^{\shuffle3}})=0\\ 
		&\zeta({\Sigma_{y_4y_2}-\Sigma_{y_3}^{\stuffle2}-\frac{4}{21}\Sigma_{y_2}^{\stuffle3}})=0&\zeta({S_{x_0^4x_1^2}-\frac{6}{35}S_{x_0x_1}^{\shuffle3}-\frac{1}{2}S_{x_0^2x_1}^{\shuffle2}})=0\\ 
		&\zeta({\Sigma_{y_5y_1}-\frac{2}{7}\Sigma_{y_2}^{\stuffle3}-\frac{1}{2}\Sigma_{y_3}^{\stuffle2}})=0&\zeta({S_{x_0^3x_1x_0x_1}-\frac{4}{105}S_{x_0x_1}^{\shuffle3}})=0\\ 
		&\zeta({\Sigma_{y_3y_1y_2}-\frac{17}{30}\Sigma_{y_2}^{\stuffle3}+\frac{9}{4}\Sigma_{y_3}^{\stuffle2}})=0&\zeta({S_{x_0^3x_1^3}-\frac{23}{70}S_{x_0x_1}^{\shuffle3}-S_{x_0^2x_1}^{\shuffle2}})=0\\ 
		6&\zeta({\Sigma_{y_3y_2y_1}-3\Sigma_{y_3}^{\stuffle2}-\frac{9}{10}\Sigma_{y_2}^{\stuffle3}})=0&\zeta({S_{x_0^2x_1x_0x_1^2}-\frac{2}{105}S_{x_0x_1}^{\shuffle3}})=0\\ 
		&\zeta({\Sigma_{y_4y_1^2}-\frac{3}{10}\Sigma_{y_2}^{\stuffle2}-\frac{3}{4}\Sigma_{y_3}^{\stuffle2}})=0&\zeta({S_{x_0^2x_1^2x_0x_1}-\frac{89}{210}S_{x_0x_1}^{\shuffle3}+\frac{3}{2}S_{x_0^2x_1}^{\shuffle2}})=0\\ 
		&\zeta({\Sigma_{y_2^2y_1^2}-\frac{11}{63}\Sigma_{y_2}^{\stuffle2}-\frac{1}{4}\Sigma_{y_3}^{\stuffle2}})=0&\zeta({S_{x_0^2x_1^4}-\frac{6}{35}S_{x_0x_1}^{\shuffle3}-\frac{1}{2}S_{x_0^2x_1}^{\shuffle2}})=0\\ 
		&\zeta({\Sigma_{y_3y_1^3}-\frac{1}{21}\Sigma_{y_2}^{\stuffle3}})=0&\zeta({S_{x_0x_1x_0x_1^3}-\frac{8}{21}S_{x_0x_1}^{\shuffle3}-S_{x_0^2x_1}^{\shuffle2}})=0\\ 
		&\zeta({\Sigma_{y_2y_1^4}-\frac{17}{50}\Sigma_{y_2}^{\stuffle3}+\frac{3}{16}\Sigma_{y_3}^{\stuffle2}})=0&\zeta({S_{x_0x_1^5}-\frac{8}{35}S_{x_0x_1}^{\shuffle3}})=0\\ 
	\hline
	\end{array}$
\end{table}

\begin{table}[htbp]	
	\centering
	\caption{Rewriting system on algebraic generators, \cite{Bui}.}\label{E4}
	$\footnotesize \allowdisplaybreaks
	\begin{array}{crclrcl}
	\hline
		&\mbox{Rewriting}\!\!&\text{on}&\!\!\{\Sigma_l\}_{l\in\Lyn Y-\{y_1\}}&\mbox{Rewriting}\!\!&\text{on}&\!\!\{S_l\}_{\Lyn X-X}\\ 
		\hline
		3&\Sigma_{y_2y_1}&\rightarrow&\frac{3}{2}\Sigma_{y_3}&S_{x_0x_1^2}&\rightarrow&S_{x_0^2x_1}\\ 
		\hline
		&\Sigma_{y_4}&\rightarrow&\frac{2}{5}\Sigma_{y_2}^2&S_{x_0^3x_1}&\rightarrow&\frac{2}{5}S_{x_0x_1}^2\\               
		4&\Sigma_{y_3y_1}&\rightarrow&\frac{3}{10}\Sigma_{y_2}^2&S_{x_0^2x_1^2}&\rightarrow&\frac{1}{10}S_{x_0x_1}^2\\  
		&\Sigma_{y_2y_1^2}&\rightarrow&\frac{2}{3}\Sigma_{y_2}^2&S_{x_0x_1^3}&\rightarrow&\frac{2}{5}S_{x_0x_1}^2\\ 
		\hline
		&\Sigma_{y_3y_2}&\rightarrow&3\Sigma_{y_3}\Sigma_{y_2}-5\Sigma_{y_5}&S_{x_0^3x_1^2}&\rightarrow&-S_{x_0^2x_1}S_{x_0x_1}+2S_{x_0^4x_1}\\ 
		&\Sigma_{y_4y_1}&\rightarrow&-\Sigma_{y_3}\Sigma_{y_2}+\frac{5}{2}\Sigma_{y_5}&S_{x_0^2x_1x_0x_1}&\rightarrow&-\frac{3}{2}S_{x_0^4x_1}+S_{x_0^2x_1}S_{x_0x_1}\\   
		5&\Sigma_{y_2^2y_1}&\rightarrow&\frac{3}{2}\Sigma_{y_3}\Sigma_{y_2}-\frac{25}{12}\Sigma_{y_5}&S_{x_0^2x_1^3}&\rightarrow&-S_{x_0^2x_1}S_{x_0x_1}+2S_{x_0^4x_1}\\ 
		&\Sigma_{y_3y_1^2}&\rightarrow&\frac{5}{12}\Sigma_{y_5}&S_{x_0x_1x_0x_1^2}&\rightarrow&\frac{1}{2}S_{x_0^4x_1}\\ 
		&\Sigma_{y_2y_1^3}&\rightarrow&\frac{1}{4}\Sigma_{y_3}\Sigma_{y_2}+\frac{5}{4}\Sigma_{y_5}&S_{x_0x_1^4}&\rightarrow&S_{x_0^4x_1}\\ 
		\hline
		&\Sigma_{y_6}&\rightarrow&\frac{8}{35}\Sigma_{y_2}^3& S_{x_0^5x_1}&\rightarrow&\frac{8}{35}S_{x_0x_1}^3\\ 
		&\Sigma_{y_4y_2}&\rightarrow&\Sigma_{y_3}^2-\frac{4}{21}\Sigma_{y_2}^3& S_{x_0^4x_1^2}&\rightarrow&\frac{6}{35}S_{x_0x_1}^3-\frac{1}{2}S_{x_0^2x_1}^2\\ 
		&\Sigma_{y_5y_1}&\rightarrow&\frac{2}{7}\Sigma_{y_2}^3-\frac{1}{2}\Sigma_{y_3}^2& S_{x_0^3x_1x_0x_1}&\rightarrow&\frac{4}{105}S_{x_0x_1}^3 \\ 
		&\Sigma_{y_3y_1y_2}&\rightarrow&-\frac{17}{30}\Sigma_{y_2}^3+\frac{9}{4}\Sigma_{y_3}^2&S_{x_0^3x_1^3}&\rightarrow&\frac{23}{70}S_{x_0x_1}^3-S_{x_0^2x_1}^2\\ 
		&\Sigma_{y_3y_2y_1}&\rightarrow&3\Sigma_{y_3}^2-\frac{9}{10}\Sigma_{y_2}^3& S_{x_0^2x_1x_0x_1^2}&\rightarrow&\frac{2}{105}S_{x_0x_1}^3\\ 
		6&\Sigma_{y_4y_1^2}&\rightarrow&\frac{3}{10}\Sigma_{y_2}^3-\frac{3}{4}\Sigma_{y_3}^2&S_{x_0^2x_1^2x_0x_1}&\rightarrow&-\frac{89}{210}S_{x_0x_1}^3+\frac{3}{2}S_{x_0^2x_1}^2 \\ 
		&\Sigma_{y_2^2y_1^2}&\rightarrow&\frac{11}{63}\Sigma_{y_2}^3-\frac{1}{4}\Sigma_{y_3}^2& S_{x_0^2x_1^4}&\rightarrow&\frac{6}{35}S_{x_0x_1}^3-\frac{1}{2}S_{x_0^2x_1}^2\\ 
		&\Sigma_{y_3y_1^3}&\rightarrow&\frac{1}{21}\Sigma_{y_2}^3&S_{x_0x_1x_0x_1^3}&\rightarrow&\frac{8}{21}S_{x_0x_1}^3-S_{x_0^2x_1}^2 \\ 
		&\Sigma_{y_2y_1^4}&\rightarrow&\frac{17}{50}\Sigma_{y_2}^3+\frac{3}{16}\Sigma_{y_3}^2&S_{x_0x_1^5}&\rightarrow&\frac{8}{35}S_{x_0x_1}^3\\ 
	\hline
	\end{array}$
\end{table}

\end{document}